\documentclass[11pt, reqno]{amsart}
\usepackage{amsfonts,amssymb}
\usepackage{amsmath,amscd}
\usepackage{amsthm}
 \usepackage[all]{xy}
\usepackage{tikz}
\usepackage{graphicx}
\usetikzlibrary{arrows, decorations.markings}

\usetikzlibrary{decorations.markings, arrows.meta}
\usepackage{color}   
\usepackage{hyperref}
\usepackage{enumitem}

\usetikzlibrary{arrows}
\usetikzlibrary{shapes,snakes}

\newlength{\bibitemsep}
\newlength{\bibparskip}
\let\oldthebibliography\thebibliography
\renewcommand\thebibliography[1]{%
  \oldthebibliography{#1}%
  \setlength{\parskip}{\bibitemsep}%
  \setlength{\itemsep}{\bibparskip}%
}

\newcommand{\cc}{\ensuremath{\mathbb{C}}}
\newcommand{\on}[1]{\operatorname{#1}}
\newcommand{\delete}[1]{{}}
\newcommand{\Mod}{\on{\!-Mod}}

\usetikzlibrary{arrows.meta}
\usetikzlibrary{decorations.pathreplacing,angles,quotes}
\tikzset{bullet/.style={draw,ellipse, text width = 4cm, text centered}}
\tikzset{rec/.style={draw, text width = 3.5 cm, text centered}}
\tikzset{plain/.style={->,>=stealth}}
\tikzset{
  barbarrow/.style={ 
     >={Straight Barb[left,length=5pt, width=5pt]}
  },
  strike through/.style={
    postaction=decorate,
    decoration={
      markings,
      mark=at position 0.5 with {
        \draw[-] (-3pt,-3pt)  --  (3pt, 3pt);
      }
    }
  }
}

\newcommand{\myhookrightarrow}{\raisebox{-1pt}{ \, \begin{tikzpicture}
  \draw [right hook-To] (0, 0) -- (0.6, 0);     
\end{tikzpicture} \, }}

\title[Yoneda algebras of the triplet vertex operator algebra]{Yoneda algebras of the triplet vertex operator algebra}
\author[A. CARADOT, C. JIANG]{Antoine CARADOT$^1$, Cuipo JIANG$^1$}
\address{$^1$School of Mathematical Sciences, Shanghai Jiao Tong University, Shanghai, 200240, China}
\author[Z. LIN]{Zongzhu LIN$^2$}
\address{$^2$Department of Mathematics, Kansas State University, Manhattan, KS, 66506, USA}
\date{\today}     
\dedicatory{Dedicated to Leonard L. Scott Jr. on the occasion of his 80th birthday}               
\keywords{Vertex operator algebras, Yoneda algebras, quivers, simple modules, projective modules}
\thanks{2020 {\it Mathematics Subject Classification:}
Primary 17B69, 13D03; Secondary 16G20\\
 C. Jiang is supported by NSFC grant No. 12171312}

\theoremstyle{plain}
\newtheorem{theorem}{Theorem}[section]
\newtheorem{lemma}{Lemma}[section]
\newtheorem{proposition}{Proposition}[section]
\newtheorem{definition}{Definition}[section]
\newtheorem{corollary}{Corollary}[section]

\theoremstyle{remark}
\newtheorem{remark}{Remark}[section]
\theoremstyle{remark}

\theoremstyle{remark}

\theoremstyle{remark}

\newtheoremstyle{dotless}{}{}{\itshape}{}{\bfseries}{}{ }{}
\theoremstyle{dotless}
\newtheorem*{theorem*}{Theorem}
\newtheorem*{lemma*}{Lemma}
\newtheorem*{proposition*}{Proposition}
\newtheorem*{definition*}{Definition}

\begin{document}
\maketitle

\begin{abstract} 
Given a vertex operator algebra $V$, one can construct two associative algebras, the Zhu algebra $ A(V)$ and the $C_2$-algebra $ R(V)$. This gives rise to two abelian categories $ A(V)\Mod$ and $R(V)\Mod$, in addition to the category of admissible  modules of $V$. In case $V$ is rational and $C_2$-cofinite, the category of admissible $V$-modules and the category of all $A(V)$-modules are equivalent. However, when $V$ is not rational, the connection between these two categories is unclear. The goal of this paper is to study the triplet vertex operator algebra $\mathcal{W}(p)$, as an example to compare  these  three  categories, in terms of abelian categories.  For each of these three abelian categories, we will determine the associated Ext quiver,  the Morita equivalent basic algebra, i.e., the algebra $ \on{End} (\oplus_{L\in \on{Irr}} P_L)^{op}$,  
and the Yoneda algebra $\on{Ext}^{*}(\oplus_{L\in \on{Irr}}L, \oplus_{L\in \on{Irr}}L)$. As a consequence, the category of admissible log-modules for the triplet VOA $ \mathcal W(p)$ has infinite global dimension, as do the Zhu algebra $A(\mathcal W(p))$, and the associated graded algebra $\on{gr}A(\mathcal W(p))$ which is isomorphic to $R(\mathcal W(p))$. We also describe the Koszul properties of the module categories of $ \mathcal W(p)$, $A(\mathcal W(p))$ and $\on{gr} A(\mathcal W(p))$.
 \end{abstract}

\tableofcontents

\section{Introduction}
Given a vertex operator algebra $(V, Y, \mathbf 1, \omega)$, one can attach two associative algebras. One  is the Zhu algebra $A(V)$ (cf. \cite{Zhu}) and the other is the Poisson algebra $R(V)$ (cf. \cite[\dots]{Zhu, Li, Ara}). In this paper, we compare the representation theories of the vertex operator algebra $V$ with those of the Zhu algebra $A(V)$ and the $C_2$-algebra $R(V)$.  The motivation for such comparison arises from a geometric point of view. The goal is to see what analogous roles the two associative algebras $R(V)$ and $A(V)$ play in terms of representation theory, and how much information about the representation theory of the vertex operator algebra is preserved by the representation theories of these two associative algebras. 

\delete{
One prototypical example is the universal affine vertex operator algebra $ V=V_k(\mathfrak g)$ of a simple Lie algebra $\mathfrak g$ at a non-critical level $k$. In this case, the associative algebras verify $A(V)=U(\mathfrak g)$ and $R(V)=\on{Sym}^*(\mathfrak g)$. The Zhu algebra $ A(V)$ always carries a filtered algebra structure with an increasing filtration:
 \[ \cdots \subseteq F^pA(V)\subseteq  F^{p+1}A(V)\subseteq \cdots \] 
 of vector subspaces. With respect to this filtered algebra structure, $A(V)$ is almost commutative in the sense that:
  \[ [F^pA(V), F^qA(V)]\subseteq F^{p+q-1}A(V).\] 
Thus the associated graded algebra $ \on{gr}A(V)$ is commutative and automatically carries a Poisson algebra structure with the Poisson bracket $\{ . , . \}$, which has degree $-1$.  In general there is surjective homomorphism of Poisson algebras $ R(V)\longrightarrow  \on{gr}A(V)$. However this homomorphism is not an isomorphism.  

From the deformation theory viewpoint, $A(V)$ is a deformation of $\on{gr}A(V)$. In the case of $ V=V_k(\mathfrak g)$, we have the isomorphism $R(V)=\on{gr}A(V)=\on{Sym}^*(\mathfrak g)$, the algebra of regular functions on the Poisson variety $\mathfrak g^*$. This suggests that $A(V)$ plays the role of the algebra of differential operators while $R(V)$ plays the role of function algebra of the cotangent bundle. 

Comparisons of the representation theory of $V$ and that of $A(V)$ was the main motivation in Zhu's thesis to define and study the Zhu algebra $A(V)$. For example, there is a one-to-one correspondence between the irreducible admissible representations of $V$ and the irreducible representations of $A(V)$. More functorial studies of these irreducible representations as well as specific constructions of irreducible modules involving higher Zhu algebras were carried out in \cite{Dong-Jiang} and \cite{DLM2}. The study of the fusion product were also carried out in \cite{Dong-Ren} and \cite{Huang-Yang}. Recently, Huang has packed all those higher Zhu algebras as well as those bimodules in \cite{Dong-Jiang} to form a large associative algebra in order to study the fusion product of $V$. The goal of this paper is, purely from the representation theory viewpoint, to compare certain representation theoretic invariants of these algebras.  

As we mentioned above, in the case of $V_k(\mathfrak g)$, the finite dimensional irreducible representations of the algebra $R(V)$ are in one-to-one correspondence with the points in $ \mathfrak g^*$. However, finite dimensional irreducible $A(V)$-modules are only corresponding to those linear functions on $\mathfrak h$ which restricts to positive coroots with non-negative integers. Furthermore, when $V$ is of CFT type and $C_2$-cofinite, $R(V)$ is a finite dimensional local algebra and thus has only one single irreducible representation while $A(V)$ can have many irreducible representations.
}
There are many categorical invariants of an abelian category $\mathcal{A}$ with enough projective objects such as  $A\Mod$ for $A$ an associative algebra. In addition to studying the irreducible objects in $\mathcal{A}$, one also wants to study the structures of projective indecomposables. There is a bijection between the irreducible $A$-modules and the indecomposable projective $A$-modules given by sending $L$ to its projective cover $P_L$. The Cartan invariants are the decomposition multiplicities  $[P_L : L']$, which form a matrix indexed by the irreducible modules $L$. However, there are much more indecomposable modules than projective indecomposable modules. Classifying all indecomposable modules is one the hardest problems in representation theory. An algebra $A$ is in general classified as finite type, tame type or wild type. 

One of the categorical invariants of $\mathcal{A}$ is the opposite of the endomorphism algebra $\mathcal{E}(\mathcal{A})=\on{End}_{\mathcal{A}}(P_\mathcal{A})$ with $P_\mathcal{A}=\bigoplus_{L \in \on{Irr}(\mathcal{A})}P_L$. The category of $\mathcal{E}(\mathcal{A})^{op}$-modules is then equivalent to $\mathcal{A}$. If $\mathcal{A}=A\Mod$, then $\mathcal{E}(\mathcal{A})^{op}$ is Morita equivalent to $A$. The ordinary quiver of $A$ is determined by the decompositions of $\on{Rad}(P_L)/\on{Rad}^2(P_L)$ into direct sums of indecomposable modules (if possible) (cf. \cite{Assem-Simson-Skowronski}). This quiver is one of the combinatorial invariants of the algebra $A$. However, it is not invariant under Morita equivalence (cf. \cite{Li-Lin} for further discussion). There is another categorical invariant of $\mathcal{A}$, the Ext quiver $Q_\mathcal{A}$. The vertices of $Q_\mathcal{A}$ correspond to irreducible objects and the multiplicities are determined by the dimensions of $\on{Ext}^1_\mathcal{A}(L, L')$ over the division rings $\on{End}_\mathcal{A}(L)$ and $\on{End}_\mathcal{A}(L')$ (this data determines a valued quiver or species in general \cite{Dlab-Ringel}). 

Another such invariant is the homological dual to the endomorphism algebra $\on{End}_{\mathcal{A}}(P_\mathcal{A})^{op}$, i.e., the graded algebra $\on{Ext}_\mathcal{A}^{*}(L_\mathcal{A}, L_\mathcal{A})$ where $L_\mathcal{A}= \bigoplus_{L \in \on{Irr}(\mathcal{A})} L$, with the product being the Yoneda composition of exact sequences.
This algebra can have infinitely many non-zero components and might not be commutative (even when the algebra $ A$ is commutative). The commutativity and finite generation of this algebra define an affine algebraic variety of finite type associated to the algebra, which will be investigated in a forthcoming paper. The Yoneda algebra $\on{Ext}_\mathcal{A}^{*}(L_\mathcal{A}, L_\mathcal{A})$ is an important subject of study in representation theory. The Koszul duality property of quadratic algebras is used to study the duality between the Yoneda algebra $\on{Ext}_\mathcal{A}^{*}(L_\mathcal{A}, L_\mathcal{A})$ and the endomorphism algebra $\on{End}_{\mathcal{A}}(P_\mathcal{A})^{op}$ (see \cite{CPS, Lu-Palmieri-Wu-Zhang, Madsen}).

This paper is an attempt to understand the relations between these representation theoretic invariants of these four representation categories for $V$, $A(V)$, $R(V)$, and $ \on{gr}(A(V))$ respectively. The goal is to understand what representation properties of $V$ are reflected in the representations of the attached associative algebras. In a forthcoming paper, we establish and study functors between these categories for dg vertex operator algebras. It is worth mentioning that, in \cite{CCG}, the Ext-algebra $\on{Ext}_V^*(V, V)$ has been computed for $V$ as a module in both the category of ordinary modules and the category of log-modules for various vertex operator algebras including $\mathcal W(p)$. 
\delete{
Since $R(V)$ is a graded Poisson algebra, it naturally determines an  (conical) affine Poisson variety $X_V$, which is called the associated variety of $V$. One natural question is to study the geometric property of $X_V$. Arakawa and his collaborators have described these varieties for a large classes of VOAs constructed from affine VOAs of simple Lie algebras $\mathfrak g$. These varieties are important in geometric representation theory. 
}

For a rational and $C_2$-cofinite VOA $V$, the Zhu algebra $A(V)$ is a finite dimensional semisimple algebra and it completely controls the representation theory of $V$ (in terms of admissible modules) (see \cite{Zhu, DLM1}).  In this case, the Ext quiver for both $V$ and $ A(V)$ becomes trivial (with vertices only and no arrows), and all the above representation theory invariants become trivially clear. If $V$ is of CFT type and $C_2$-cofinite,  $R(V)$ is a local finite dimensional algebra. The associated variety $X_V=\on{Spec}(R(V))$ (which is conical) is just one single point, the vertex of the cone (see \cite{Ara}). However, the representations of $R(V)$ can get very complicated and the Ext quiver for $R(V)$ can have many arrows. Although the endomorphism algebra $\on{End}_{R(V)}(P_{R(V)})^{op}=R(V)$ is finite dimensional, its homological dual, the Yoneda algebra $\on{Ext}_{R(V)}^{*}(L_{R(V)}, L_{R(V)})$, can be infinite dimensional. This algebra reflects the singularity of the vertex point in the associated variety $ X_V=\on{Spec}(R(V))$. The Yoneda algebra $\on{Ext}_{R(V)}^{*}(L_{R(V)}, L_{R(V)})$  is (graded) commutative if $X_V$ is a complete intersection with some restrictions on the degrees of the relations (see \cite{Caradot-Jiang-Lin3}). For a commutative local algebra $R$, the Yoneda algebra has been studied extensively (see \cite{Lu-Palmieri-Wu-Zhang, Madsen}).  It is still very subtle at this point what properties of the vertex operator algebra $V$ are reflected by this algebra.  Computation of these algebras for affine rational vertex operator algebras are done in \cite{Caradot-Jiang-Lin}.

\delete{We have examples at this point to see how the fusion product in the category of $V$ is related to this algebra.  }

In addition to many known examples of non rational $C_2$-cofinite VOAs, the triplet VOA $\mathcal{W}(p)$ (see \cite[\dots]{Adamovic-Milas1, Adamovic-Milas2}) is a natural model to explore. In this paper we will take this VOA as the main example to do the comparison. 

The paper is organised as follows. In Section~\ref{Section2} we clarify certain module categories for a VOA and establish certain functors between module categories of the VOA and those of the associative algebras. To do so, we need certain filtered and graded structures on a module. Thus, the category of admissible modules can be realized as the full subcategory of weak modules, which are the images of a full subcategory of graded modules.  We thus define the categories of filtered modules and graded modules, as well as the homomorphisms.  We will consider the Zhu algebra as a filtered algebra, and its modules are also filtered modules. Hence for a graded $V$-module $ M$, there is a filtered $A(V)$-module $A(M)$ which can be transformed into a graded $\on{gr}A(V)$-module $\on{gr}A(M)$. The surjective Poisson algebra homomorphism $R(V) \longrightarrow \on{gr}A(V)$ induces a functor $\on{gr}A(V)\otimes _{R(V)}-$ from the category of $R(V)$-modules to the category of $ \on{gr}A(V)$-modules. It follows that a graded $V$-module $M$ also defines a graded $R(V)$-module $R(M)$, and there is a surjective homomorphism $\on{gr}A(V)\otimes _{R(V))}R(M)\longrightarrow \on{gr}A(M)$. We could not prove that this is an isomorphism.  Logarithmic modules are needed in order to study the homological properties. The category of ordinary modules is not closed under extension and the category of logarithmic modules is a Serre subcategory of the category of weak modules. Thus the Yoneda algebra is computed in this category.

In Section~\ref{Section3}, we briefly review the Ext quiver associated to an abelian category with enough projective objects and finitely many irreducibles. We also define such an abelian category to be Koszul if the Yoneda algebra is generated by the degree $1$ part over the degree $0$ part. Section~\ref{QuiverWp} is devoted to the computation of the quiver for the triplet VOA $\mathcal{W}(p)$. The category of logarithmic $\mathcal{W}(p)$-modules as well as the projective covers are described in \cite{Nagatomo-Tsuchiya} and also in \cite[Section 7]{McRae-Yang}. This category is equivalent to the module category of the restricted quantum group $\overline{U}_q(\on{sl}_2)$ at the parameter  $q=e^{\pi i/p}$ as abelian categories using the description of the projective covers of irreducible modules as argued in \cite{Nagatomo-Tsuchiya}. Thus the Yoneda algebras as well as the Ext quivers for these two categories are isomorphic. It should be mentioned that the Yoneda algebra and the Ext quiver for  $\overline{U}_q(\on{sl}_2)$  have been determined in \cite{GSTF}. Thus the Yoneda algebra and the Ext quiver for $\mathcal{W}(p)$ can be transferred from those  in \cite{GSTF}. Since our goal is to compare different categories associated to $\mathcal{W}(p)$ and to establish functors among these, we include the computation in terms of modules for the vertex operator algebra $\mathcal W (p)$.  
 It is worth mentioning that it has been recently proved in \cite{Gannon-Negron} that the two categories are equivalent as ribbon tensor categories (up to a cocycle twist).  Although we do not use the ribbon structure in this paper, the ribbon structure on the categories also provide additional structures on the Yoneda algebra in general. This structure will be addressed in future works. 
 
Section~\ref{Section5} is devoted to computing the Ext quivers of the Zhu algebra and the associated graded algebra. In this case, $R(V)\longrightarrow \on{gr}A(V)$ is an isomorphism. In Section~\ref{Section6} we compute the Yoneda algebras of the associative algebras  $A(\mathcal W(p))$ and $\on{gr} A(\mathcal W(p))$, as well as determine the Koszul properties of the modules categories for $\mathcal W(p)$, $A(\mathcal W(p))$ and $\on{gr} A(\mathcal W(p))$.

\section{Vertex operator algebras and associated algebras}\label{Section2}
\subsection{Vertex operator algebras and their module categories}
The definitions and basic properties of vertex (operator) algebras can be found in  \cite{B, DLM1, DLM2, FHL, FLM, Lepowsky-Li, Zhu}.

\delete{\begin{definition} 
A \textbf{vertex  algebra} is a vector space $V$ over $ \cc$ equipped with a linear map (vertex operator map) 
\begin{align*}
\begin{array}{cccc}
Y(\cdot, x): & V & \longrightarrow &\on{End}(V)[[x, x^{-1}]] \\
           & v & \longmapsto      & Y(v, x)=\displaystyle \sum_{n \in \mathbb{Z}}v_n x^{-n-1}
\end{array}
\end{align*}
and a particular vector $\mathbf{1} \in V$, the vacuum vector, satisfying the following conditions:   
\begin{itemize}\setlength{\itemsep}{4pt}
\item For any $u, v \in V$, $u_nv=0$ for $n$ sufficiently large, i.e., $Y(u, x)v \in V((x))$ (Truncation property),
\item $Y(\mathbf{1}, x)=1$ ($1$ is the identity operator on $V$) (Vacuum property),
\item $Y(v, x)\mathbf{1} \in V[[x]]$ and $\displaystyle \lim_{x \to 0} Y(v, x)\mathbf{1}=v$ (Creation property),
\item The Jacobi identity 
\begin{align*}
\resizebox{\hsize}{!}{$ 
 x_0^{-1}\delta(\frac{x_1-x_2}{x_0})Y(u, x_1)Y(v, x_2)-x_0^{-1}\delta(\frac{x_2-x_1}{-x_0})Y(v, x_2)Y(u, x_1)=x_2^{-1}\delta(\frac{x_1-x_0}{x_2})Y(Y(u, x_0)v, x_2), $}
\end{align*}
where $\delta(x)=\displaystyle \sum_{n \in \mathbb{Z}}x^n$ is a formal series.
\end{itemize}
We will denote a vertex algebra by $ (V, Y, \mathbf{1})$. The vertex algebra is equipped with an endomorphism defined by $\mathcal D(v)=v_{-2}\mathbf{1}$ satisfying 
\[ [ \mathcal D, Y(v, x)]=Y(\mathcal D (v), x)=\frac{d}{dx}Y(v, x).\]

A {\bf vertex operator algebra } is a vertex algebra $ (V, Y, \mathbf{1})$ together with 
an element $ \omega \in V$ called the conformal vector (or Virasoro element) of $V$ such that $ Y(\omega, x)=\sum_{n\in \mathbb Z}L(n)x^{-n-2}$ and: 
\begin{itemize}\setlength{\itemsep}{4pt}
\item $[L(m),L(n)]=(m-n)L(m+n)+\frac{1}{12}(m^3-m)\delta_{m+n,0}c_V$ for $m, n \in \mathbb{Z}$ (the Virasoro relations) where $c_V \in \mathbb{C}$ is the central charge of $V$,
\item The linear operator $ L(0)$  on $ V$ is semisimple and $ V=\bigoplus_{n\in \mathbb Z} V_n$ with $L(0)v=nv=\on{wt}(v)v$ for $ n\in \mathbb{N}$, $v \in V_{n}$. Furthermore $\omega \in V_2$, $\dim V_n<\infty $, and $V_n=0$ for $ n \ll 0$.  
\item $L(-1)=\mathcal D$.
\end{itemize}
\end{definition}}

We will denote a vertex algebra by $ (V, Y(\cdot, x), \mathbf{1})$ where $V$ is the underlying vector space, $Y(\cdot, x)$ is the vertex operator map given by $Y(v, x)=\sum_{n \in \mathbb{Z}}v_n x^{-n-1}$ for $v \in V$, and $\mathbf{1}$ is the vacuum vector. If $V$ has a vertex operator algebra structure, the conformal element will be written as $\omega$.

\delete{
We remark that a vertex operator algebra $V=\bigoplus_{n\in\mathbb Z} V_n$ is automatically  a graded vector space.  A vertex operator algebra homomorphism sends the conformal element  to the conformal element and thus preserves the graded structure.  This structure has the property that, for any $ u \in V_n$,
\[ u_m(V_r)\subseteq V_{r+n-m-1}.\]
}

\delete{
\begin{definition}
Let $(V,Y,\mathbf{1})$ be a vertex  algebra. A vertex algebra \textbf{module} is a vector space $M$ equipped with a linear map
\begin{align*}
\begin{array}{cccc}
Y_M(\cdot, x): & V & \longrightarrow &\on{End}(M)[[x, x^{-1}]] \\
           & v & \longmapsto      & Y_M(v, x)=\displaystyle \sum_{n \in \mathbb{Z}}v_n x^{-n-1}
\end{array}
\end{align*}
such that for any $u,v \in V$, the following properties are verified:
\begin{itemize}\setlength{\itemsep}{4pt}
\item For any $v \in V$, $w \in M$,  $v_n w=0$ for $n$ sufficiently large, i.e., $Y_M(v, x)w \in M((x))$ (Truncation property),
\item $Y_M(\mathbf{1}, x)=\mathrm{Id}_{| M}$ (Vacuum property),
\item The Jacobi identity 
\begin{align*} 
\begin{array}{r}
 x_0^{-1}\delta(\frac{x_1-x_2}{x_0})Y_M(u, x_1)Y_M(v, x_2)-x_0^{-1}\delta(\frac{x_2-x_1}{-x_0})Y_M(v, x_2)Y_M(u, x_1)\\[5pt]
 =x_2^{-1}\delta(\frac{x_1-x_0}{x_2})Y_M(Y(u, x_0)v, x_2). 
 \end{array}
\end{align*}
\end{itemize}
\end{definition}}

For a vertex algebra $(V,Y(\cdot, x),\mathbf{1})$, a $V$-module $M$ will be written as $(M, Y_M(\cdot, x))$ (cf. \cite{Lepowsky-Li}). If $(V, Y(\cdot, x), \mathbf{1}, \omega)$ is a vertex operator algebra, then a vertex algebra module $M$ is called a {\bf weak} module for the vertex operator algebra. As consequences of the definition above we have:
\begin{itemize}\setlength{\itemsep}{4pt}
\item $Y_M(L(-1)v, x)=\frac{d}{dx}Y_M(v, x)$,
\item  $[L_m,L_n]=(m-n)L_{m+n}+\frac{1}{12}(m^3-m)\delta_{m+n,0}c_V$ with $c_V \in \cc$, 
\end{itemize}
where $Y_M(\omega, x)=\sum_{n \in \mathbb{Z}}L_n x^{-n-2}$. 

We remark that although the operator $L_0$ is acting on $M$, there is no requirement on $L_0$ to be diagonalizable or even locally finite. A homomorphism $f : M \longrightarrow M'$ between weak modules is a linear map such that $f(v_nm) = v_nf(m)$ for $v \in V$, $m \in M$. The category of all weak modules is an abelian category which is closed under coproduct, and will be denoted by $V\Mod^w$. 

Similarly, we can define a \textbf{graded} module  as a  weak $V$-module $M$ together with a $\mathbb C$-vector space decomposition $M=\bigoplus_{n \in \mathbb Z}M(n)$ such that, for all $n \in \mathbb{Z}$, $u \in V_n$, 
\[ u_m(M(r))\subseteq M(r+n-m-1) \] 
for all $ m, r \in \mathbb Z$. We will use $M(\bullet)$ to denote the $\mathbb Z$-graded vector space structure on $ M$. The same weak $V$-module $M$ can have many different $\mathbb Z$-graded vector space structures. 

 Let $V\Mod^{gr}$ be the category of all graded $V$-modules. We will use $(M, M(\bullet))$ to denote the objects of $V\Mod^{gr}$ in order to specify the graded structure. A homomorphism $ f: (M, M(\bullet))\longrightarrow  (N, N(\bullet))$ in $V\Mod^{gr}$ is a $V$-module  homomorphism  $ f: M\longrightarrow N$ such that $ f(M(n))\subseteq N(n)$ for all $ n\in \mathbb Z$. 
 It is straightforward to verify that  $V\Mod^{gr}$ is an abelian category. 
 We remark that $V\Mod^{gr}$ is not a full subcategory of $ V\Mod^w$.   However there is a natural additive forgetful functor $V\Mod^{gr} \longrightarrow V\Mod^w$, which sends $(M, M(\bullet))$ to $M$ forgetting the graded structure. The category $V\Mod^{gr}$ has a full subcategory $V\Mod^{gr}_+$ consisting of all modules $M(\bullet)$ such that $M(n)=0$ for $n \ll 0$. 

An \textbf{admissible} $V$-module is an object  $M$ in $V\Mod^w$ such that there is an object $(M,M(\bullet)) $ in $V\Mod^{gr}$  satisfying $M(n)=0$ for all $ n<0$. Let $ V\Mod^{adm}$ be the full subcategory of $ V\Mod^{w}$ consisting of all admissible modules $M$.  Thus every object in $ V\Mod^{adm}$ is of the form $M$ for some $ (M, M(\bullet))$ in $ V\Mod^{gr}_+$.  We want to mention that for a given admissible module $M$, there could be many different graded structures $ M(\bullet)$ on $M$ to make $(M, M(\bullet))$ an object in $V\Mod^{gr}_+$. Therefore the graded structure on $M$ is not preserved under homomorphisms in $ V\Mod^{adm}$. 

An \textbf{ordinary} $V$-module is a $\mathbb{C}$-graded weak $V$-module $M=\bigoplus_{\lambda \in \mathbb{C}} M_\lambda$ such that:
\begin{itemize}[topsep=1pt]\setlength{\itemsep}{0pt}
\item[(i)] $\text{dim }M_\lambda < \infty$ for all $\lambda \in \mathbb{C}$,
\item[(ii)] $M_{\lambda+n}=0$ for a fixed $\lambda$ and $n$ small enough,
\item[(iii)] $M_\lambda$ is the eigenspace of $L_0$ of eigenvalue $\lambda$.
\end{itemize}
We will use $ V\Mod^{ord}$ to denote the full subcategory of $V\Mod^w$ of all ordinary modules. 

In general, given two modules $M $ and  $N$ in $V\Mod^{ord}$ and a short exact sequence
\[ 0\longrightarrow N\longrightarrow P\longrightarrow M\longrightarrow 0\]
in the category $V\Mod^{w}$,   the weak module $P$ needs not be in $V\Mod^{ord}$. The full subcategory $V\Mod^{ord}$ is not necessarily closed under extensions in $V\Mod^{w}$. Requiring the operator $L_0$ to act semisimply on $M$ seems to be too strong a condition.  A weak module $M$ is called \textbf{logarithmic} if it satisfies all conditions of ordinary modules except  that the condition (iii)  is relaxed by assuming that $ M_\lambda $ is a generalized eigenspace for $L_0$, i.e., $(L_0-\lambda)$ acts nilpotently on $M_\lambda$.
We will use $V\Mod^{log}$ to denote the full subcategory of $V\Mod^{w}$ consisting all such weak modules $M$.  Indeed, $ V\Mod^{log}$ is closed under taking extensions, subquotients in $V\Mod^w$. Thus $V\Mod^{log}$ is a Serre subcategory of $ V\Mod^w$ (see \cite[Section 12.10]{Stacks-Project}). The categories $V\Mod^{log}$ and $V\Mod^{ord}$ can have quite different homological properties. See \cite[Section 9]{CCG} for a comparison between the self-extensions of the module $V$ in the two different categories. 

By using the identity
\[ [L_0, v_n]=(\on{wt}(v)-n-1)v_n\] 
for homogeneous $v$ in $V$, we can see that for any logarithmic module $M$, the subspace
\[ \bigoplus_{\lambda}\on{Ker}((L_0-\lambda): M_\lambda\to M_\lambda)\]
of $M$ is a weak $V$-submodule of $M$. Hence we have that every irreducible logarithmic module is ordinary. 

\delete{The operators $ L_{-1}, \; L_0,\;  L_1$ acting on all ordinary modules form a Lie algebra isomorphic to $ \mathfrak s\mathfrak l_2$. The category $\mathcal O^-$ for $ \mathfrak{sl}_2$ corresponding to lowest weight modules is defined as the Lie algebra modules $M$ such that $L(0)$ acts semisimply with finite dimensional weight spaces and $L(1)$ (corresponding to the negative roots) acts locally nilpotently.  It was shown in \cite{Dong-Lin-Mason} that every module in $ \mathcal O^-$ can be decomposed uniquely as direct sum of (possibly infinitely many) indecomposable modules in $\mathcal O^-$. In particular every indecomposable $N\in V\Mod^{ord}$ must have the form $N=\bigoplus_{n\geq 0}N_{\lambda+n}$ for a unique $\lambda\in \mathbb C$ such that $N_\lambda\neq 0$.

}
\begin{proposition}
Every module $M$ in $V\Mod^{log}$ can be decomposed uniquely into a direct sum of (possibly infinitely many) indecomposable modules such that each indecomposable component appears only finitely many times.
\end{proposition}
The proof is similar to the proof for ordinary modules. The uniqueness follows from the fact that  each generalized weight space is finite dimensional. 

Any  $M$ in $V\Mod^{log}$ can be given a canonical graded structure $M(\bullet)$ so that $(M,M(\bullet))$ is in $V\Mod^{gr}_+$. Indeed, we have $M=\bigoplus_{n \in \mathbb{N}}M(n)$ where $M(n)=\bigoplus_{\lambda \in \Lambda}M_{\lambda+n}$ and $\Lambda=\{\lambda \in \cc \ | \ M_\lambda \neq 0\;\text{ and }\; M_{\lambda -m}=0 \text{ if } m \in \mathbb{N}^*\}$. It follows that a log-module can be seen as an admissible module with a well-constructed gradation.  We can summarize the relations among the different full subcategories  as follows:
\begin{align*}
V\Mod^{ord} \subseteq V\Mod^{log}\subseteq V\Mod^{adm} \subseteq V\Mod^w.
 \end{align*}
We remark that $ V\Mod^{gr}$ is not a full subcategory of $V\Mod^w$. 

\delete{Before closing this subsection, we recall the following definition \cite{Zhu, DLM1}:
\begin{definition}
A vertex operator algebra $V$ is \textbf{rational} if every admissible $V$-module is a direct sum of irreducible admissible $V$-modules. In particular, a rational vertex operator algebra is semisimple, i.e., it is a direct sum of irreducible modules for itself.
\end{definition}}

It follows from the definition that, for a rational vertex operator algebra $V$, any logarithmic module is isomorphic to a direct sum of irreducible modules. Thus all logarithmic modules are ordinary. 
\subsection{$C_2$-algebras and Zhu algebras}\label{SectionC2ZhuAlgebra} 
Let $V$ be a vertex  algebra. The \textbf{$C_2$-algebra} $R(V)$ was introduced by  Zhu in \cite{Zhu} and is defined as $R(V)=V/C_2(V)$ where $C_2(V)$ is the vector subspace of $V$ spanned by the set $\{a_{-2}b \ | \ a,b \in V\}$. This quotient has a commutative Poisson algebra structure given by 
\begin{align*}
 \overline{a} \cdot \overline{b}= \overline{a_{-1}b}  \quad  \text{ and } \quad \{\overline{a}, \overline{b}\} =\overline{a_0b} \quad \text{ for } a,b \in V,
\end{align*}
where $\overline{x}$ is the image of $x \in V$ in $R(V)$. If $R(V)$ is finite dimensional then $V$ is said to be \textbf{$C_2$-cofinite}. 

Although $R(V)$ can be defined for any vertex algebra and its Poisson structure is independent of the choice of the conformal vector $ \omega$, the conformal structure determines a commutative graded algebra structure on $ R(V)$ with $R(V)_n$ being the image of $V_n$. 

The next algebra we are considering was also defined by Zhu (cf. \cite{Zhu}).

\begin{definition} \label{ZhuAlgebra}
We define the following bilinear operation on a vertex operator algebra $V$. For $a \in V$ homogeneous and $b \in V$, set
\begin{align*}
a \ast b =\displaystyle \text{Res}_x \left( Y(a, x)\frac{(1+x)^{\on{wt}a}}{x}b \right).
\end{align*}
Let $O(V)$ be the subspace of $V$ spanned by the elements of the form
\begin{align*}
 a \circ b= \text{Res}_x \left( Y(a, x)\frac{(1+x)^{\on{wt}a}}{x^2}b \right).
\end{align*}
The \textbf{Zhu algebra} of $V$ is $A(V)=V/O(V)$.
\end{definition}

It was proven in \cite{Zhu} that $A(V)$ is an associative algebra with the associative multiplication $[a] [b]=[a\ast b]$, where we write $[v]$ for the image of $v \in V$ in $A(V)$. There is a unit given by the image of $\mathbf{1}$. Finally, $A(V)$ is a filtered algebra with an ascending filtration where $F^nA(V)$ is the image of $\bigoplus_{i  \leq n}V_i$ in the quotient.

\delete{
\begin{remark}
We will use the convention that, for any variables $x, y$ and any $n \in \mathbb{Z}$,  $(x+y)^n=\displaystyle \sum_{k=0}^\infty \binom{n}{k}x^{n-k}y^k$ is a power series of the second variable $y$.
\end{remark}

\begin{theorem}[\cite{Zhu}] 
$O(V)$ is a two sided ideal for the multiplication $\ast$, and so $\ast$ is defined on $A(V)$. Moreover:
\begin{enumerate}\setlength{\itemsep}{4pt}
\item $A(V)$ is an associative algebra under the multiplication induced by $\ast$.
\item The image of the vacuum vector $\mathbf{1}$ is the unit of the algebra $A(V)$. 
\item The image of the conformal vector $\omega$ is in the center of $A(V)$.
\item $A(V)$ is a filtered algebra with an ascending filtration where $F^nA(V)$ is the image of $\bigoplus_{i  \leq n}V_i$ in the quotient.
\end{enumerate}
\end{theorem}}

The algebra $A(V)$ is a filtered algebra, so the associated graded algebra is defined as $\on{gr}A(V)= \bigoplus_{i} F^iA(V)/F^{i-1}A(V)$. It is shown in \cite{Zhu} that 
\begin{equation*}
\begin{cases}
F^i A(V)  F^j A(V) \subseteq F^{i+j} A(V), \\[5pt]
[F^i A(V),F^j A(V)] \subseteq F^{i+j-1} A(V),
\end{cases}       
\end{equation*}
so $A(V)$ is almost commutative, and hence $\on{gr}A(V)$ is a commutative Poisson algebra with the product being induced by the one in $A(V)$ (see \cite[1.3]{Chriss-Ginzburg}). 

The algebras $\on{gr}A(V)$ and $R(V)$ are linked by the following result:
\begin{proposition}\emph{\textbf{(\cite{Arakawa-Lam-Yamada}).}}\label{etaVsurjective} 
The map below is a well-defined surjective Poisson algebra homomorphism:
\begin{align*}
\begin{array}{cccc}
\eta_V: & R(V) & \longrightarrow &\on{gr}A(V)\\
           & \overline{a} & \longmapsto    & a + O(V) + \displaystyle \bigoplus_{i<\mathrm{wt}(a)} V_i.
\end{array}
\end{align*}
\end{proposition}

\subsection{Functors among module categories}

Given a graded module $M$ for a vertex operator algebra $V$, we can define a graded $R(V)$-module by setting $R(M)=M/C_2(M)$ with $C_2(M)=\on{Span}\{v_{-2}w \ | \ v \in V, w \in M\}$. Likewise, we can define a filtered $A(V)$-module by $A(M)=M/O(M)$ where $O(M)$ is the linear span of elements of the form $\text{Res}_x \left(Y_M(v, x)\frac{(1+x)^{\on{wt}v}}{x^2}w \right)$ where $v \in V, w \in M$.

Let $W$ be a filtered module over the filtered algebra $A(V)$, i.e., there exists an ascending filtration $\cdots\subseteq F^jW\subseteq F^{j+1}W\subseteq \cdots $ of subspaces in $W$ such that  $F^iA(V)F^jW \subseteq F^{i+j}W$. We can construct the associated graded $\on{gr} A(V)$-module $\on{gr}(W)= \bigoplus_{i} F^iW/F^{i-1}W$ where, if $a \in F^mA(V)$ and $w \in F^nW$, the action is defined as
\begin{align*}
(a + F^{m-1}A(V))(w + F^{n-1}W) = aw + F^{m+n-1}W.
\end{align*}
Thus $\on{gr}(W)$ is a graded module over the graded algebra $\on{gr} A(V)$. 

Finally, consider a graded $R(V)$-module $W$. Then, as $\eta_V$ is surjective and preserves the gradation, we see that $\on{gr} A(V) \otimes_{R(V)}W=W/(\on{Ker}\eta_V)W$ has a graded $\on{gr} A(V)$-module structure.

We have the following picture:
\begin{equation}\label{alg-diagram}
\begin{tikzpicture}[baseline=(current  bounding  box.center), scale=1, transform shape]
\tikzset{>=stealth}
\node (1) at (0,0) []{$V\Mod^{gr}$};
\node (2) at (2.5,1.5) []{$A(V)\Mod^{fil}$};
\node (3) at (2.5,-1.5) []{$R(V)\Mod^{gr}$};
\node (4) at (5,0) []{$\on{gr} A(V)\Mod^{gr}$};

\draw[->]  (1) -- node[above left] {$A(-)$} (2);
\draw[->]  (1) -- node[below left] {$R(-)$} (3);
\draw[->]  (2) -- node[above right] {$\on{gr}(-)$} (4);
\draw[->]  (3) -- node[below right] {$\on{gr} A(V) \otimes_{R(V)}-$} (4) ;

\end{tikzpicture}
\end{equation}
This diagram is in general not strictly commutative.

\delete{Let $M$ be a weak $V$-module and define
\[
\Omega(M)= \{w \in M \ | \ u_n w = 0, \ \text{ for all }  u \in V \ \text{ and }   n > \on{wt}(u)-1\}.
\]
\begin{theorem}\emph{\textbf{(\cite{DLM1, DLM2} and \cite{Zhu}).}} \label{TheoremZhuAlgebra}
Let $V$ be a vertex operator algebra. 
\begin{enumerate}\setlength{\itemsep}{4pt}
\item Let $M$ be a weak $V$-module and define the linear map
\begin{align*}
\begin{array}{cccc}
o: & A(V) & \longrightarrow &\on{End}(\Omega(M))\\
           & [a] & \longmapsto      & a(\on{wt}(a)-1)_{| \Omega(M)}
\end{array}
\end{align*}
for $a \in V$ homogeneous, and extend it to all of $V$ by linearity. This map is an algebra homomorphism from $A(V)$ to $\on{End}(\Omega(M))$ making $\Omega(M)$ an $A(V)$-module. Furthermore, if $M= \bigoplus_{n=0}^\infty M(n)$ with $M(0)\neq 0$ is admissible and irreducible, then $\Omega(M)=M(0)$. 
\item $\Omega$ is a covariant functor from $V\Mod^w$ to $A(V)\Mod$.
\item The map $M \longrightarrow M(0)$ induces a bijection between the set of equivalence classes of irreducible admissible $V$-modules with $M(0) \neq 0$ and the set of equivalence classes of irreducible $A(V)$-modules.
\item If $V$ is rational then $A(V)$ is finite dimensional and semisimple.
\item If $V$ is rational then there are only finitely many irreducible admissible $V$-modules up to isomorphism, and each irreducible admissible $V$-module is ordinary.
\end{enumerate}
\end{theorem}}

Our goal is, given a vertex operator algebra $V$, to describe the category of $V$-modules using the other module categories in Diagram \eqref{alg-diagram}. In particular, the algebras involved might be finite dimensional, which would simplify computations. Our focus will be on abelian module categories with enough projectives. Let $\mathcal{A}=A\Mod$ be such a category of modules for an algebra $A$. Then one can define the Ext quiver $kQ_{A}$ of $A$. Furthermore, if $\{L_i\}_{i \in I}$ is the set of irreducible $A$-modules, we can define the Yoneda algebra $\on{Ext}_{A}^{*}(\bigoplus_{i \in I}L_i,\bigoplus_{i \in I}L_i)$. Finally, since the category $\mathcal{A}$ has enough projectives, every simple module $L_i$ has a projective cover $P_i$. This leads to the algebra $\on{End}_{A}(\bigoplus_{i \in I}P_i)$. All these objects contain information on the original category $\mathcal{A}$. Our goal is to construct these objects for the algebras in Diagram \eqref{alg-diagram} when $V$ is the triplet vertex operator algebra $\mathcal{W}(p)$ (see Section~\ref{QuiverWp}), and determine how much information is preserved through the passage from vertex operator algebra to algebra.

\section{Duality between the Ext quiver and $\on{Rad}(P_\mathcal{A})/\on{Rad}^2(P_\mathcal{A})$}\label{Section3}

\delete{
\subsection{Path algebra of a quiver}
The reference for this section is \cite{Assem-Simson-Skowronski}.

In this section we introduce quivers and how they relate to algebras.
\begin{definition}
A \textbf{quiver} is a quadruplet $Q=(Q_0, Q_1, s, t)$ with $Q_0$ a set of vertices, $Q_1$ a set of arrows between the vertices and $s, t: Q_1 \longrightarrow Q_0$ two maps that give to each arrow a source and a target respectively. We shall denote the vertices by letters $i ,j,\ldots $, and the arrow $\alpha$ such that $s(\alpha)=i$ and $t(\alpha)=j$ will be written $\alpha: i \longrightarrow j$.
\end{definition}

The \textbf{path algebra} $k Q$ of a quiver $Q$ is the algebra generated by the paths $e_i$ ($i \in Q_0$) of length $0$ and the paths $\alpha$ ($\alpha \in Q_1$) of length $1$ subjected to the relations:
\[
e_i^2 = e_i, \quad e_ie_j = 0 \ (i \neq j), \quad e_{t(\alpha)}\alpha = \alpha e_{s(\alpha)} = \alpha.
\]
By assigning the $e_i$'s with degree $0$ and the $\alpha$'s with degree $1$, the path algebra is a graded algebra $\bigoplus_{l=0}^{\infty} k Q_l$ with $kQ_l$ being the set of all directed paths of length $l$.

We note that the path algebra $k Q$ does not have an identity unless $Q_0$ is finite.  Let $R_Q$ be the ideal of $kQ$ generated by all $kQ_l$, $l \geq 1$. An ideal $\mathcal{I}$ of $kQ$ is called \textbf{admissible} if there exists an integer $m \geq 2$ such that $R_Q^m \subseteq \mathcal{I} \subseteq R_Q^2$.}

In this section $\mathbf{k}$ is a field. Recall that a $\mathbf k$-linear abelian category $ \mathcal A$ is an abelian category that is enriched in the category of $\mathbf k$-vector spaces. Thus $\on{Ext}^i_{\mathcal A}(a, b)$ are all $ \mathbf k$-vector spaces and compositions are all $ \mathbf k$-bilinear. In this section, $\mathcal{A}$ will be a $\mathbf k$-linear abelian category with finitely many irreducibles and enough projective objects. Let $\on{Irr}(\mathcal A)$ be the set of isomorphism classes of all simple objects in $ \mathcal A$. Let $L_\mathcal{A}=\bigoplus_{L \in \on{Irr}(\mathcal{A})}L$ and $P_\mathcal{A}=\bigoplus_{L \in \on{Irr}(\mathcal{A})}P_L$ where $P_L$ is the projective cover of $L$.

We will later on need the following two lemmas. The proof of the first one relies on the long exact sequence of cohomology, while the second is obtained by an induction on the length of the composition series.

\begin{lemma}\label{LemmaExt}
If $ L$ is an irreducible object and $P_L$ is its projective cover, then for any irreducible object $L'$ we have:
\[ \on{Ext}^i_{\mathcal{A}}(L, L')= \on{Ext}^{i-1}_{\mathcal{A}}(\on{Rad}(P_L), L').\]
\end{lemma}

\begin{lemma}\label{LemmaHom}
Let $L$ be an irreducible object in $\mathcal{A}$ with projective cover $P_L$, and let $M$ be an object of finite length. Then we have:
\[ \on{dim}_{\mathbf k} \on{Hom}_\mathcal{A}(P_L, M)= [M:L] \times  \on{dim}_{\mathbf k} \on{End}_\mathcal{A}(L),\]
where $[M:L]$ is the multiplicity of $L$ in $M$, i.e., the number of composition factors of the Jordan-H\"older composition series of $M$ isomorphic to $L$.
\end{lemma}


\begin{definition}\label{DefinitionExtQuiver}
Assume that $ \on{End}_{\mathcal A}(a)\cong \mathbf k$ for all $ a \in \on{Irr}(\mathcal A)$. The \textbf{Ext quiver} of $\mathcal{A}$, denoted by $Q_{\mathcal A}$, is defined as follows:
\begin{enumerate}[topsep=1pt]\setlength{\itemsep}{0pt}
\item Let $(Q_{\mathcal A})_0=\on{Irr}(\mathcal A)$ be the set of all vertices;
\item Given $a,b \in (Q_{\mathcal A})_0$, the set of arrows $a \longrightarrow b$ is indexed by a (fixed) $\mathbf k$-basis of   $\on{Ext}^1_{\mathcal A}(a,b)$.
\end{enumerate}
\end{definition}

When $A$ is a finite dimensional $\mathbf k$-algebra with $\mathbf k$ being algebraically closed, the category $\mathcal A=A\Mod$ is a $\mathbf k$-linear category satisfying the conditions of the definition. We will write $Q_A$ for $Q_{A\Mod}$ and $\on{Irr}(A)$ for $\on{Irr}(A\Mod)$. 

\begin{remark}\label{rem:3.4}
\delete{A finite dimensional $\mathbf k$-algebra $A$ is called basic if in its decomposition using primitive orthogonal idempotents, the indecomposable summands are pairwise non-isomorphic. One property of a basic algebra is that all its irreducible representations are one dimensional over $\mathbf k$. Every finite dimensional $\mathbf k$-algebra $ A$ is Morita equivalent to a finite dimensional basic algebra, which is the algebra $ \on{End}_{A}(P)$ where $ P=\oplus _{L\in\on{Irr}(A)}P_L$ and $P_L$ is the projective cover of an irreducible module $L$. Computing this algebra for the categories related to a vertex operator algebra is also a main subject of this paper. }

If $A$ is a basic, finite dimensional $\mathbf{k}$-algebra, then the Ext quiver is isomorphic to the ordinary quiver obtained from $\on{Rad}(A)/\on{Rad}^2(A)$ (cf. \cite[II.3, III.2]{Assem-Simson-Skowronski}). However the Ext quiver is Morita invariant, while the ordinary quiver is not. For example, let $A=M_n(\mathbf{k})$. The basic algebra $A^b$ associated to $A$ is $\mathbf{k}$, and the categories of finite dimensional $A$-modules and $A^b$-modules are equivalent. We can see that the ordinary quiver of $A$ is $n$ isolated points (because $A$ is semisimple), the ordinary quiver of $A^b$ is just a point, while $Q_A$ and $Q_{A^b}$ are both one isolated point. See \cite{Li-Lin} on further discussions of various different quivers associated to finite dimensional algebras (Artinian algebras). 
\end{remark}

The isomorphism mentioned in Remark~\ref{rem:3.4} needs some precisions. Because of Lemma~\ref{LemmaExt}, we have
\[
\on{Ext}^1_\mathcal{A}(L_\mathcal{A}, L_\mathcal{A})=\on{Hom}_\mathcal{A}(\on{Rad}(P_\mathcal{A}), L_\mathcal{A})=\on{Hom}_\mathcal{A}(\on{Rad}(P_\mathcal{A})/\on{Rad}^2(P_\mathcal{A}), L_\mathcal{A}).
\]
It follows that the functor $\on{Hom}_\mathcal{A}(-, L_\mathcal{A})$ establishes a duality between the Ext quiver and the ordinary quiver.

\delete{
The next result is due to P. Gabriel (\cite{Gabriel}).

\begin{theorem}\label{theoremGabriel}
Let $A$ be a finite dimensional, basic algebra over an algebraically closed field $\mathbf{k}$. There exists an admissible ideal $\mathcal{I}$ of $kQ_A$ such that $A \cong kQ_A/\mathcal{I}$.
\end{theorem}}

Assume the abelian category $\mathcal A$ has enough projectives and  any irreducible object $L$ in $\mathcal A$  has a projective cover $P_L$. One can define the Yoneda algebra $\on{Ext}_{\mathcal A}^*(L_\mathcal{A}, L_\mathcal{A})$ (cf. \cite{MacLane, Weibel}) and the endomorphism algebra $\on{End}_{\mathcal A}(P_\mathcal{A})$. These two algebras are related to the path algebra in the following way:

\begin{lemma}\label{kQandExt}
Let $\mathcal{A}$ be an abelian category with enough projectives and finitely many irreducibles. Then there is a homomorphism of graded algebras:
\[
\varphi :kQ_{\mathcal{A}} \longrightarrow  \on{Ext}_\mathcal{A}^*(L_\mathcal{A}, L_\mathcal{A}).
\]
\end{lemma}

\delete{
\begin{proof}
It is known $\on{Ext}_A^0(\bigoplus_{L \in \on{Irr}(A)} L, \bigoplus_{L \in \on{Irr}(A)} L)=\on{Hom}_A(\bigoplus_{L \in \on{Irr}(A)} L, \bigoplus_{L \in \on{Irr}(A)} L)=\bigoplus_{i \in Q_0}\mathbf{k} \on{id}_{L_i}$ because of Schur's lemma. Here $\{L_i\}_{i \in Q_0}$ is the set of irreducible $A$-modules. Furthermore, the vector space $(kQ_A)_1$ is spanned by the basis elements of $\on{Ext}_A^1(L_i ,L_j)$ for all $i, j \in Q_0$. As $kQ_A$ is generated by its elements of degree $0$ and $1$, we have naturally defined homomorphism:
\[
\begin{array}[b]{cccc}
\varphi: & kQ_A & \longrightarrow & \on{Ext}_A^*(\bigoplus_{L \in \on{Irr}(A)} L, \bigoplus_{L \in \on{Irr}(A)} L), \\
        & e_i  & \longmapsto & \on{id}_{L_i}, \\
        & \alpha:i \to j & \longmapsto & \text{basis element of } \on{Ext}_A^1(L_i, L_j). 
\end{array}\qedhere 
\] 
\end{proof}}

We notice that in general $\varphi$ is neither injective nor surjective. We finish the section with the following well-known lemma:

\begin{lemma}\label{kQandEnd}
Let $\mathcal{A}$ be an abelian category with enough projectives and finitely many irreducibles. Then there exists an admissible ideal $I \subseteq kQ_\mathcal{A}$ and an isomorphism of algebras:
\[
kQ_\mathcal{A}/I \cong  \on{End}_\mathcal{A}(P_\mathcal{A})^{op}.
\]
\end{lemma}

\delete{
\begin{proof}
There is an algebra isomorphism $A^{op} \stackrel{\cong}{\longrightarrow} \on{End}_A(A)$ that sends $a \in A$ to $(r \longmapsto ra) \in \on{End}_A(A)$. Assuming $A$ is basic, then $A=\bigoplus_{i=1}^kAe_i$ with $(e_1, \dots, e_k)$ a complete set of primitive orthogonal idempotents, and $Ae_i \ncong Ae_j$ if $i \neq j$. The $A$-module $Ae_i$ is the projective cover of $\text{top}(Ae_i)$ which is simple. Hence $A$ is the direct sum of the projective covers of its simple modules. But we also know from Theorem~\ref{theoremGabriel} that $A \cong kQ_A/I$ where $Q_A$ is the Ext quiver of $A$ and $I$ is an admissible ideal. Therefore we have
\begin{align*}
(kQ_A/I)^{op} \cong \on{End}_A(\bigoplus_{L \in \on{Irr}(A)}P_L).
\end{align*}
If $A$ is not basic, then it is Morita equivalent to its basic associated algebra $A^b$ and the isomorphism still stands. Indeed $kQ_A=kQ_{A^b}$, and for any idempotent $e \in A$, there is an isomorphism of algebras $eAe \cong \on{End}_A(Ae)^{op}$. So if $A^b=(e_{i_1}+\dots +e_{i_n})A(e_{i_1}+\dots +e_{i_n})$ where the $e_{i_j}$ are orthogonal idempotents, then $A^b \cong \on{End}_A(\bigoplus_{j=1}^n Ae_{i_j})^{op} \cong \on{End}_A(\bigoplus_{L \in \on{Irr}(A)} P_L)^{op}$ (cf. \cite[I. Lemma 6.5]{Assem-Simson-Skowronski}).
\end{proof}}

\begin{definition}
An abelian $\mathbf{k}$-category $\mathcal{A}$ with finitely many irreducibles is \textbf{Koszul} if the Yoneda algebra $\on{Ext}_\mathcal{A}^*(L_\mathcal{A}, L_\mathcal{A})$ is generated in degree $0$ and $1$. 
\end{definition}

In \cite{Green-Villa}, a graded algebra $A$ is called Koszul  if $\on{Ext}_A^*(A/\on{Rad}(A), A/\on{Rad}(A))$ is generated in degree $0$ and $1$, which is equivalent to $A/\on{Rad}(A)$ having a linear projective resolution as a graded $A$-module. Therefore $\mathcal{A}$ is Koszul if and only if $\mathcal{E}(\mathcal{A})$ is Koszul in the sense of \cite{Green-Villa}.

An algebra $A=kQ/I$ is called quadratic if $I=\langle I_2 \rangle$ where $I_2$ is the subset of paths of length $2$ in $I$. There is a natural pairing $kQ_1\otimes_\mathbf{k} kQ^{op}_1 \longrightarrow \mathbf{k}$ given by
\[
\langle \alpha, \overline{\beta} \rangle=\delta_{\alpha, \beta}
\]
where $\overline{\beta}$ is the opposite of $\beta$ in $kQ^{op}$. As $kQ_2=kQ_1 \otimes_{kQ_0}kQ_1$, we have a pairing  $kQ_2\otimes_\mathbf{k} kQ^{op}_2 \longrightarrow \mathbf{k}$ given by
\[
\langle \alpha \otimes \beta, \overline{\gamma} \otimes \overline{\delta} \rangle=\langle \alpha, \overline{\delta} \rangle \langle \beta, \overline{\gamma} \rangle.
\]
We write $I_2^\perp=\{\alpha \in kQ^{op}_2 \ | \ \langle I_2, \alpha \rangle=0  \}$. The quadratic dual of a quadratic algebra $A$ is then $A^!=kQ^{op}/\langle I_2^\perp \rangle$.

\begin{proposition}[\cite{Villa}]\label{prop:Koszul}
If a graded algebra $A$ is Koszul then it is quadratic and
\[
\on{Ext}_A^*(L_A, L_A)=A^!.
\]
\end{proposition}

\section{The triplet vertex operator algebra $\mathcal{W}(p)$ and its quiver}\label{QuiverWp}
We recall the triplet vertex operator algebra $\mathcal{W}(p)$ and then study its simple modules. The main references for the following construction are \cite[\dots]{Adamovic, Adamovic-Milas1, Adamovic-Milas2}.

Let $p\geq 2$ be an integer. Let $L$ be the rank one lattice $\mathbb{Z}\alpha$ with $\langle \alpha,\alpha \rangle=2p$. We denote by $(V_L, Y(\cdot, x), \mathbf{1}, \omega)$ the corresponding lattice vertex operator algebra (cf. \cite{Lepowsky-Li}). As a vector space,
\begin{align*}
V_L =U(\hat{\mathfrak{h}}^-) \otimes \mathbb{C}[L]
\end{align*}
where $ \mathbb{C}[L]$ is the group algebra of $L$, $\hat{\mathfrak{h}}$ is the affinization of the one-dimensional algebra $\mathfrak{h}$ spanned by $\alpha$, $\hat{\mathfrak{h}}^-=\mathfrak{h}\otimes t^{-1}\mathbb{C}[t^{-1}]$, the vacuum vector is
\begin{align*}
\mathbf{1}=1 \otimes 1,
\end{align*}
and the conformal vector is
\begin{align*}
\omega=\frac{\alpha(-1)^2}{4p}\mathbf{1}+\frac{p-1}{2p}\alpha(-2)\mathbf{1}.
\end{align*}
 
\begin{remark}
The usual conformal vector is $\omega=\frac{1}{2.2p}\alpha(-1)^2\mathbf{1}$ and we would get a vertex operator algebra of central charge $\on{dim} \mathfrak{h}=1$. By choosing $\omega$ as above, we obtain a vertex operator algebra of central charge $c_{p,1}=1-\frac{6(p-1)^2}{p}$.
\end{remark}

For $i \in \mathbb{Z}$, define $\gamma_i=\frac{i}{2p}\alpha \in L^\circ$, where $L^{\circ}=\frac{\mathbb{Z}}{2p}\alpha$ is the dual lattice of $L$. Then it is known (cf. \cite[\dots]{D, DL, Lepowsky-Li}) that the set $\{V_{L+\gamma_i} \ | \ i=0, \dots ,2p-1\}$ is the set of all irreducible $V_L$-modules up to equivalence. Furthermore, the contragredient $V_L$-modules of those simple modules are given by $V_{L+\gamma_i}' \cong V_{L+\gamma_{2p-2-i}}$ for $i=0, \dots, 2p-2$, and $V_{L+\gamma_{2p-1}}$ is self-dual (cf. \cite{Adamovic-Milas2}).

Because $V_L$ is a vertex operator algebra, one can define:
\begin{equation*}
\begin{cases}
      \displaystyle Y(e^\alpha, x)=\sum_{n \in \mathbb{Z}}e^\alpha_n x^{-n-1} \in \mathrm{End}(V_L)[[x, x^{-1}]],\\
      \displaystyle Y(e^{-\frac{\alpha}{p}}, x)=\sum_{n \in \frac{1}{p}\mathbb{Z}}e^{-\frac{\alpha}{p}}_n x^{-n-1} \in \mathrm{End}(V_{L^{\circ}})[[x^{\frac{1}{p}}, x^{-\frac{1}{p}}]].
    \end{cases}       
\end{equation*}
We then define the operators
\begin{equation*}
\begin{cases}
      \displaystyle Q=e^\alpha_0 : V_L \longrightarrow V_L, \\
      \displaystyle \widetilde{Q}=e^{-\frac{\alpha}{p}}_0 : V_{L^{\circ}} \longrightarrow V_{L^{\circ}}.
    \end{cases}       
\end{equation*}

\begin{remark}
$Q$ is a derivation, i.e., $Q(a_n b)=(Qa)_n b+a_n(Qb)$ for any $a, b \in V_L$, $n \in \mathbb{Z}$.
\end{remark}

The triplet vertex operator algebra is defined as, noting that $ V_L\subseteq V_{L^{\circ}}$, 
\begin{equation*}
\mathcal{W}(p)=V_L\cap\on{Ker}\widetilde{Q}.    
\end{equation*}
By setting $F=e^{-\alpha}$, $H=QF$ and $E=Q^2F$, one can show that $\mathcal{W}(p)$ is a simple vertex operator algebra strongly generated by $E$, $F$, $H$ and $\omega$, i.e., it is spanned by:
\begin{equation*}
(v_1)_{n_1} (v_2)_{n_2}  \dots (v_k)_{n_k} \mathbf{1}, \text{ with } v_i \in \{E,F,H,\omega\} \text{  and  } n_i<0 \text{  for all } 1 \leq i \leq k.
\end{equation*}
It was showed in \cite{Adamovic-Milas2} that $\mathcal{W}(p)$ is $C_2$-cofinite and  an upper bound for the dimension of $R(\mathcal{W}(p))$ was given. Later in  \cite{Adamovic-Milas4}, the dimension of $R(\mathcal{W}(p))$ was proved to be $6p-1$ (cf. Section~\ref{SectionZhuAlgebraA(W(p))}).

It was proved in \cite{Adamovic-Milas2} that the simple $\mathcal{W}(p)$-modules are the following:
\begin{equation*}
\begin{cases}
X_s^+= \overline{V_{L+\gamma_{s-1}}} \text{ for } s=1, \dots ,p-1, \\[5pt]
 X_s^-=\overline{V_{L+\gamma_{2p-s-1}}} \text{ for } s=1, \dots ,p-1,\\[5pt]
 X_p^+=V_{L+\gamma_{p-1}},\\[5pt]
 X_p^-= V_{L+\gamma_{2p-1}}.
\end{cases}       
\end{equation*} 
where $\overline{V_{L+\gamma_k}}$ is the submodule of $V_{L+\gamma_k}$ generated by the singular vectors. In particular, $\mathcal{W}(p)=X_1^+= \overline{V_{L+\gamma_{0}}}$. The projective covers $P_s^\pm$ of $X_s^\pm$ for all $s$ are determined in \cite{Nagatomo-Tsuchiya}.  There were concerns on some of the arguments in the construction of the projective modules $P_s^\pm$ in \cite[Prop. 4.1]{Nagatomo-Tsuchiya}. However an alternative approach in \cite[Thm. 7.9]{McRae-Yang} confirms that the projective modules $P_s^\pm$ do have the required structures described below. These structures are the basis of the computations of the Yoneda algebras. The projective modules for the restricted quantum group $\bar{U}_{e^{\pi i/p}}(sl_2)$ were given in \cite{GSTF}.

For $1 \leq s \leq p-1$, the projective covers  $P_s^\pm$ \delete{ is given by the direct sum $V_{L+\gamma_{s-1}} \bigoplus V_{L+\gamma_{2p-s-1}}$ where the action of $\mathcal{W}(p)$ is modified by a well-chosen operator. In particular, for $P_s^+$ (resp. $P_s^-$), the action of $L(0)$ is modified on $V_{L+\gamma_{2p-s-1}}$ (resp. $V_{L+\gamma_{s-1}}$) and untouched elsewhere. Finally it was also shown in \cite{Nagatomo-Tsuchiya} that for $1 \leq s \leq p-1$,}
are rigid (the socle filtrations and radical filtration coincide) and  the socle series of $P_s^+$ and $P_s^-$ are given by the following diagrams:
\[
 P_s^+: \quad \xymatrix{ & X^{+}_s \ar[rd]\ar[ld]& \\
 X^-_s\ar[rd]&& X^{-}_s\ar[ld]\\
 &X^{+}_s &} \hspace{1cm} P_s^-: \quad \xymatrix{ & X^{-}_s \ar[rd]\ar[ld]& \\
 X^+_s\ar[rd]&& X^{+}_s\ar[ld]\\
 &X^{-}_s &.}\]
 \delete{\[\begin{tikzpicture}[scale=1, transform shape]
\tikzset{>=stealth}
\node (1) at (0,1.5) []{$X_s^+$};
\node (2) at (-1,0) []{$X_s^-$};
\node (3) at (1,0) []{$X_s^-$};
\node (4) at (0,-1.5) []{$X_s^+$};

\draw[->]  (1) -- (2);
\draw[->]  (1) -- (3);
\draw[->]  (2) -- (4);
\draw[->]  (3) -- (4) ;
\end{tikzpicture} \hspace{2cm}  \begin{tikzpicture}[scale=1, transform shape]
\tikzset{>=stealth}
\node (1) at (0,1.5) []{$X_s^-$};
\node (2) at (-1,0) []{$X_s^+$};
\node (3) at (1,0) []{$X_s^+$};
\node (4) at (0,-1.5) []{$X_s^-$};

\draw[->]  (1) -- (2);
\draw[->]  (1) -- (3);
\draw[->]  (2) -- (4);
\draw[->]  (3) -- (4) ;
\end{tikzpicture}
\]}
Furthermore, we have $P_p^+=X_p^+$ and $P_p^-=X_p^-$.
\delete{ We can see in particular that $P_s^\pm$ is rigid (i.e., its socle series equals its radical series) for all $s$. }

Let us compute the Ext quiver of $\mathcal{W}(p)$. We will determine the extensions of the simple $\mathcal{W}(p)$-modules in the category $\mathcal{W}(p)\Mod^{log}$ of logarithmic $\mathcal{W}(p)$-modules. 
Since $\mathcal W(p)\Mod^{log}$ is a Serre subcategory of $ \mathcal W(p)\Mod^w$, then $ \on{Ext}_{\mathcal W(p)}^1(M, N)$ in $ \mathcal W(p)\Mod^{log}$ and $ \mathcal W(p)\Mod^w$ are the same. In particular, it is also the same as in $ \mathcal W(p)\Mod^{adm}$. So we simply use the notation $ \on{Ext}_{\mathcal W(p)}^1(M, N)$ without specifying which category we are working in (not in the category of ordinary modules).

The following result is proved in \cite{Nagatomo-Tsuchiya}. However we will give a more straightforward proof.

\begin{proposition}\label{PropExtW(p)} 
For the triplet vertex operator algebra $\mathcal{W}(p)$ we have
 \begin{align*}
\on{Ext}_{\mathcal{W}(p)}^1(X_{s_1}^{\epsilon_1}, X_{s_2}^{\epsilon_2})= 
\begin{cases}
\cc^2 \text{ if } 1 \leq s_1=s_2 \leq p-1, \epsilon_1 = -\epsilon_2,\\[5pt]
\{0\} \text{ otherwise}. 
\end{cases}
\end{align*}

\end{proposition}

\begin{proof} By Lemma~\ref{LemmaExt}, we have 
\[\on{Ext}_{\mathcal{W}(p)}^1(X_{s_1}^{\epsilon_1}, X_{s_2}^{\epsilon_2}) = \on{Hom}_{\mathcal W(p)}(\on{Rad}(P_{s_1}^{\epsilon_1}), X_{s_2}^{\epsilon_2}) = \on{Hom}_{\mathcal W(p)}(\on{Rad}(P_{s_1}^{\epsilon_1})/\on{Rad}^2(P_{s_1}^{\epsilon_1}), X_{s_2}^{\epsilon_2}),\] 
because the radical of a module is the intersection of the kernels of all homomorphisms from the module to simple modules. Using the module diagrams given before and the fact that $ \on{Rad}(P_{p}^{\pm})=0$, we see that
 \begin{align*}
\on{Ext}_{\mathcal{W}(p)}^1(X_{s_1}^{\epsilon_1}, X_{s_2}^{\epsilon_2})= 
\begin{cases}
\on{Hom}_{\mathcal W(p)}(X_{s_1}^{-\epsilon_1} \oplus X_{s_1}^{-\epsilon_1}, X_{s_2}^{\epsilon_2}) \text{ if } s_1 \neq p,\\[5pt]
\{0\} \text{ if } s_1=p,
\end{cases}
\end{align*}
and the proposition follows because the $X_s^{\pm}$ are irreducible.

\delete{Let $0 \longrightarrow X_s^+ \longrightarrow E \stackrel{f}{\longrightarrow} X_s^+ \longrightarrow 0$ be a non-trivial self-extension of $X_s^+$, $1 \leq s \leq p-1$. Let $\pi : P_s^+ \longrightarrow X_s^+$ be the projective cover map. Because $f$ is surjective and $P_s^+$ is projective, there exists $\varphi :P_s^+ \longrightarrow E$ such that $\pi = f \circ \varphi$. Furthermore, $\varphi$ is surjective because the extension is not trivial and so $P_s^+/\text{Ker }\varphi \cong E$. 

Let $X$ be the preimage of $X_s^+$ in $P_s^+$ by $\varphi$. Because $\text{Ker }f =X_s^+$, we have $X/\text{Ker }\varphi \cong X_s^+$, therefore $X_s^+ \cong E/\text{Ker }f \cong (P_s^+/\text{Ker }\varphi)/X_s^+ \cong (P_s^+/\text{Ker }\varphi)/(X/\text{Ker }\varphi) \cong P_s^+/X$. But we also know that $X_s^+ \cong P_s^+/S_2(P_s^+)$ with $S_2(P_s^+)$ being the second term in the socle sequence of $P_s^+$. 

Furthermore, the diagram of the socle sequence given above states that $\text{Ker }\pi= S_2(P_s^+)$ and so $\varphi(S_2(P_s^+)) \subseteq \text{Ker }f=X_s^+$. It follows that $S_2(P_s^+) \subseteq X$. Based on the previous paragraph, we get that $X=S_2(P_s^+)$, and the restriction of $\varphi$ to $S_2(P_s^+)$ gives a surjection $S_2(P_s^+) \longrightarrow X_s^+$. This implies that $X_s^+$ is a quotient of $S_2(P_s^+)$. But based on the diagram, the only simple quotient of $S_2(P_s^+)$ is $X_s^-$, which is not isomorphic to $X_s^+$. Hence the extension is trivial and $\on{Ext}_{\mathcal{W}(p)}^1(X_s^+, X_s^+)=\{0\}$. The proof is similar for $X_s^-$.

We know that $X_p^\pm=P_p^\pm$ is a projective module, and so $\on{Ext}_{\mathcal{W}(p)}^1(X_p^\pm, M)=\{0\}$ for any module $M$. Furthermore, by applying Lemma~\ref{LemmaExt}, we obtain $\on{Ext}_{\mathcal{W}(p)}^1(X_s^{\epsilon_1}, X_p^{\pm})=\on{Hom}_{\mathcal{W}(p)}(\on{Rad}(P_s^{\epsilon_1}), X_p^{\pm})$. When $s < p$, using the diagrams of socle sequences, we see that $\on{Ext}_{\mathcal{W}(p)}^1(X_s^{\epsilon_1}, X_p^{\pm})=\{0\}$. If $s=p$, $\on{Rad}(P_p^{\epsilon_1})=\on{Rad}(X_p^{\epsilon_1})=\{0\}$, so the result is clear.}
\end{proof}

\delete{The Ext quiver of the vertex operator algebra $\mathcal{W}(p)$ is constructed 
like in the case of classical algebras, i.e., the vertices of the quiver are the irreducible modules of $\mathcal{W}(p)$, and the number of arrows $i \longrightarrow j$ is given by the dimension of $\on{Ext}_{\mathcal{W}(p)}^1(i, j)$.} 

Based on the previous proposition, we get that the Ext quiver $Q_{\mathcal{W}(p)}$ of $\mathcal{W}(p)$ is:

\begin{equation} \label{Ext-quiver-wp}
\begin{tikzpicture}[baseline=(current  bounding  box.center), scale=1,  transform shape]
\tikzset{>=stealth}
\tikzstyle{point}=[circle,draw,fill]

\node (1) at ( 0,0) []{$\bullet$};
\node (2) at ( 0,-2) []{$\bullet$};
\node (3) at ( 2,0) [] {$\bullet$};
\node (4) at ( 2,-2) [] {$\bullet$};
\node (5) at ( 6,0) [] {$\bullet$};
\node (6) at ( 6,-2) [] {$\bullet$};
\node (7) at ( 9,0) [] {$\bullet$};
\node (8) at ( 9,-2) [] {$\bullet$};

\node at (0,0.5)  {$X_1^+$};
\node at (0,-2.5) {$X_1^-$};
\node at (2,0.5)  {$X_2^+$};
\node at (2,-2.5) {$X_2^-$};
\node at (6,0.5)  {$X_{p-1}^+$};
\node at (6,-2.5) {$X_{p-1}^-$};
\node at (9,0.5)  {$X_p^+$};
\node at (9,-2.5) {$X_p^-$};

\node at (4,-1) {$\dots \dots$};

\draw [decoration={markings,mark=at position 1 with
    {\arrow[scale=1.2,>=stealth]{>}}},postaction={decorate}] (-0.25,-1.8) -- (-0.25,-0.2);
\draw [decoration={markings,mark=at position 1 with
    {\arrow[scale=1.2,>=stealth]{>}}},postaction={decorate}] (-0.1,-1.8) -- (-0.1,-0.2);
\draw [decoration={markings,mark=at position 1 with
    {\arrow[scale=1.2,>=stealth]{>}}},postaction={decorate}]  (0.25,-0.2) -- (0.25,-1.8);
\draw [decoration={markings,mark=at position 1 with
    {\arrow[scale=1.2,>=stealth]{>}}},postaction={decorate}]  (0.1,-0.2) -- (0.1,-1.8);

\draw [decoration={markings,mark=at position 1 with
    {\arrow[scale=1.2,>=stealth]{>}}},postaction={decorate}] (1.75,-1.8) -- (1.75,-0.2);
\draw [decoration={markings,mark=at position 1 with
    {\arrow[scale=1.2,>=stealth]{>}}},postaction={decorate}] (1.9,-1.8) -- (1.9,-0.2);
\draw [decoration={markings,mark=at position 1 with
    {\arrow[scale=1.2,>=stealth]{>}}},postaction={decorate}]  (2.25,-0.2) -- (2.25,-1.8);
\draw [decoration={markings,mark=at position 1 with
    {\arrow[scale=1.2,>=stealth]{>}}},postaction={decorate}]  (2.1,-0.2) -- (2.1,-1.8);

\draw [decoration={markings,mark=at position 1 with
    {\arrow[scale=1.2,>=stealth]{>}}},postaction={decorate}] (5.75,-1.8) -- (5.75,-0.2);
\draw [decoration={markings,mark=at position 1 with
    {\arrow[scale=1.2,>=stealth]{>}}},postaction={decorate}] (5.9,-1.8) -- (5.9,-0.2);
\draw [decoration={markings,mark=at position 1 with
    {\arrow[scale=1.2,>=stealth]{>}}},postaction={decorate}]  (6.25,-0.2) -- (6.25,-1.8);
\draw [decoration={markings,mark=at position 1 with
    {\arrow[scale=1.2,>=stealth]{>}}},postaction={decorate}]  (6.1,-0.2) -- (6.1,-1.8);
\end{tikzpicture}
\end{equation}
We notice that it is a wild quiver.

\section{The algebras $A(\mathcal{W}(p))$ and $\on{gr}A(\mathcal{W}(p))$ and their quivers}\label{Section5}
\subsection{The Zhu algebra $A(\mathcal{W}(p))$}\label{SectionZhuAlgebraA(W(p))}
Let $A(\mathcal{W}(p))$ be the Zhu algebra of $\mathcal{W}(p)$ (cf. Section~\ref{SectionC2ZhuAlgebra}). It follows that $A(\mathcal{W}(p))$ is an associative algebra with product given in Definition~\ref{ZhuAlgebra}. It is proved in \cite{Adamovic-Milas4} that $\dim R(\mathcal{W}(p))= \dim A(\mathcal{W}(p)) = 6p-1$.  
Therefore the construction in \cite[Thm. 5.9]{Adamovic-Milas2} provides that $A(\mathcal{W}(p))$ has the following decomposition:
\begin{equation*}\label{eq:1}
A(\mathcal{W}(p)) \cong \bigoplus_{i=2p}^{3p-1}M_{h_i}(\cc) \oplus \bigoplus_{i=1}^{p-1}I_{h_i}(\cc) \oplus \mathbb{C},
\end{equation*}
where $M_{h_i}(\cc) \cong M_2(\cc)$ and $I_{h_i}(\cc) \cong \cc[x]/(x^2)$ for all $i$ (here we use the notations of \cite{Adamovic-Milas2}). Thus the following proposition is  a consequence of Theorem 1.1. in \cite{Adamovic-Milas4}. 

\delete{:
\begin{equation*}
 [a] [b]= [a \ast b]= \Big[ \mathrm{Res}_x\Big(Y(a, x)\frac{(1+x)^{\on{wt}(a)}}{x}b\Big)\Big]=\Big[\sum_{i=0}^\infty \binom{\on{wt}(a)}{i}a_{i-1}b \Big]
\end{equation*}
where $[\cdot]$ is the equivalence class in the quotient. }

\begin{proposition} [\cite{Adamovic-Milas4}]\label{grA(W(p))=R(W(p))}
For the triplet vertex operator algebra $\mathcal{W}(p)$, there is an isomorphism of Poisson algebras:
\begin{align*}
\eta_{\mathcal{W}(p)}: R(\mathcal{W}(p)) \stackrel{\cong}{\longrightarrow} \on{gr}A(\mathcal{W}(p)).
\end{align*} 
\end{proposition}

Based on the description in \cite{Adamovic-Milas2}, the irreducible $\mathcal{W}(p)$-modules are $\mathbb{N}$-graded admissible with non-zero lowest weight space, and these lowest weight spaces give irreducible $A(\mathcal{W}(p))$-modules (cf. \cite{DLM1, DLM2, Zhu}).

Let us consider the Ext quiver of $A(\mathcal{W}(p))$ in the category $A(\mathcal{W}(p))\Mod$. The algebra $A(\mathcal{W}(p))$ is a direct sum of algebras, so its Ext quiver $Q_{A(\mathcal{W}(p))}$ is the reunion of the Ext quivers of the summands:
\begin{itemize}
\item for $1 \leq i \leq p-1$, $I_{h_i}(\cc) \cong \cc[x]/(x^2)$ so it is a basic algebra with a unique simple module. Thus its Ext quiver has only one vertex and the number of loops is equal to the dimension of $\text{Rad}(I_{h_i}(\cc) )/\text{Rad}^2(I_{h_i}(\cc))=\mathbb{C}(x)$, which is $1$. The irreducible module corresponds to $X_i^+(0)$. 
\item  for $2p \leq i \leq 3p-1$, $M_{h_i}(\cc) \cong M_2(\cc)$ so it is a semisimple algebra with only one irreducible module. Hence its Ext quiver is an isolated vertex corresponding to $X_{i-2p}^-(0)$ if $i \neq 2p$ and $X_p^-(0)$ otherwise. 
\item  $\cc$ is a semisimple algebra and has only one irreducible module. Thus its Ext quiver is an isolated vertex corresponding to $X_p^+(0)$.
\end{itemize}
We get that the quiver $Q_{A(\mathcal{W}(p))}$ is as follows:
 \begin{center}
  \begin{tikzpicture}[scale=1,  transform shape]
  \tikzset{>=stealth}
\tikzstyle{point}=[circle,draw,fill]

\node (1) at ( 0,0) []{$\bullet$};
\node (2) at ( 0,-2) []{$\bullet$};
\node (3) at ( 2,0) [] {$\bullet$};
\node (4) at ( 2,-2) [] {$\bullet$};
\node (5) at ( 6,0) [] {$\bullet$};
\node (6) at ( 6,-2) [] {$\bullet$};
\node (7) at ( 9,0) [] {$\bullet$};
\node (8) at ( 9,-2) [] {$\bullet$};

\node at (0.8,0)  {$X_1^+(0)$};
\node at (0.8,-2) {$X_1^-(0)$};
\node at (2.8,0)  {$X_2^+(0)$};
\node at (2.8,-2) {$X_2^-(0)$};
\node at (6.9,0)  {$X_{p-1}^+(0)$};
\node at (6.9,-2) {$X_{p-1}^-(0)$};
\node at (9.8,0)  {$X_p^+(0)$};
\node at (9.8,-2) {$X_p^-(0)$};

\node at (4.5,-1) {$\dots\dots$};

\draw[->]  (1) edge[out=125, in=55, looseness=15]  (1);
\draw[->]  (3) edge[out=125, in=55, looseness=15]  (3);
\draw[->]  (5) edge[out=125, in=55, looseness=15]  (5);
\end{tikzpicture}
 \end{center}

We notice here that even though $Q_{\mathcal{W}(p)}$ and $Q_{A(\mathcal{W}(p))}$ have the same number of vertices, their block decomposition is different. The non-trivial blocks of $Q_{\mathcal{W}(p)}$ indexed by $1 \leq s \leq p-1$ split in two to give two non-connected vertices in $Q_{A(\mathcal{W}(p))}$. This is due to the fact that the module block corresponding to $(X_s^+, X_s^-)$ in $Q_{\mathcal{W}(p)}$ is sent by Zhu's construction to the module block $(X_s^+(0), X_s^-(0))$. However this module block corresponds to the summand $M_{h_{s+2p}}(\cc)  \oplus I_{h_s}(\cc)$ in $A(\mathcal{W}(p))$ and $M_{h_{s+2p}}(\cc)$ is semisimple. We also notice that self-extensions appear in $Q_{A(\mathcal{W}(p))}$ whereas there is none in $Q_{\mathcal{W}(p)}$. We can therefore say that the algebra $A(\mathcal{W}(p))$ does not reflect completely the category of $\mathcal{W}(p)$-modules. In order to get more information on this category, we now look into the algebra $\on{gr}A(\mathcal{W}(p))$.

\subsection{The graded algebra $\on{gr}A(\mathcal{W}(p))$}\label{SectiongrA(W(p))}

If $X$ is an element of $\mathcal{W}(p)$, let $\overline{X}$ denote the image of $X$ through the composition $\mathcal{W}(p) \longrightarrow A(\mathcal{W}(p)) \longrightarrow \on{gr}A(\mathcal{W}(p)) \cong R(\mathcal{W}(p)) $. In \cite{Adamovic-Milas2} the relations defining $A(\mathcal{W}(p))$ are given. It is then straightforward to compute that $\on{gr}A(\mathcal{W}(p))$ is the quotient of $\cc[\overline{E},\overline{F},\overline{H},\overline{\omega}]$ by the ideal generated by the following relations:
\begin{equation*}
\begin{cases}
\overline{\omega}^{3p-1}, \overline{E}^2, \overline{F}^2, \\[5pt]
\overline{E}\overline{H}, \overline{F}\overline{H}, \overline{\omega}^p\overline{E}, \overline{\omega}^p\overline{F}, \overline{\omega}^p\overline{H}, \\[5pt]
\overline{H}^2-C_p\overline{\omega}^{2p-1}, \overline{E}\overline{F}+C_p\overline{\omega}^{2p-1},
\end{cases}       
\end{equation*} 
with $C_p=\frac{(4p)^{2p-1}}{(2p-1)!^2}$. Furthermore, $\on{gr}A(\mathcal{W}(p))$ is a graded algebra with $\text{deg }\overline{\omega}=2$ and $\text{deg }\overline{E}= \text{deg }\overline{F}= \text{deg }\overline{H}= 2p-1$.

\subsubsection{Ext quiver for the category of $\on{gr}A(\mathcal{W}(p))$-modules}
We look at $\on{gr}A(\mathcal{W}(p))$ as an algebra without the gradation. As $\on{gr}A(\mathcal{W}(p))$ is a commutative algebra, all simple modules are one dimensional. We also notice from the relations that the generators are nilpotent, and thus there is only one simple module, i.e., the trivial module $\cc$. It follows that the Ext quiver $Q_{\on{gr}A(\mathcal{W}(p))}$ has only one vertex, and the number of loops on said vertex is equal to $\on{dim}_\cc \ \on{Ext}_{\on{gr}A(\mathcal{W}(p))}^1(\cc, \cc)$. If we compute the first steps of the minimal resolution $P^\bullet$ of $\cc$ as a $\on{gr}A(\mathcal{W}(p))$-module, for example using Tate's approach \cite{Tate}, we see that $\on{dim}_\cc \ \on{Ext}_{\on{gr}A(\mathcal{W}(p))}^1(\cc, \cc)=\on{dim}_\cc \ P^1=4$ because $\on{gr}A(\mathcal{W}(p))$ has four generators. It follows that the Ext quiver $Q_{\on{gr}A(\mathcal{W}(p))}$ of $\on{gr}A(\mathcal{W}(p))$ is given by:
\[
\begin{tikzpicture}[scale=1,  transform shape]
  \tikzset{>=stealth}
\tikzstyle{point}=[circle,draw,fill]

\node (1) at ( 0,0) []{$\bullet$};

\draw[->]  (1) edge[out=125, in=55, looseness=15]  (1);
\draw[->]  (1) edge[out=35, in=325, looseness=15]  (1);
\draw[->]  (1) edge[out=305, in=235, looseness=15]  (1);
\draw[->]  (1) edge[out=215, in=145, looseness=15]  (1);
\end{tikzpicture}
\]

As explained in Proposition~\ref{grA(W(p))=R(W(p))}, we have $R(\mathcal{W}(p)) \cong \on{gr}A(\mathcal{W}(p))$ and thus $Q_{R(\mathcal{W}(p))}=Q_{\on{gr}A(\mathcal{W}(p))}$.

\subsubsection{Ext quiver for the category of graded $\on{gr}A(\mathcal{W}(p))$-modules}
We now take into account the gradation of $\on{gr}A(\mathcal{W}(p))$ and look at the simple graded modules, i.e., of the form $M=\bigoplus_{i \in \mathbb{Z}}M_i$ such that $\on{gr}A(\mathcal{W}(p))_i  M_j \subseteq M_{i+j}$. This category has an auto-equivalence functor $ M\mapsto M[1] $ with $ M[1]_i=M_{i+1}$. This functor is called the degree shifting functor. 

\delete{The algebra $\on{gr}A(\mathcal{W}(p))$ is positively graded and finite dimensional over $\cc$. Thus it is a graded Artinian $\cc$-algebra. According to \cite[4.5]{Gordon-Green}, every simple graded module has length one, i.e., only one summand of $M$ is non-zero. Therefore there exists $k \in \mathbb{Z}$ such that $M=M_k$, and
\begin{align*}
\begin{cases}
\overline{X}.M_k \subseteq M_{k+2p-1}=\{0\}, \text{ for } X=E,F,H, \\[5pt]
\overline{\omega}.M_k \subseteq M_{k+2}=\{0\},  \\[5pt]
\overline{\mathbf{1}}. M_k \subseteq M_{k}.
\end{cases}
\end{align*}
So the action of $\overline{E},\overline{F},\overline{H},\overline{\omega}$ is zero, and $\overline{\mathbf{1}}$ acts as the identity. Because $M$ is simple, it follows that $M_k=\cc$. }

Since $ \on{gr}A(\mathcal W(p))$ is a graded local algebra with nilpotent maximal ideal, all simple graded $\on{gr}A(\mathcal{W}(p))$-modules form an infinite family indexed on $\mathbb{Z}$. The simple module corresponding to $k \in \mathbb{Z}$ is written as
\begin{align*}
M_{(k)}=\bigoplus_{i \in \mathbb{Z}}M_i \text{ with } M_k=\cc, \ M_i=0 \text{ for } i \neq k.
\end{align*}
Therefore the Ext quiver $Q^{gr}_{\on{gr}A(\mathcal{W}(p))}$ of $\on{gr}A(\mathcal{W}(p))$ for the category of graded modules has infinitely many vertices.
We now want to determine $\on{Ext}_{\on{gr}A(\mathcal{W}(p))}^1(M_{(k)},M'_{(k')})$ in the category of graded modules. Let 
\begin{align*}
0 \longrightarrow M'_{(k')} \stackrel{f}{\longrightarrow} N \stackrel{g}{\longrightarrow} M_{(k)} \longrightarrow 0
\end{align*}
be such an extension. As homomorphisms of graded modules preserve the gradation, we obtain a family of short exact sequences:
\begin{align*}
0 \longrightarrow M'_{i}  \stackrel{f_i}{\longrightarrow} N_i  \stackrel{g_i}{\longrightarrow} M_{i} \longrightarrow 0
\end{align*}
for all $i \in \mathbb{Z}$. By applying the degree shifting functor, we can assume $k=0$. 
We claim that $\on{Ext}_{\on{gr}A(\mathcal{W}(p))}^1(M_{(k)},M'_{(k')})\neq 0$  only when $ k'=\deg(X)$ for $ X\in \{ E, F, H, \omega\}$. 
In fact, if $ k'=0$ then $\on{Ext}_{\on{gr}A(\mathcal{W}(p))}^1(M_{(k)},M'_{(k')})=0$ since $X\in \{ E, F, H, \omega\}$ acts as zero.    If $k'<0$, any non-zero vector  $ v\in N_0$ generates a submodule that defines a splitting.  
If $ k'>0$, then $(\on{gr}A(\mathcal{W}(p)) N_0)_{k'}\neq 0$, and so $\on{gr}A(\mathcal{W}(p)) N_0$ is a quotient of $\on{gr}A(\mathcal{W}(p)) $ of length 2. 

\delete{
{\bf  the rest of the cases computations can be removed} 
{\bf \em Case I: $k\neq k'$.} We first assume that $k \neq k'$. By looking at the sequences above for $i \neq k, k'$, $i=k$, and $i=k'$, we see that $N=\bigoplus_{i \in \mathbb{Z}}N_i$ with $N_k=N_{k'}=\cc$, $N_i=\{0\}$ if $i \neq k, k'$.

 $\bullet$ If $2< |k-k'| <2p-1$ or $|k-k'| > 2p-1$, then:
 \begin{align*}
\begin{cases}
 \overline{X}.N_k \subseteq N_{k+2p-1}=\{0\}, \text{ for } X=E,F,H, \\[5pt]
 \overline{X}.N_{k'} \subseteq N_{k'+2p-1}=\{0\}, \text{ for } X=E,F,H, \\[5pt]
\overline{\omega}.N_k \subseteq N_{k+2}=\{0\}, \\[5pt]
\overline{\omega}.N_{k'} \subseteq N_{k'+2}=\{0\}.
\end{cases}
\end{align*}
We see that $\overline{E},\overline{F},\overline{H},\overline{\omega}$ act as zero and $\overline{\mathbf{1}}$ as the identity, so $N=M_{(k)} \bigoplus M'_{(k')}$ and the sequence splits.

$\bullet$ If $k-k' = 2p-1$ then:
 \begin{align*}
\begin{cases}
\overline{X}.N_k \subseteq N_{k+2p-1}=\{0\}, \text{ for } X=E,F,H, \\[5pt]
\overline{X}.N_{k'} \subseteq N_{k'+2p-1}=N_k, \text{ for } X=E,F,H, \\[5pt]
\overline{ \omega}.N_k \subseteq N_{k+2}=\{0\}, \\[5pt]
\overline{ \omega}.N_{k'} \subseteq N_{k'+2}=\{0\}. 
\end{cases}
\end{align*}
So the action of $\overline{\omega}$ is zero, but the action of $\overline{E},\overline{F},\overline{H}$ need not be trivial. Let $\{u_k\},\{u_{k'}\}$ be the basis of $N_k$ and $N_{k'}$. In the basis $(u_k, u_{k'})$ the matrix of $\overline{X}$ is $\begin{pmatrix}
0 & x \\
0 & 0 
\end{pmatrix}$ for some $x \in \cc$. We define the homomorphism $s:M_{(k)} \longrightarrow N$ such that $s(y)=y u_k$ for all $y \in M_k=\cc$. Then:
 \begin{align*}
\begin{array}{ll}
s(\overline{X}.y)=s(0)=0, & s(\overline{\omega}.y)=s(0)=0, \\[5pt]
\overline{X}.s(y)=y\overline{X}.u_k=0, & \overline{\omega}.s(y)=y\overline{\omega}.u_k=0.
\end{array}
\end{align*}
Therefore $s$ is a well-defined homogeneous homomorphism of graded modules. We have $g \circ s(y)=g(y u_k)=y g(u_k)=y=\on{id}(y)$ for all $y \in M_{(k)}$. Hence $s$ is a section and the sequence splits.

 $\bullet$ If $k'-k = 2p-1$ then:
 \begin{align*}
\begin{cases}
\overline{X}.N_k \subseteq N_{k+2p-1}=N_{k'}, \text{ for } X=E,F,H, \\[5pt]
\overline{X}.N_{k'} \subseteq N_{k'+2p-1}=\{0\}, \text{ for } X=E,F,H, \\[5pt]
\overline{ \omega}.N_k \subseteq N_{k+2}=\{0\}, \\[5pt]
\overline{ \omega}.N_{k'} \subseteq N_{k'+2}=\{0\}. 
\end{cases}
\end{align*}
So the action of $\overline{\omega}$ is zero, but the action of $\overline{E},\overline{F},\overline{H}$ might not be. With the same notation as above, the matrix of $\overline{X}$ in the basis $(u_k, u_{k'})$ is $\begin{pmatrix}
0 & 0 \\
x & 0 
\end{pmatrix}$ for some $x \in \cc$. Assuming there exists a section $s:M_{(k)} \longrightarrow N$, i.e $g \circ s=\on{id}_{| M_{(k)}}$, then it requires the existence of $t \in \cc$ such that $s(y)=t y u_k$ for all $y \in M_k=\cc$. Indeed, we have $s(M_k) \subseteq N_k=\cc u_k$. As $s$ is non-trivial, we have $t \neq 0$. Then we have:
\begin{align*}
0=s(0)=s(\overline{X}.y)=\overline{X}.s(y)=ty\overline{X}.u_k=t y x u_{k'}.
\end{align*}
It follows that $xy=0$ for all $y \in \cc$. If $x \neq 0$, this is impossible, and so there is no section. 

So if $\overline{E}$, $\overline{F}$, or $\overline{H}$ acts non-trivially, then the sequence does not split. As the three variables play symmetric roles, there are three non-split cases: 1 nilpotent, 2 nilpotents, 3 nilpotents. Hence there are three non-isomorphic extensions $0 \longrightarrow M_{(k+2p-1)} \longrightarrow N \longrightarrow M_{(k)} \longrightarrow 0$.

$\bullet$ If $k-k' = 2$ then:
 \begin{align*}
\begin{cases}
\overline{X}.N_k \subseteq N_{k+2p-1}=\{0\}, \text{ for } X=E,F,H, \\[5pt]
\overline{X}.N_{k'} \subseteq N_{k'+2p-1}=\{0\}, \text{ for } X=E,F,H, \\[5pt]
\overline{ \omega}.N_k \subseteq N_{k+2}=\{0\}, \\[5pt]
\overline{ \omega}.N_{k'} \subseteq N_{k'+2}=N_k. 
\end{cases}
\end{align*}
So the actions of $\overline{E}$, $\overline{F}$ and $\overline{H}$ are zero, but the action of $\overline{\omega}$ need not be trivial. In the basis $(u_k, u_{k'})$ the matrix of $\overline{\omega}$ is $\begin{pmatrix}
0 & z \\
0 & 0 
\end{pmatrix}$ for some $z \in \cc$. We define the homomorphism $s:M_{(k)} \longrightarrow N$ such that $s(y)=y u_k$ for all $y \in M_k=\cc$. Then:
 \begin{align*}
\begin{array}{ll}
s(\overline{X}.y)=s(0)=0, & s(\overline{\omega}.y)=s(0)=0, \\[5pt]
\overline{X}.s(y)=y\overline{X}.u_k=0, & \overline{\omega}.s(y)=y\overline{\omega}.u_k=0.
\end{array}
\end{align*}
Therefore $s$ is a well-defined homogeneous homomorphism of graded modules. We have $g \circ s(y)=g(y u_k)=y g(u_k)=y$ for all $y \in M_{(k)}$. Hence $s$ is a section and the sequence splits.

$\bullet$ If $k'-k = 2$ then:
 \begin{align*}
\begin{cases}
\overline{X}.N_k \subseteq N_{k+2p-1}=\{0\}, \text{ for } X=E,F,H, \\[5pt]
\overline{X}.N_{k'} \subseteq N_{k'+2p-1}=\{0\}, \text{ for } X=E,F,H, \\[5pt]
\overline{ \omega}.N_k \subseteq N_{k+2}=N_{k'}, \\[5pt]
\overline{ \omega}.N_{k'} \subseteq N_{k'+2}=\{0\}. 
\end{cases}
\end{align*}
So the action of $\overline{E}$, $\overline{F}$ and $\overline{H}$ is zero. In the basis $(u_k, u_{k'})$, the matrix of $\overline{\omega}$ is $\begin{pmatrix}
0 & 0 \\
z & 0 
\end{pmatrix}$ for some $z \in \cc$. By proceeding the same way as the case $k'-k = 2p-1$, we see that if $z \neq 0$, then the sequence does not split. There is therefore one extension $0 \longrightarrow M_{(k+2)} \longrightarrow N \longrightarrow M_{(k)} \longrightarrow 0$ up to isomorphism.

$\bullet$ If $|k'-k| = 1$ then all variables act as zero, so the sequence splits.

{\bf \em Case II: $k=k'$.}  If $k = k'$, then $N_k=\cc^2$. We see that all variables act as zero, and thus $N=M_{(k)}\bigoplus M_{(k)}$.
}

Therefore we have 
 \begin{align*}
\on{Ext}_{\on{gr}A(\mathcal{W}(p))}^1(M_{(k)},M'_{(k')})= 
\begin{cases}
\cc^3 \text{ if } k'=k+2p-1,\\[5pt]
\cc \text{ if } k'=k+2,\\[5pt]
\{0\} \text{ otherwise}. 
\end{cases}
\end{align*}

Therefore the Ext quiver $Q^{gr}_{\on{gr}A(\mathcal{W}(p))}$ has infinitely many vertices indexed by $\mathbb{Z}$, and each vertex $k \in \mathbb{Z}$ has 4 neighbours:

\tikzset{Rightarrow/.style={double distance=3pt,>={Implies},->},
triple/.style={-, preaction={draw,Rightarrow}}}
 \begin{center}
  \begin{tikzpicture}[scale=1,  transform shape]
  \tikzset{>=stealth}
\tikzstyle{point}=[circle,draw,fill]

\node (1) at (-2,0) [label={[above]:$k-(2p-1)$}]{$\bullet$};
\node (2) at (0,0) [label={[shift={(0.25,0.03)}]:$k$}]{$\bullet$};
\node (3) at (2,0) [label={[above]:$k+(2p-1)$}]{$\bullet$};
\node (4) at ( 0,2) [label={[right]:$k-2$}]{$\bullet$};
\node (5) at ( 0,-2) [label={[right]:$k+2$}]{$\bullet$};

\draw [decoration={markings,mark=at position 1 with
    {\arrow[scale=1.2,>=stealth]{>}}},postaction={decorate}]  (-1.8,0) -- (-0.2,0);
\draw [decoration={markings,mark=at position 1 with
    {\arrow[scale=1.2,>=stealth]{>}}},postaction={decorate}] (-1.8,0.15) -- (-0.2,0.15);
\draw [decoration={markings,mark=at position 1 with
    {\arrow[scale=1.2,>=stealth]{>}}},postaction={decorate}] (-1.8,-0.15) -- (-0.2,-0.15);
\draw [decoration={markings,mark=at position 1 with
    {\arrow[scale=1.2,>=stealth]{>}}},postaction={decorate}]  (0.2,0) -- (1.8,0);
\draw [decoration={markings,mark=at position 1 with
    {\arrow[scale=1.2,>=stealth]{>}}},postaction={decorate}] (0.2,0.15) -- (1.8,0.15);
\draw [decoration={markings,mark=at position 1 with
    {\arrow[scale=1.2,>=stealth]{>}}},postaction={decorate}] (0.2,-0.15) -- (1.8,-0.15);
    
\draw [decoration={markings,mark=at position 1 with
    {\arrow[scale=1.2,>=stealth]{>}}},postaction={decorate}] (0,1.8) -- (0,0.2);
\draw [decoration={markings,mark=at position 1 with
    {\arrow[scale=1.2,>=stealth]{>}}},postaction={decorate}] (0,-0.2) -- (0,-1.8);

\end{tikzpicture}
 \end{center}
The horizontal arrows correspond to $\overline{E}$, $\overline{F}$, $\overline{H}$, and the vertical arrows correspond to $\overline{\omega}$. For example, when $p=2$, the quiver $Q^{gr}_{\on{gr}A(\mathcal{W}(2))}$ is given below:
\vspace{-2cm}
\tikzset{Rightarrow/.style={double distance=3pt,>={Implies},->},
triple/.style={-, preaction={draw,Rightarrow}}}
 \begin{center}
  \begin{tikzpicture}[scale=1,  transform shape]
  \tikzset{>=stealth}
\tikzstyle{point}=[circle,draw,fill]

\node (1) at (-6,0) []{$\dots$};
\node (2) at (-4,0) [label={[above]:-8}]{$\bullet$};
\node (3) at (-2,0) [label={[above]:-5}]{$\bullet$};
\node (4) at ( 0,0) [label={[shift={(0.5,0.03)}]:-2}]{$\bullet$};
\node (5) at ( 2,0) [label={[shift={(0.5,0.03)}]:1}]{$\bullet$};
\node (6) at ( 4,0) [label={[shift={(0.5,0.03)}]:4}]{$\bullet$};
\node (7) at ( 6,0) []{$\dots$};

\node (8) at (-6,-2) []{$\dots$};
\node (9) at (-4,-2) [label={[above right]:-6}]{$\bullet$};
\node (10) at (-2,-2) [label={[above right]:-3}]{$\bullet$};
\node (11) at (0,-2) [label={[above right]:0}]{$\bullet$};
\node (12) at (2,-2) [label={[above right]:3}]{$\bullet$};
\node (13) at (4,-2) [label={[above right]:6}]{$\bullet$};
\node (14) at (6,-2) []{$\dots$};

\node (15) at (-6,-4) []{$\dots$};
\node (16) at (-4,-4) [label={[above right]:-4}]{$\bullet$};
\node (17) at (-2,-4) [label={[above right]:-1}]{$\bullet$};
\node (18) at (0,-4) [label={[above right]:2}]{$\bullet$};
\node (19) at (2,-4) [label={[above right]:5}]{$\bullet$};
\node (20) at (4,-4) [label={[above right]:8}]{$\bullet$};
\node (21) at (6,-4) []{$\dots$};

\draw[triple]  (1) -- (2);
\draw[triple]  (2) -- (3);
\draw[triple]  (3) -- (4);
\draw[triple]  (4) -- (5);
\draw[triple]  (5) -- (6);
\draw[triple]  (6) -- (7);

\draw[triple]  (8) -- (9);
\draw[triple]  (9) -- (10);
\draw[triple]  (10) -- (11);
\draw[triple]  (11) -- (12);
\draw[triple]  (12) -- (13);
\draw[triple]  (13) -- (14);

\draw[triple]  (15) -- (16);
\draw[triple]  (16) -- (17);
\draw[triple]  (17) -- (18);
\draw[triple]  (18) -- (19);
\draw[triple]  (19) -- (20);
\draw[triple]  (20) -- (21);

\draw[->] (2) -- (9);
\draw[->] (9) -- (16);

\draw[->] (3) -- (10);
\draw[->] (10) -- (17);

\draw[->] (4) -- (11);
\draw[->] (11) -- (18);

\draw[->] (5) -- (12);
\draw[->] (12) -- (19);

\draw[->] (6) -- (13);
\draw[->] (13) -- (20);

\draw[->, dashed]  (16) edge[out=-145, in=65, looseness=1]  (4);
\draw[->, dashed]  (17) edge[out=-145, in=65, looseness=1]  (5);
\draw[->, dashed]  (18) edge[out=-145, in=65, looseness=1]  (6);
\draw[->, dashed]  (19) edge[out=-145, in=65, looseness=1]  (7);
\end{tikzpicture}
 \end{center}
\vspace{-1.5cm}
The dashed arrows correspond to $\overline{\omega}$ and go behind the quiver. As the quiver does not fit on the plane without self-intersection, we can draw it on a cylinder:
\begin{center}

\includegraphics[scale=0.25]{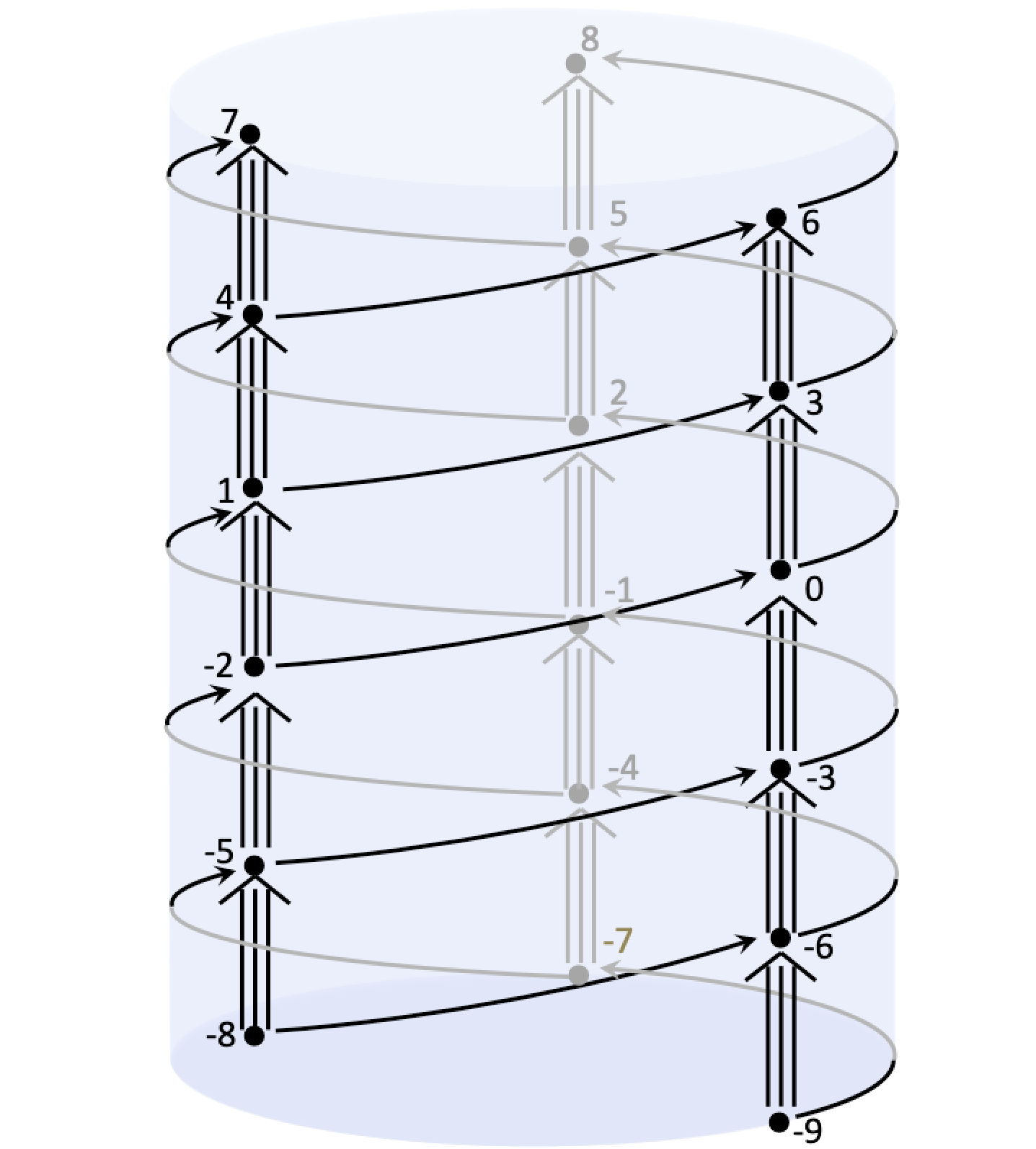}
\end{center}

We notice that if we were to collapse $Q^{gr}_{\on{gr}A(\mathcal{W}(p))}$ on one vertex, then the arrows corresponding to each variable would merge and we would obtain $Q_{\on{gr}A(\mathcal{W}(p))}$. Therefore there is a folding-unfolding relationship between the quivers of the categories of $\on{gr}A(\mathcal{W}(p))$-modules and graded $\on{gr}A(\mathcal{W}(p))$-modules.

\begin{remark}
The graded algebra $\on{gr}A(\mathcal{W}(p))$ can be equipped with a derivation $\overline{D}$ coming from $L(-1)$. The derivation is zero on the algebra but might be non-trivial on a module. It is a linear $\on{gr}A(\mathcal{W}(p))$-equivariant homomorphism. As $L(-1)$ acts on $\mathcal{W}(p)$-modules with degree 1, it is natural to specify that $\overline{D}$ acts on graded $\on{gr}A(\mathcal{W}(p))$-modules with degree 1. The same reasoning as above leads to a new non-splitting case when $k'=k+1$. Thus, in the quiver of $(\on{gr}A(\mathcal{W}(p)),\overline{D})$ in the category of graded modules, every vertex has 6 neighbours:
\tikzset{Rightarrow/.style={double distance=3pt,>={Implies},->},
triple/.style={-, preaction={draw,Rightarrow}}}
 \begin{center}
  \begin{tikzpicture}[scale=1,  transform shape]
  \tikzset{>=stealth}
\tikzstyle{point}=[circle,draw,fill]

\node (1) at (-2,0) [label={[above]:$k-(2p-1)$}]{$\bullet$};
\node (2) at (0,0) [label={[shift={(0.2,0.03)}]:$k$}]{$\bullet$};
\node (3) at (2,0) [label={[above]:$k+(2p-1)$}]{$\bullet$};
\node (4) at ( 0,2) [label={[right]:$k-2$}]{$\bullet$};
\node (5) at ( 0,-2) [label={[below right]:$k+2$}]{$\bullet$};
\node (6) at ( 2,2) [label={[right]:$k+1$}]{$\bullet$};
\node (7) at ( -2,-2) [label={[below right]:$k-1$}]{$\bullet$};

\draw [decoration={markings,mark=at position 1 with
    {\arrow[scale=1.2,>=stealth]{>}}},postaction={decorate}]  (-1.8,0) -- (-0.2,0);
\draw [decoration={markings,mark=at position 1 with
    {\arrow[scale=1.2,>=stealth]{>}}},postaction={decorate}] (-1.8,0.15) -- (-0.2,0.15);
\draw [decoration={markings,mark=at position 1 with
    {\arrow[scale=1.2,>=stealth]{>}}},postaction={decorate}] (-1.8,-0.15) -- (-0.2,-0.15);
\draw [decoration={markings,mark=at position 1 with
    {\arrow[scale=1.2,>=stealth]{>}}},postaction={decorate}]  (0.2,0) -- (1.8,0);
\draw [decoration={markings,mark=at position 1 with
    {\arrow[scale=1.2,>=stealth]{>}}},postaction={decorate}] (0.2,0.15) -- (1.8,0.15);
\draw [decoration={markings,mark=at position 1 with
    {\arrow[scale=1.2,>=stealth]{>}}},postaction={decorate}] (0.2,-0.15) -- (1.8,-0.15);
    
\draw [decoration={markings,mark=at position 1 with
    {\arrow[scale=1.2,>=stealth]{>}}},postaction={decorate}] (0,1.8) -- (0,0.2);
\draw [decoration={markings,mark=at position 1 with
    {\arrow[scale=1.2,>=stealth]{>}}},postaction={decorate}] (0,-0.2) -- (0,-1.8);
    
\draw [decoration={markings,mark=at position 1 with
    {\arrow[scale=1.2,>=stealth]{>}}},postaction={decorate}] (-1.8,-1.8) -- (-0.2,-0.2);
    
\draw [decoration={markings,mark=at position 1 with
    {\arrow[scale=1.2,>=stealth]{>}}},postaction={decorate}] (0.2,0.2) -- (1.8,1.8);
\end{tikzpicture}
 \end{center}
\end{remark}

\begin{remark}
The quivers $Q_{\mathcal{W}(p)}$, $Q_{A(\mathcal{W}(p))}$ and $Q_{\on{gr}A(\mathcal{W}(p))}=Q_{R(\mathcal{W}(p))}$ share similarities. We see that from $Q_{\mathcal{W}(p)}$ to $Q_{A(\mathcal{W}(p))}$, the number of vertices remains the same while the number of arrows decreases. \delete{, which illustrates Theorem~\ref{TheoremZhuAlgebra}}If we go from $Q_{\mathcal{W}(p)}$ to $Q_{R(\mathcal{W}(p))}$, then the number of vertices collapses to one, but the arrows are somewhat preserved. Indeed, if in $Q_{\mathcal{W}(p)}$ we identify the vertices $X_i^+$ and $X_i^-$, $1 \leq i \leq p$, then we obtain $p-1$ copies of $Q_{R(\mathcal{W}(p))}$ and an isolated vertex. We can illustrate the properties of the quivers which are transferred from one algebra to another by the following diagram:
\[
\begin{tikzpicture}[scale=1,  transform shape]
\tikzset{>=stealth}
\tikzstyle{point}=[circle,draw,fill]

\node (1) at (0,0) []{$Q_{\mathcal{W}(p)}$};
\node (2) at (4,2) []{$Q_{A(\mathcal{W}(p))}$};
\node (3) at (4,-2) []{$Q_{R(\mathcal{W}(p))}$};

\draw [decoration={markings,mark=at position 1 with
    {\arrow[scale=1.2,>=stealth]{>}}},postaction={decorate}] (1) -- node [above, midway, sloped] {preserve vertices} node [below, midway, sloped] {arrows drop} (2);
\draw [decoration={markings,mark=at position 1 with
    {\arrow[scale=1.2,>=stealth]{>}}},postaction={decorate}] (1) -- node [below, midway, sloped] {preserve arrows} node [above, midway, sloped] {vertices drop} (3);
\end{tikzpicture}.
\]
Here $Q_{A(\mathcal{W}(p))}$ is the Ext quiver for the category of $A(V)$-modules. We note that $A(V)\Mod^{fil}$ is not an abelian category, and thus its Ext quiver would require a different interpretation.
\end{remark}

\section{Algebras related to $\mathcal{W}(p)$}\label{Section6}
\subsection{The Ext algebra of $\mathcal{W}(p)$}\label{Section6.1}

It is proved in \cite{Nagatomo-Tsuchiya} that the category of logarithmic $\mathcal{W}(p)$-modules is equivalent as an abelian category to the category of finite dimensional modules for the restricted quantum group $\overline{U}_q(\on{sl}_2)$ at $q=e^{\frac{\pi i}{p}}$. Therefore their Yoneda algebras are isomorphic.  The Yoneda algebra and the indecomposable modules of $\overline{U}_q(\on{sl}_2)$ were determined in \cite{GSTF}.

\begin{remark}
The abelian category of finite dimensional  $\overline{U}_q(\on{sl}_2)$-modules was also investigated earlier in \cite{Suter}. The study of some infinite dimensional modules over $\overline{U}_q(\on{sl}_2)$ was carried out in \cite{Xiao} earlier too. As conjectured in \cite{GSTF}, the two categories should be equivalent as ribbon braided monoidal categories (up to a cocyle twist). The Grothendieck ring of $\overline{U}_q(\on{sl}_2)$ has been determined in \cite{FGST}. This conjecture has been proved by a sequence of papers by various authors (see a recent preprint \cite[Section 1]{FL} for a list of references).
\end{remark}

For any $1 \leq s \leq p-1$, the vector space $\on{Ext}_{\mathcal{W}(p)}^{1}(X_s^-, X_s^+)$ (resp. $\on{Ext}_{\mathcal{W}(p)}^{1}(X_s^+, X_s^-)$) is of dimension $2$ with basis $(\alpha^s_1,\alpha^s_2)$ (resp. $(\beta^s_1,\beta^s_2)$) (see Proposition~\ref{PropExtW(p)}). For any $1 \leq s \leq p$, we also write $e_s^{\pm}$ for the identity in $\on{Ext}_{\mathcal{W}(p)}^0(X_s^\pm, X_s^\pm)$.

\begin{theorem}[\cite{GSTF}]\label{W(p)_Yoneda}
Set $X_{\mathcal{W}(p)}=\bigoplus_{L \in \on{Irr}(\mathcal{W}(p))}L$. The Ext algebra
\[
\on{Ext}_{\mathcal{W}(p)}^*(X_{\mathcal{W}(p)}, X_{\mathcal{W}(p)})=\bigoplus_{s=1}^p\on{Ext}_{\mathcal{W}(p)}^*(X_s^+ \oplus X_s^-, X_s^+ \oplus X_s^-)
\]
 of $\mathcal{W}(p)$ is generated in degree $0$ and $1$ by the $e^\pm_s$ and $\alpha^s_i, \beta^s_j$. The non trivial relations are
\[
\alpha^s_1\beta^s_2+\alpha^s_2\beta^s_1=\beta^s_1\alpha^s_2+\beta^s_2  \alpha^s_1=0.
\]
for all $1 \leq s \leq p-1$.
\end{theorem}

\begin{corollary}
For any $1 \leq s \leq p-1$, we have $\text{proj.dim }X_s^{\pm}=+\infty$, and so the global dimension of $\mathcal{W}(p)$ is infinite.
\end{corollary}

\delete{
In this section, we compute the Ext algebra $\on{Ext}_{\mathcal{W}(p)}^*(X_{\mathcal{W}(p)}, X_{\mathcal{W}(p)})$ of $\mathcal{W}(p)$ where $X= \bigoplus_{L \in \on{Irr}(\mathcal{W}(p))}L$ is the sum of all irreducible $\mathcal{W}(p)$-modules. We see $\on{Ext}_{\mathcal{W}(p)}^n$ as the $n$-th derived functor of $\on{Hom}$ in the category of logarithmic modules $\mathcal{W}(p)\Mod^{log}$. We will then establish connections between $\on{Ext}_{\mathcal{W}(p)}^*(X_{\mathcal{W}(p)}, X_{\mathcal{W}(p)})$, the path algebra $kQ_{\mathcal{W}(p)}$ and $\on{End}_{\mathcal{W}(p)}(\bigoplus_{L \in \on{Irr}(\mathcal{W}(p))}P_L)$ where $P_L$ is the projective cover of the $\mathcal{W}(p)$-module $L$.

Set $1 \leq s \leq p-1$ and define $\partial_1: P_s^-\oplus P_s^- \longrightarrow P_s^+$ as follows:
\begin{center}
 \begin{tikzpicture}[scale=1, transform shape]
\tikzset{>=stealth}
\node (1) at (0,1.5) []{$X_s^-$};
\node (2) at (-1,0) []{$X_s^+$};
\node (3) at (1,0) []{$X_s^+$};
\node (4) at (0,-1.5) []{$X_s^-$};
\draw[->]  (1) -- (2);
\draw[->]  (1) -- (3);
\draw[->]  (2) -- (4);
\draw[->]  (3) -- (4) ;

\node (5) at (3,1.5) []{$X_s^-$};
\node (6) at (2,0) []{$X_s^+$};
\node (7) at (4,0) []{$X_s^+$};
\node (8) at (3,-1.5) []{$X_s^-$};
\draw[->]  (5) -- (6);
\draw[->]  (5) -- (7);
\draw[->]  (6) -- (8);
\draw[->]  (7) -- (8) ;

\draw[-]  (0.5,-0.25) -- (2.5,-0.25) ;
\draw[-]  (-0.5,1.75) -- (3.5,1.75) ;
\draw[-]  (0.5,-0.25) -- (-0.5,1.75) ;
\draw[-]  (2.5,-0.25) -- (3.5,1.75) ;

\node (9) at (7,1.5) []{$X_s^+$};
\node (10) at (6,0) []{$X_s^-$};
\node (11) at (8,0) []{$X_s^-$};
\node (12) at (7,-1.5) []{$X_s^+$};
\draw[->]  (9) -- (10);
\draw[->]  (9) -- (11);
\draw[->]  (10) -- (12);
\draw[->]  (11) -- (12) ;

\draw[-]  (5.5,0.25) -- (8.5,0.25) ;
\draw[-]  (6.5,-1.75) -- (7.5,-1.75) ;
\draw[-]  (5.5,0.25) -- (6.5,-1.75) ;
\draw[-]  (8.5,0.25) -- (7.5,-1.75) ;

\draw[->]  (4,0.5) -- (5.25,0.25) node[midway, above] {$\partial_1$} ;
\end{tikzpicture}
\end{center}
where the circled part in $P_s^-\oplus P_s^-$ is sent to the one in $P_s^+$, sending composition factor to composition factor, and the rest of $P_s^-\oplus P_s^-$ is sent to $0$. This construction gives a well-defined homomorphism. Furthermore, $\partial_1(\text{Rad}(P_s^-\oplus P_s^-)) \subseteq \text{Rad}^2(P_s^+) \subset \text{Rad}(P_s^+)$. By repeating this process we can construct an acyclic complex:
\begin{align*}
P_{s,+}^\bullet: \quad \dots \longrightarrow P_s^-\oplus P_s^- \oplus P_s^-\oplus P_s^- \xrightarrow{\partial_3} P_s^+\oplus P_s^+ \oplus P_s^+ \xrightarrow{\partial_2} P_s^-\oplus P_s^- \xrightarrow{\partial_1} P_s^+ \xrightarrow{\partial_0} X_s^+
\end{align*} 
where $\partial_0$ is the projection of the projective cover of $X_s^+$. Furthermore, for each $n \in \mathbb{N}^*$, we have $\partial_n(P_{s,+}^n) \subseteq \text{Rad}(P_{s,+}^{n-1})$. As $P_s^\pm$ is projective, we deduce that $P_{s,+}^\bullet$ is a minimal projective resolution of $X_s^+$. Likewise, we construct $P_{s,-}^\bullet$ and see that it is a minimal projective resolution of $X_s^-$.

Let $i \in \mathbb{N}^*$. Because the resolutions are minimal and the $X_s^\pm$ are simple, it follows that:
\[ \on{Ext}_{\mathcal{W}(p)}^{2i}(X_{s_1}^{\epsilon_1}, X_{s_2}^{\epsilon_2})  = \on{Hom}_{\mathcal{W}(p)}(P_{s_1,\epsilon_1}^{2i}, X_{s_2}^{\epsilon_2}) = \bigoplus_{k=1}^{2i+1} \on{Hom}_{\mathcal{W}(p)}(P_{s_1}^{\epsilon_1}, X_{s_2}^{\epsilon_2}). \]
But as $P_{s_1}^{\epsilon_1}$ is the projective cover of $X_{s_1}^{\epsilon_1}$, Lemma~\ref{LemmaHom} implies that if $s_1 \neq s_2$, we have $\on{Ext}_{\mathcal{W}(p)}^{2i}(X_{s_1}^{\epsilon_1}, X_{s_2}^{\epsilon_2}) =\{0\}$. Similarly, for $i \in \mathbb{N}$, $\on{Ext}_{\mathcal{W}(p)}^{2i+1}(X_{s_1}^{\epsilon_1}, X_{s_2}^{\epsilon_2}) =\{0\}$ if $s_1 \neq s_2$. Hence the Ext algebra will be composed of blocks labeled $1,2,\dots, p$ corresponding to the labels of the simple modules (we will see shortly that there are in fact $p+1$ blocks because the last one splits in two). We therefore obtain:
\begin{align*}
\begin{array}{r}
\on{Ext}_{\mathcal{W}(p)}^*(X_{\mathcal{W}(p)}, X_{\mathcal{W}(p)})= \displaystyle  \bigoplus_{s=1}^p\bigoplus_{i=0}^\infty \left( \on{Ext}_{\mathcal{W}(p)}^i(X_s^+, X_s^+) \oplus\on{Ext}_{\mathcal{W}(p)}^i(X_s^-, X_s^-)  \right. \\[5pt]
  \left.   \oplus\on{Ext}_{\mathcal{W}(p)}^i(X_s^+, X_s^-) \oplus\on{Ext}_{\mathcal{W}(p)}^i(X_s^-, X_s^+)\right).
 \end{array}
\end{align*} 

Set $1 \leq s \leq p-1$. From the minimality of the resolution $P_{s,+}^\bullet$, we have:
\begin{equation*}
\on{Ext}_{\mathcal{W}(p)}^n(X_s^+, X_s^\pm)=\text{Hom}_{\mathcal{W}(p)}(P_{s,+}^n, X_s^\pm)=\begin{cases}
    \displaystyle  \bigoplus_{i=1}^{n+1}\text{Hom}_{\mathcal{W}(p)}(P_{s}^-, X_s^\pm) \text{ if $n$ is odd},\\[10pt]
    \displaystyle  \bigoplus_{i=1}^{n+1}\text{Hom}_{\mathcal{W}(p)}(P_{s}^+, X_s^\pm) \text{ if $n$ is even}.
    \end{cases}       
\end{equation*}
We then apply Lemma~\ref{LemmaHom} to obtain the following equalities:
\begin{equation*}
\begin{cases}
     \on{Ext}_{\mathcal{W}(p)}^{2i}(X_s^+, X_s^+) \cong \cc^{2i+1},\\[5pt]
     \on{Ext}_{\mathcal{W}(p)}^{2i}(X_s^+, X_s^-) \cong \{0\}, \\[5pt]
     \on{Ext}_{\mathcal{W}(p)}^{2i+1}(X_s^+, X_s^+) \cong \{0\}, \\[5pt]
     \on{Ext}_{\mathcal{W}(p)}^{2i+1}(X_s^+, X_s^-) \cong \cc^{2i+2}.
    \end{cases}       
\end{equation*}
Similarly we compute the values of $ \on{Ext}_{\mathcal{W}(p)}^{n}(X_s^-, X_s^\pm)$. The formulas can be obtained from those above by permuting $X_s^+$ and $X_s^-$. 

We also know that $X_p^\pm$ is projective so $\on{Ext}_{\mathcal{W}(p)}^n(X_p^{\epsilon_1}, X_p^{\epsilon_2})=\{0\}$ for any $n\geq 1$, $\epsilon_1$, $\epsilon_2=\pm$. Furthermore, $\on{Ext}_{\mathcal{W}(p)}^0(X_p^{\epsilon_1}, X_p^{\epsilon_2})=\text{Hom}_{\mathcal{W}(p)}(X_p^{\epsilon_1}, X_p^{\epsilon_2})= \cc$ if $\epsilon_1 = \epsilon_2$ and $\{0\}$ otherwise.
    
We have thus proved the following result:
\begin{theorem}\label{dimExtn}
In the direct sum of the Ext algebra given above, for $1 \leq s \leq p-1$ and $i \geq 0$, we have:
\begin{equation*}
\begin{cases}
     \on{Ext}_{\mathcal{W}(p)}^{2i}(X_s^+, X_s^+) \cong \cc^{2i+1},\\[5pt]
     \on{Ext}_{\mathcal{W}(p)}^{2i+1}(X_s^+, X_s^-) \cong \cc^{2i+2},
    \end{cases}  
    \begin{cases}
     \on{Ext}_{\mathcal{W}(p)}^{2i}(X_s^-, X_s^-) \cong \cc^{2i+1},\\[5pt]
     \on{Ext}_{\mathcal{W}(p)}^{2i+1}(X_s^-, X_s^+) \cong \cc^{2i+2},
    \end{cases} 
     \begin{cases}
     \on{Ext}_{\mathcal{W}(p)}^{0}(X_p^+, X_p^+) \cong \cc,\\[5pt]
     \on{Ext}_{\mathcal{W}(p)}^{0}(X_p^-, X_p^-) \cong \cc.
    \end{cases}       
\end{equation*}
All the other summands are zero. We can thus rewrite:
\begin{align*}
\begin{array}{cl}
\on{Ext}_{\mathcal{W}(p)}^*(X_{\mathcal{W}(p)}, X_{\mathcal{W}(p)})= & \displaystyle \bigoplus_{s=1}^{p-1}\left[ \bigoplus_{k \text{ even}} \left( \on{Ext}_{\mathcal{W}(p)}^{k}(X_s^+, X_s^+) \oplus \on{Ext}_{\mathcal{W}(p)}^{k}(X_s^-, X_s^-)\right) \right.  \\
 & \quad \quad \displaystyle \oplus \left. \bigoplus_{k \text{ odd}} \left( \on{Ext}_{\mathcal{W}(p)}^{k}(X_s^+, X_s^-) \oplus  \on{Ext}_{\mathcal{W}(p)}^{k}(X_s^-, X_s^+) \right) \right]  \oplus \cc \oplus \cc.
\end{array}
\end{align*}
\end{theorem}

We see that $\on{Ext}_{\mathcal{W}(p)}^*(X_{\mathcal{W}(p)}, X_{\mathcal{W}(p)})$ decomposes into $p+1$ blocks, one for each $1 \leq s \leq p-1$ and the $p$-th block splits in two with one for $X_p^+$ and one for $X_p^-$.

\begin{corollary}
For any $1 \leq s \leq p-1$, we have $\text{proj.dim }X_s^{\pm}=+\infty$, and so the global dimension of $\mathcal{W}(p)$ is infinite.
\end{corollary}

We now want to show that the Ext algebra is generated by its terms of degree 0 and 1. The last two blocks only contain degree 0 elements, so there is nothing to show. Fix $1 \leq s \leq p-1$ so we are in a non-trivial block of the Ext algebra. We know that $\on{dim} \on{Ext}_{\mathcal{W}(p)}^{1}(X_s^-, X_s^+)= \on{dim} \on{Ext}_{\mathcal{W}(p)}^{1}(X_s^+, X_s^-)=2$, so we fix a basis $(\alpha_1,\alpha_2)$ of $\on{Ext}_{\mathcal{W}(p)}^{1}(X_s^-, X_s^+)$ and a basis $(\beta_1,\beta_2)$ of $\on{Ext}_{\mathcal{W}(p)}^{1}(X_s^+, X_s^-)$. We have a homomorphism $\pi: \on{Ext}_{\mathcal{W}(p)}^{1}(X_s^-, X_s^+) \otimes \on{Ext}_{\mathcal{W}(p)}^{1}(X_s^+, X_s^-) \longrightarrow \on{Ext}_{\mathcal{W}(p)}^{2}(X_s^+, X_s^+)$ given by the Yoneda product. 

The proof of the next proposition is given in Section~\ref{Proofs_of_Section_6.1} in the appendix.

\begin{proposition}\label{RelationsInExt}
The homomorphism 
\[\pi: \on{Ext}_{\mathcal{W}(p)}^{1}(X_s^-, X_s^+) \otimes \on{Ext}_{\mathcal{W}(p)}^{1}(X_s^+, X_s^-) \longrightarrow \on{Ext}_{\mathcal{W}(p)}^{2}(X_s^+, X_s^+)\]
is surjective and its kernel is 1-dimensional with basis $\alpha_1\otimes\beta_2+\alpha_2 \otimes \beta_1$. 

By symmetry, the homomorphism $\on{Ext}_{\mathcal{W}(p)}^{1}(X_s^+, X_s^-) \otimes \on{Ext}_{\mathcal{W}(p)}^{1}(X_s^-, X_s^+) \longrightarrow \on{Ext}_{\mathcal{W}(p)}^{2}(X_s^-, X_s^-)$ is surjective and its kernel is 1-dimensional with basis $\beta_1 \otimes \alpha_2+\beta_2 \otimes \alpha_1$.
\end{proposition}

We can now show the main result of this section.

\begin{theorem}\label{GeneratorsOfExt}
The Ext algebra $\on{Ext}_{\mathcal{W}(p)}^*(X_{\mathcal{W}(p)}, X_{\mathcal{W}(p)})$ of $\mathcal{W}(p)$ is generated in degree $0$ and $1$.
\end{theorem}

The proof is based on an explicit realisation of a Yoneda product in $\on{Ext}_{\mathcal{W}(p)}^*(X_{\mathcal{W}(p)}, X_{\mathcal{W}(p)})$ and can be found in Section~\ref{Proofs_of_Section_6.1} in the appendix.
}

Because of Theorem~\ref{W(p)_Yoneda}, we know that $\on{Ext}_{\mathcal{W}(p)}^0(X_{\mathcal{W}(p)}, X_{\mathcal{W}(p)})\oplus\on{Ext}_{\mathcal{W}(p)}^1(X_{\mathcal{W}(p)}, X_{\mathcal{W}(p)})$ is a finite dimensional generating subspace of the Yoneda algebra. Using the dimensions of the summands of the Yoneda algebra found in \cite{GSTF}, one can determine that $\on{dim}_\cc (\bigoplus_{i \leq n} \on{Ext}_{\mathcal{W}(p)}^{i}(X_{\mathcal{W}(p)}, X_{\mathcal{W}(p)}))=(n+1)n(p-1)+2n(p-1)+2p$. This leads to the following corollary. The readers are referred to \cite{Krause-Lenagan} for definitions and properties of Gelfand-Kirillov dimensions of associative algebras over a field. In case an associative algebra is finitely generated and commutative, the Gelfand-Kirillov dimension agrees with the Krull dimension, which is the dimension of the corresponding affine algebraic variety. 

\begin{corollary}
The Gelfand-Kirillov dimension of the algebra $\on{Ext}_{\mathcal{W}(p)}^{*}(X_{\mathcal{W}(p)}, X_{\mathcal{W}(p)})$ is $2$.
\end{corollary}

We end this section by presenting the algebra $\on{Ext}_{\mathcal{W}(p)}^*(X_{\mathcal{W}(p)}, X_{\mathcal{W}(p)})$ as a quotient of the path algebra $kQ_{\mathcal{W}(p)}$, and we give the relations between the generators. We add to the notations of Section~\ref{QuiverWp} that for any $1 \leq s \leq p-1$, the two arrows from $X_s^-$ to $X_s^+$ in $Q_{\mathcal{W}(p)}$ are called $\alpha_{s,1}$ and $\alpha_{s,2}$. Furthermore, the two arrows from $X_s^+$ to $X_s^-$ are called $\beta_{s,1}$ and $\beta_{s,2}$. Let $\mathcal{I}_{\mathcal{W}(p)}$ be the ideal of $kQ_{\mathcal{W}(p)}$ generated by the relations $\alpha_{s,1}\beta_{s,2}+\alpha_{s,2} \beta_{s,1}$, $\beta_{s,1} \alpha_{s,2}+\beta_{s,2} \alpha_{s,1}$ for $1 \leq s \leq p-1$. Based on Theorem~\ref{W(p)_Yoneda}, we have a natural isomorphism of graded algebras 
\begin{align}\label{eq:Yoneda}
kQ_{\mathcal{W}(p)}/\mathcal{I}_{\mathcal{W}(p)} \stackrel{\cong}{\longrightarrow} \on{Ext}_{\mathcal{W}(p)}^*(X_{\mathcal{W}(p)}, X_{\mathcal{W}(p)})
\end{align}
that sends $(\alpha_{s,1},\alpha_{s,2})$ (resp. $(\beta_{s,1},\beta_{s,2})$) to $(\alpha^s_1, \alpha^s_2)$ (resp. $(\beta^s_1, \beta^s_2)$), and the product in $kQ_{\mathcal{W}(p)}$ is sent to the Yoneda product in $\on{Ext}_{\mathcal{W}(p)}^*(X_{\mathcal{W}(p)}, X_{\mathcal{W}(p)})$. Naturally, the paths of length $0$ are sent to the identity on the respective simple modules. 

\begin{corollary}
In the case of $\mathcal{W}(p)$, the homomorphism $\varphi$ of Lemma~\ref{kQandExt} is surjective.
\end{corollary}

\delete{
It turns out that this homomorphism is an isomorphism (see Section~\ref{Proofs_of_Section_6.1} in the appendix for the proof):
\begin{proposition}\label{kQandExtforW(p)}
The homomorphism $kQ_{\mathcal{W}(p)}/\mathcal{I}_{\mathcal{W}(p)} \longrightarrow \on{Ext}_{\mathcal{W}(p)}^*(X_{\mathcal{W}(p)}, X_{\mathcal{W}(p)})$ is an algebra isomorphism. Therefore the relations of $\mathcal{I}_{\mathcal{W}(p)}$ completely define $\on{Ext}_{\mathcal{W}(p)}^*(X_{\mathcal{W}(p)}, X_{\mathcal{W}(p)})$. In particular, $\on{Ext}_{\mathcal{W}(p)}^*(X_{\mathcal{W}(p)}, X_{\mathcal{W}(p)})$ is generated in degree $0$ and $1$ with relations in degree $2$.
\end{proposition}
}

\subsection{The endomorphism algebra of the projective covers}
We know from Lemma~\ref{kQandEnd} that for a finite dimensional algebra $A$ over an algebraically closed field, we have $kQ_A/I \cong  \on{End}_A(\bigoplus_{L \in \on{Irr}(A)} P_L)^{op}$. We would like to apply this reasoning to $\mathcal{W}(p)$. Indeed, we have the Ext quiver $Q_{\mathcal{W}(p)}$ and the projective covers $P_s^\pm$ for $1 \leq s \leq p$. In this section, we will find an admissible ideal $I_{\mathcal{W}(p)}$ of $kQ_{\mathcal{W}(p)}$ such that $kQ_{\mathcal{W}(p)}/I_{\mathcal{W}(p)} \cong \on{End}_{\mathcal{W}(p)}(\bigoplus_{L \in \on{Irr}(\mathcal{W}(p))}P_L)^{op}$. 

Set $\mathcal{E}(\mathcal{W}(p))=\on{End}_{\mathcal{W}(p)}(\bigoplus_{L \in \on{Irr}(\mathcal{W}(p))}P_L)$. Based on the socle sequences given in Section~\ref{QuiverWp} and Lemma~\ref{LemmaHom}, we see that if $1 \leq s \leq p-1$, then $\on{dim}_\mathbb{C} \ \text{Hom}_{\mathcal{W}(p)}(P_s^{\epsilon},P_s^{\epsilon'})=2$ for $\epsilon, \epsilon'=\pm$. Furthermore, $\on{dim}_\mathbb{C} \ \text{Hom}_{\mathcal{W}(p)}(P_p^\epsilon,P_p^\epsilon)=1$ for $\epsilon=\pm$, and all the other summands in $\mathcal{E}(\mathcal{W}(p))$ are zero. Hence: 
\begin{align*}
\begin{array}{l}
\mathcal{E}(\mathcal{W}(p)) \cong  \displaystyle  \bigoplus_{s=1}^{p-1} \left( \on{Hom}_{\mathcal{W}(p)}(P_s^+,P_s^+) \oplus \on{Hom}_{\mathcal{W}(p)}(P_s^-,P_s^-) \right. \\
\quad \quad \quad \left. \oplus \on{Hom}_{\mathcal{W}(p)}(P_s^+,P_s^-) \oplus \on{Hom}_{\mathcal{W}(p)}(P_s^-,P_s^+) \right)  \oplus \cc \oplus \cc. 
 \end{array}
\end{align*}
In particular, $\on{dim}_\mathbb{C} \ \mathcal{E}(\mathcal{W}(p)) =8p-6$.

For $1 \leq s \leq p-1$, we label the arrows of the diagrams of $P_s^\pm$ as follows (see \cite{GSTF}):
\[
 P_s^+: \quad \xymatrix{ & X^{+}_s \ar[rd]^{V_{s, 2}^-} \ar[ld]_{V_{s, 1}^-} & \\
 X^-_s\ar[rd]_{V_{s, 2}^+}&& X^{-}_s\ar[ld]^{V_{s, 1}^+}\\
 &X^{+}_s &} \hspace{1cm} P_s^-: \quad \xymatrix{ & X^{-}_s \ar[rd]^{V_{s, 2}^+}\ar[ld]_{V_{s, 1}^+}& \\
 X^+_s\ar[rd]_{V_{s, 2}^-}&& X^{+}_s\ar[ld]^{V_{s, 1}^-}\\
 &X^{-}_s &.}\]
An arrows $X^\mp_s \stackrel{V_{s, i}^\pm}{\longrightarrow} X^\pm_s$ means that there exists a short exact sequence $0 \longrightarrow X^\pm_s \longrightarrow V_{s, i}^\pm \longrightarrow X^\mp_s \longrightarrow 0$.

Define $a_{s, 1} \in \text{Hom}_{\mathcal{W}(p)}(P_s^+,P_s^-)$ as the composition $P_s^+ \relbar\joinrel\twoheadrightarrow V_{s, 1}^- \myhookrightarrow P_s^-$. Likewise, $a_{s, 2}$ is the composition $P_s^+ \relbar\joinrel\twoheadrightarrow V_{s, 2}^- \myhookrightarrow P_s^-$. As $\text{Hom}_{\mathcal{W}(p)}(P_s^+,P_s^-)$ is of dimension $2$, we see that $(a_{s, 1}, a_{s, 2})$ is a basis of $\text{Hom}_{\mathcal{W}(p)}(P_s^+,P_s^-)$.

Similarly, we define $b_{s, 1} \in \text{Hom}_{\mathcal{W}(p)}(P_s^-,P_s^+)$ as the composition $P_s^- \relbar\joinrel\twoheadrightarrow V_{s, 1}^+ \myhookrightarrow P_s^+$ and $b_{s, 2}:P_s^- \relbar\joinrel\twoheadrightarrow V_{s, 2}^+  \myhookrightarrow P_s^+$. We see that $(b_{s, 1}, b_{s, 2})$ is a basis of $\text{Hom}_{\mathcal{W}(p)}(P_s^-,P_s^+)$.

By looking at the composition homomorphisms between the basis elements, we see that:
\begin{equation*}
\begin{cases}
b_{s, 1} \circ a_{s, 1}=b_{s, 2} \circ a_{s, 2}=b_{s, 2} \circ a_{s, 1}-b_{s, 1} \circ a_{s, 2}=0, \\[5pt]
a_{s, 1} \circ b_{s, 1}=a_{s, 2} \circ b_{s, 2}=a_{s, 1} \circ b_{s, 2}-a_{s, 2} \circ b_{s, 1} =0,  \\[5pt]
a_{s, i} \circ b_{s, j} \circ a_{s, k}=0 \ \forall \ i, j, k, \\[5pt]
b_{s, i} \circ a_{s, j} \circ b_{s, k} =0 \ \forall \ i, j, k.
\end{cases}
\end{equation*}

Let $Q_{\mathcal{W}(p)}^{op}$ be the opposite quiver of $Q_{\mathcal{W}(p)}$. Define the algebra homomorphism $\pi : kQ_{\mathcal{W}(p)}^{op} \longrightarrow \mathcal{E}(\mathcal{W}(p))$ such that
\begin{equation*}
\begin{cases}
\pi(\overline{e_s^\pm})=\on{id}_{P_s^\pm} \text{ for } 1 \leq s \leq p, \\[5pt]
\pi(\overline{\alpha_{s, i}})=a_{s,i} \text{ for } 1 \leq s \leq p-1, \ i=1,2,   \\[5pt]
\pi(\overline{\beta_{s, i}})=b_{s,i} \text{ for } 1 \leq s \leq p-1, \ i=1,2.
\end{cases}       
\end{equation*}
We have $\pi(\overline{\alpha_{s, i}\beta_{s, j}})=\pi(\overline{\beta_{s, j}}\overline{\alpha_{s, i}})=\pi(\overline{\beta_{s, j}})\circ \pi(\overline{\alpha_{s, i}})=b_{s, j} \circ a_{s, i}$. We see that $b_{s, 2} \circ a_{s, 1} \in \text{Hom}_{\mathcal{W}(p)}(P_s^+,P_s^+)$ is neither $0$ nor a multiple of the identity, so $(\on{id}_{P_s^+},b_{s, 2} \circ a_{s, 1})$ is a basis of $\text{Hom}_{\mathcal{W}(p)}(P_s^+,P_s^+)$. Likewise, we have $(\on{id}_{P_s^-},a_{s, 2}\circ b_{s, 1})$ is a basis of $\text{Hom}_{\mathcal{W}(p)}(P_s^-,P_s^-)$. It follows that $\pi$ is surjective and its kernel contains the relations above.

Let $I_{\mathcal{W}(p)}^{op}$ be the ideal of $kQ_{\mathcal{W}(p)}^{op}$ generated by:
\begin{equation}\label{eq:pres_E(W(p))}
\begin{cases}
\overline{\alpha_{s, 1}\beta_{s, 2}}-\overline{\alpha_{s, 2}\beta_{s, 1}} \text{ for } 1 \leq s \leq p-1, \\[5pt]
\overline{\alpha_{s, i}\beta_{s, i}} \text{ for } 1 \leq s \leq p-1, \ i=1, 2, \\[5pt]
\overline{\beta_{s, 1}\alpha_{s, 2}}-\overline{\beta_{s, 2}\alpha_{s, 1}} \text{ for } 1 \leq s \leq p-1, \\[5pt]
\overline{\beta_{s, i}\alpha_{s, i}} \text{ for } 1 \leq s \leq p-1, \ i=1, 2. 
\end{cases} 
\end{equation}
We can check that $I_{\mathcal{W}(p)}^{op}$ contains $\overline{\alpha_{s, i}\beta_{s, j}\alpha_{s, k}}$ and $\overline{\beta_{s, i}\alpha_{s, j}\beta_{s, k}}$ for all $i, j, k$. We have $I_{\mathcal{W}(p)}^{op} \subseteq \text{Ker }\pi$ so $\pi$ factors through the surjection $kQ_{\mathcal{W}(p)}^{op}/I_{\mathcal{W}(p)}^{op} \longrightarrow \mathcal{E}(\mathcal{W}(p))$. By counting the number of independent elements in $kQ_{\mathcal{W}(p)}^{op}/I_{\mathcal{W}(p)}^{op} $ degree by degree, we can determine that $\on{dim}_\mathbb{C} \ kQ_{\mathcal{W}(p)}^{op}/I_{\mathcal{W}(p)}^{op} =2p+4(p-1)+2(p-1)=8p-6=\on{dim}_\mathbb{C} \ \mathcal{E}(\mathcal{W}(p))$. It follows that $\pi$ is injective and so it is an isomorphism of algebras.

\delete{

We will write $(a_{s,1},a_{s,2})$ for a basis of $\text{Hom}_{\mathcal{W}(p)}(P_s^+,P_s^-)$ and $(b_{s,1},b_{s,2})$ for a basis of $\text{Hom}_{\mathcal{W}(p)}(P_s^-,P_s^+)$.

Let $\overline{Q_{\mathcal{W}(p)}}$ be the opposite quiver $Q_{\mathcal{W}(p)}$. Define the algebra homomorphism $\pi : kQ_{\mathcal{W}(p)}^{op} \longrightarrow \mathcal{E}$ such that
\begin{equation*}
\begin{cases}
\pi(\overline{e_s^\pm})=\on{id}_{P_s^\pm} \text{ for } 1 \leq s \leq p, \\[5pt]
\pi(\overline{\alpha_{s, i}})=a_{s,i} \text{ for } i=1,2, \ 1 \leq s \leq p-1,  \\[5pt]
\pi(\overline{\beta_{s, i}})=b_{s,i} \text{ for } i=1,2, \ 1 \leq s \leq p-1.
\end{cases}       
\end{equation*}
We have $\pi(\overline{\alpha_{s, i}\beta_{s, j}})=\pi(\overline{\beta_{s, j}}\overline{\alpha_{s, i}})=\pi(\overline{\beta_{s, j}})\circ \pi(\overline{\alpha_{s, i}})=b_{s, j} \circ a_{s, i}$. We can illustrate $b_{s, j} \circ a_{s, i}$ as follows:
\begin{center}
 \begin{tikzpicture}[scale=1, transform shape]
\tikzset{>=stealth}
\node (1) at (0,1.5) []{$X_s^+$};
\node (2) at (-1,0) []{$X_s^-$};
\node (3) at (1,0) []{$X_s^-$};
\node (4) at (0,-1.5) []{$X_s^+$};
\draw[->]  (1) -- (2);
\draw[->]  (1) -- (3);
\draw[->]  (2) -- (4);
\draw[->]  (3) -- (4) ;

\node (5) at (4,1.5) []{$X_s^-$};
\node (6) at (3,0) []{$X_s^+$};
\node (7) at (5,0) []{$X_s^+$};
\node (8) at (4,-1.5) []{$X_s^-$};
\draw[->]  (5) -- (6);
\draw[->]  (5) -- (7);
\draw[->]  (6) -- (8);
\draw[->]  (7) -- (8) ;

\node (9) at (8,1.5) []{$X_s^+$};
\node (10) at (7,0) []{$X_s^-$};
\node (11) at (9,0) []{$X_s^-$};
\node (12) at (8,-1.5) []{$X_s^+$};
\draw[->]  (9) -- (10);
\draw[->]  (9) -- (11);
\draw[->]  (10) -- (12);
\draw[->]  (11) -- (12) ;

\draw[->]  (1) -- (6);
\draw[->]  (1) -- (7);
\draw[->]  (2) -- (8);
\draw[->]  (3) -- (8);

\draw[->]  (5) -- (10);
\draw[->]  (5) -- (11);
\draw[->]  (6) -- (12);
\draw[->]  (7) -- (12);

\node (13) at (2,-2) []{$a_{s, i}$};
\node (14) at (6,-2) []{$b_{s, j}$};

\end{tikzpicture}
\end{center}
We can construct  $a_{s, i}$ and $b_{s, j}$ coordinate by coordinate such that, for all $1 \leq s \leq p-1$:
\begin{equation*}
\begin{cases}
b_{s, 1} \circ a_{s, 1}=b_{s, 2} \circ a_{s, 2}=b_{s, 1} \circ a_{s, 2}=\on{id}_{X_s^+}=-b_{s, 2} \circ a_{s, 1}, \\[5pt]
a_{s, 1}\circ b_{s, 1} =a_{s, 2}\circ b_{s, 2} =a_{s, 1}\circ b_{s, 2} =\on{id}_{X_s^-}=-a_{s, 2}\circ b_{s, 1}.
\end{cases}
\end{equation*}
With this, we have $b_{s, j} \circ a_{s, i} \in \text{Hom}_{\mathcal{W}(p)}(P_s^+,P_s^+)$ and $(\on{id}_{P_s^+},b_{s, j} \circ a_{s, i})$ is a free family in a space of dimension 2. Hence $(\on{id}_{P_s^+},b_{s, j} \circ a_{s, i})$ is a basis of $\text{Hom}_{\mathcal{W}(p)}(P_s^+,P_s^+)$. Likewise we have $(\on{id}_{P_s^-},a_{s, i}\circ b_{s, j})$ is a basis of $\text{Hom}_{\mathcal{W}(p)}(P_s^-,P_s^-)$. It follows that $\pi$ is surjective and its kernel contains the relations above.

From the previous diagram, we see that $ a_{s, k} \circ b_{s, j} \circ a_{s, i}=0$ and $b_{s, k} \circ a_{s, j}\circ b_{s, i}=0$ for any $i, j, k$. This implies that $\overline{\alpha_{s, i}\beta_{s, j}\alpha_{s, k}}$ and $\overline{\beta_{s, i}\alpha_{s, j}\beta_{s, k}}$ are also in the kernel of $\pi$ for any $i, j, k$. 

Let $I_{\mathcal{W}(p)}^{op}$ be the ideal of $kQ_{\mathcal{W}(p)}^{op}$ generated by:
\begin{equation*}
\begin{cases}
\overline{\alpha_{s, 1}\beta_{s, 2}}+\overline{\alpha_{s, 2}\beta_{s, 1}} \text{ for all } 1 \leq s \leq p-1, \\[5pt]
\overline{\alpha_{s, 1}\beta_{s, 1}}-\overline{\alpha_{s, 2}\beta_{s, 2}} \text{ for all } 1 \leq s \leq p-1, \\[5pt]
\overline{\alpha_{s, 2}\beta_{s, 2}}-\overline{\alpha_{s, 2}\beta_{s, 1}} \text{ for all } 1 \leq s \leq p-1, \\[5pt]
\overline{\beta_{s, 1}\alpha_{s, 2}}+\overline{\beta_{s, 2}\alpha_{s, 1}} \text{ for all } 1 \leq s \leq p-1, \\[5pt]
\overline{\beta_{s, 1}\alpha_{s, 1}}-\overline{\beta_{s, 2}\alpha_{s, 2}} \text{ for all } 1 \leq s \leq p-1, \\[5pt]
\overline{\beta_{s, 2}\alpha_{s, 2}}-\overline{\beta_{s, 2}\alpha_{s, 1}} \text{ for all } 1 \leq s \leq p-1, \\[5pt]
\overline{\alpha_{s, i}\beta_{s, j}\alpha_{s, k}}\text{ for some } i, j, k \in \{1,2\} \text{ (one triple is enough)}, \\[5pt]
\overline{\beta_{s, i}\alpha_{s, j}\beta_{s, k}}\text{ for some } i, j, k, \in \{1,2\} \text{ (one triple is enough)}. 
\end{cases} 
\end{equation*}
We have $I_{\mathcal{W}(p)}^{op} \subseteq \text{Ker }\pi$ so $\pi$ factors through the surjection $kQ_{\mathcal{W}(p)}^{op}/I_{\mathcal{W}(p)}^{op} \longrightarrow \mathcal{E}$. By counting the number of independent elements in $kQ_{\mathcal{W}(p)}^{op}/I_{\mathcal{W}(p)}^{op} $ degree by degree, we can determine that $\on{dim}_\mathbb{C} \ kQ_{\mathcal{W}(p)}^{op}/I_{\mathcal{W}(p)}^{op} =2p+4(p-1)+2(p-1)=8p-6=\on{dim}_\mathbb{C} \ \mathcal{E}$. It follows that $\pi$ is injective and so it is an isomorphism of algebras.
}

We have proved the following result.

\begin{proposition}\label{prop:E(W(p))}
There is an isomorphism of algebras
\begin{align*}
\mathcal{E}(\mathcal{W}(p)) \cong kQ_{\mathcal{W}(p)}^{op}/I_{\mathcal{W}(p)}^{op} ,
\end{align*}
where $I_{\mathcal{W}(p)}^{op}$ is given by the relations~\eqref{eq:pres_E(W(p))}. In particular, $\mathcal{E}(\mathcal{W}(p))$ is a quadratic algebra and $\mathcal{E}(\mathcal{W}(p))^!=kQ_{\mathcal{W}(p)}/(I_{\mathcal{W}(p)}^{op})^\perp$.
\end{proposition}

Through the presentation of  $\mathcal{E}(\mathcal{W}(p))$ in Proposition~\ref{prop:E(W(p))}, we can compute $(I_{\mathcal{W}(p)}^{op})^\perp$ and show that $(I_{\mathcal{W}(p)}^{op})^\perp=\mathcal{I}_{\mathcal{W}(p)}$. Using the isomorphism~\eqref{eq:Yoneda}, we obtain:

\begin{theorem}\label{Koszul_dual_W(p)}
The algebra $\mathcal{E}(\mathcal{W}(p))$ is Koszul and there exists an isomorphism of graded algebras
\[
\on{Ext}^*_{\mathcal{W}(p)}(X_{\mathcal{W}(p)}, X_{\mathcal{W}(p)}) \cong \mathcal{E}(\mathcal{W}(p))^!.
\]

The category of logarithmic $\mathcal{W}(p)$-modules (equivalently the category of $\overline{U}_q(\on{sl}_2)$-modules) is Koszul.
\end{theorem}

\begin{remark}
We see that there is an algebra isomorphism $\mathcal{E}(\mathcal{W}(p)) \cong \mathcal{E}(\mathcal{W}(p))^{op}$ given by $\alpha_{i, s} \longmapsto \overline{\beta_{i, s}}$ and $\beta_{i, s} \longmapsto \overline{\alpha_{i, s}}$. 
\end{remark}

\delete{
We know that the ideal $\mathcal{I}_{\mathcal{W}(p)}$ defined in Section~\ref{Section6.1} is generated by the relations $\alpha_{s,1}\beta_{s,2}+\alpha_{s,2} \beta_{s,1}$, $\beta_{s,1}\alpha_{s,2}+\beta_{s,2}\alpha_{s,1}$ for $1 \leq s \leq p-1$. Based on the relations defining $I_{\mathcal{W}(p)}$, we see that $\mathcal{I}_{\mathcal{W}(p)} \subseteq I_{\mathcal{W}(p)}$. We have thus proved:

\begin{theorem}
There exist ideals $\mathcal{I}_{\mathcal{W}(p)}, I_{\mathcal{W}(p)} \subseteq kQ_{\mathcal{W}(p)}$ such that $\mathcal{I}_{\mathcal{W}(p)} \subseteq I_{\mathcal{W}(p)}$ and they induce the following surjective algebra homomorphisms:
\begin{equation*}
kQ_{\mathcal{W}(p)}  \  \raisebox{2pt}{\begin{tikzpicture}[scale=1, transform shape] \tikzset{>=stealth} \draw[->]  (0,0) -- (0.5,0); \draw[->]  (0,0) -- (0.4,0); \end{tikzpicture}} \ kQ_{\mathcal{W}(p)}/\mathcal{I}_{\mathcal{W}(p)} \cong \on{Ext}_{\mathcal{W}(p)}^*(X_{\mathcal{W}(p)}, X_{\mathcal{W}(p)}) \  \raisebox{2pt}{\begin{tikzpicture}[scale=1, transform shape] \tikzset{>=stealth} \draw[->]  (0,0) -- (0.5,0); \draw[->]  (0,0) -- (0.4,0); \end{tikzpicture}} \ kQ_{\mathcal{W}(p)}/I_{\mathcal{W}(p)} \cong \on{End}_{\mathcal{W}(p)}(\bigoplus_{L \in \on{Irr}(\mathcal{W}(p))}P_L)^{op}. 
\end{equation*}
\end{theorem}

\begin{remark}
It is proved in \cite{Nagatomo-Tsuchiya} that the category of logarithmic $\mathcal{W}(p)$-modules is equivalent as an abelian category to the category of finite dimensional modules for the restricted quantum group $\overline{U}_q(\on{sl}_2)$ at $q=e^{\frac{\pi i}{p}}$. Therefore their Yoneda algebras are isomorphic.  The Yoneda algebras of $\overline{U}_q(\on{sl}_2)$ are determined in \cite{GSTF} which also determines the indecomposable modules of the restricted quantum group. The abelian category of finite dimensional  $\overline{U}_q(\on{sl}_2)$-modules was also investigated earlier in \cite{Suter}. The study of some infinite dimensional modules over $\overline{U}_q(\on{sl}_2)$ was carried out in \cite{Xiao} earlier too. As conjectured in \cite{GSTF}, the two categories should be equivalent as ribbon braided monoidal categories. The Grothendieck ring of $\overline{U}_q(\on{sl}_2)$ has been determined in \cite{FGST}. This conjecture has been proved by a sequence of papers by various authors (see a recent preprint \cite[Section 1]{FL} for a list of references).   As our goal is to compare the representation categories of the triplet vertex operator algebra and those of its $C_2$-algebra and Zhu algebra, it is natural to obtain the description of the Ext quiver and Yoneda algebra of $\mathcal{W}(p)$ while remaining in the vertex algebra context. Thus the proofs given above do not rely on the category equivalence mentioned at the beginning of the remark.
\end{remark}
}

\subsection{The Ext algebra and endomorphism algebra of $A(\mathcal{W}(p))$} 
In this section, we compute the Ext algebra of $A(\mathcal{W}(p))$ as well as the endomorphism algebra of the sum of the projective covers within the category of $A(\mathcal{W}(p))$-modules without considering the filtered structure of $A(\mathcal{W}(p))$. Our aim is then to do a comparison with what we obtained in the previous sections. We know that $A(\mathcal{W}(p))$ is a direct sum of ideals given in Section~\ref{SectionZhuAlgebraA(W(p))}. Each summand has a unique irreducible module: 
\begin{itemize}
\item for $1 \leq i \leq p-1$, $X_i^+(0)$ is the irreducible module of $I_{h_i}(\cc) \cong \cc[x]/(x^2)$.
\item  for $2p+1 \leq i \leq 3p-1$, $X_{i-2p}^-(0)$ is the irreducible module of $M_{h_i}(\cc) \cong M_2(\cc)$.  The irreducible module of $M_{h_{2p}}(\cc)$ is $X_p^-(0)$.
\item  $\cc$ has $X_p^+(0)$ as its only irreducible module.
\end{itemize}
Therefore the sum $X_A$ of the simple $A(\mathcal{W}(p))$-modules is given by:
\begin{equation*}
X_A = \bigoplus_{i=2p+1}^{3p-1}X_{i-2p}^-(0) \oplus \bigoplus_{i=1}^{p-1} X_i^+(0) \oplus X_p^+(0) \oplus X_p^-(0).
\end{equation*}
We want to determine the Ext algebra $\on{Ext}_{A(\mathcal{W}(p))}^*(X_A, X_A)$ of $A(\mathcal{W}(p))$. Because $A(\mathcal{W}(p))$ is a direct sum of algebras, there are no non-trivial extensions between blocks corresponding to distinct summands. Therefore the Ext algebra decomposes as:
\begin{equation*}
\resizebox{\hsize}{!}{$
\begin{split}
\on{Ext}_{A(\mathcal{W}(p))}^*(X_A, X_A)= & \bigoplus_{i=2p+1}^{3p-1}\on{Ext}_{A(\mathcal{W}(p))}^*(X_{i-2p}^-(0), X_{i-2p}^-(0))\oplus \bigoplus_{i=1}^{p-1}\on{Ext}_{A(\mathcal{W}(p))}^*(X_i^+(0), X_i^+(0)) \\
& \oplus \on{Ext}_{A(\mathcal{W}(p))}^*(X_p^+(0), X_p^+(0)) \oplus \on{Ext}_{A(\mathcal{W}(p))}^*(X_p^-(0), X_p^-(0)).
\end{split}$}
\end{equation*}

For $2p+1 \leq i \leq 3p-1$, $M_{h_i}(\cc) \cong M_2(\cc)$ is semisimple so $X_{i-2p}^-(0) \cong \cc^2$ is a projective $M_{h_i}(\cc)$-module. Thus:
\[
\begin{array}{cl}
\on{Ext}_{A(\mathcal{W}(p))}^*(X_{i-2p}^-(0), X_{i-2p}^-(0))& =\on{Ext}_{M_{h_i}(\cc)}^*(X_{i-2p}^-(0), X_{i-2p}^-(0)), \\
& =\text{Hom}_{M_{h_i}(\cc)}(X_{i-2p}^-(0), X_{i-2p}^-(0)), \\
&  \cong \cc.
\end{array}
\]
Similarly we get $\on{Ext}_{A(\mathcal{W}(p))}^*(X_{p}^-(0), X_{p}^-(0)) \cong \cc$ and $\on{Ext}_{A(\mathcal{W}(p))}^*(X_p^+(0), X_p^+(0)) \cong \cc$.

Finally, using the following exact sequence:
\begin{align*}
\dots \longrightarrow \cc[x]/(x^2) \xrightarrow{\times x} \cc[x]/(x^2) \xrightarrow{\times x} \cc[x]/(x^2) \xrightarrow{\times x} \cc[x]/(x^2) \xrightarrow{\times x} (x) \longrightarrow 0,
\end{align*} 
we prove that for any for $1 \leq i \leq p-1$, the Yoneda algebra $\on{Ext}_{A(\mathcal{W}(p))}^*(X_i^+(0), X_i^+(0))=\on{Ext}_{\cc[x]/(x^2)}^*((x),(x))$ is isomorphic to $\cc[\alpha]$ with $\alpha$ a basis of $\on{Ext}_{\cc[x]/(x^2)}^1((x),(x))$.

\delete{Consider
It is a minimal projective resolution of the $\cc[x]/(x^2)$-module $(x)$, and thus for $n \geq 1$:
\[
\on{Ext}_{\cc[x]/(x^2)}^n((x),(x))  \cong \text{Hom}_{\cc[x]/(x^2)}(\cc[x]/(x^2),(x)) \cong \cc.
\]
We also have $\on{Ext}_{\cc[x]/(x^2)}^0((x),(x)) \cong \cc$.

Let $\alpha$ be a basis of $\on{Ext}_{\cc[x]/(x^2)}^1((x),(x))$. Then $\alpha$ is represented by a non-trivial $\cc[x]/(x^2)$-module homomorphism $\alpha : \cc[x]/(x^2) \longrightarrow (x)$. We have the following diagram
\begin{center}
 \begin{tikzpicture}[scale=1, transform shape]
\tikzset{>=stealth}
\node (1) at (2,0) []{$\dots$};
\node (2) at (4,0) []{$\cc[x]/(x^2)$};
\node (3) at (7,0) []{$\cc[x]/(x^2)$};
\draw[->]  (1) -- (2);
\draw[->]  (2) -- (3)  node[midway, above] {$\times x$};

\node (4) at (2,-2) []{$\dots$};
\node (5) at (4,-2) []{$\cc[x]/(x^2)$};
\node (6) at (7,-2) []{$\cc[x]/(x^2)$};
\node (7) at (10,-2) []{$(x)$};
\node (8) at (12,-2) []{$0$};
\draw[->]  (4) -- (5);
\draw[->]  (5) -- (6) node[midway, below] {$\times x$};
\draw[->]  (6) -- (7) node[midway, below] {$\times x$};
\draw[->]  (7) -- (8) ;

\draw[->,dashed]  (2) -- (5) node[midway, right] {$\times c$};
\draw[->,dashed]  (3) -- (6) node[midway, right] {$\times c$};
\draw[->]  (3) -- (7) node[midway, above right] {$\alpha$};

\node (9) at (4,-4) []{$(x)$};
\draw[->]  (5) -- (9) node[midway, right] {$\alpha$};
\end{tikzpicture}
\end{center}
All $\cc[x]/(x^2)$-module homomorphisms $\cc[x]/(x^2) \longrightarrow (x)$ are proportional and $\alpha$ is non-trivial so there exists $c \in \cc^*$ such that $\alpha=\times (cx)$. Hence the diagram commutes and $\alpha^2$ is represented by $c\alpha:  \cc[x]/(x^2) \longrightarrow (x)$. It is a non-trivial homomorphism and $\text{dim Ext}_{\cc[x]/(x^2)}^2((x),(x))=1$ so it is a basis of this space. By repeating this reasoning we see that $\alpha^n$ is a basis of $\on{Ext}_{\cc[x]/(x^2)}^n((x),(x))$ for all $n \geq 0$. We therefore have $\on{Ext}_{\cc[x]/(x^2)}^*((x),(x))=\cc[\alpha]$.}

We saw that the blocks of $\on{Ext}_{A(\mathcal{W}(p))}^*(X_A, X_A)$ are generated by the elements of degree $0$ and $1$ and there is no relations within blocks. Based on the quiver $kQ_{A(\mathcal{W}(p))}$, we see that the homomorphism of Lemma~\ref{kQandExt} is an isomorphism. Let $\gamma_i$ be the label of the loop on the vertex $X_i^+(0)$ in $Q_{A(\mathcal{W}(p))}$ for $1 \leq i \leq p-1$. We have found the following result:
\begin{proposition}
The Ext algebra of $A(\mathcal{W}(p))$ decomposes as:
\begin{equation*}
\on{Ext}_{A(\mathcal{W}(p))}^*(X_A, X_A) \cong  \bigoplus_{i=2p+1}^{3p-1}\cc \oplus \bigoplus_{i=1}^{p-1}\cc[\gamma_i] \oplus \cc \oplus \cc = kQ_{A(\mathcal{W}(p))}.
\end{equation*}
\end{proposition}

\begin{corollary}
For any $1 \leq i \leq p-1$, we have $\on{Ext}_{A(\mathcal{W}(p))}^*(X_i^+(0), X_i^+(0))=\cc[\gamma_i]$. Therefore $\text{proj.dim }X_i^{+}(0)=+\infty$, and the global dimension of $A(\mathcal{W}(p))$ is infinite.
\end{corollary}

Using the fact that the algebra $\on{Ext}_{A(\mathcal{W}(p))}^*(X_A, X_A) $ decomposes as a direct sum of polynomial rings and fields, we obtain the following corollary.

\begin{corollary}
The Gelfand-Kirillov dimension of the Yoneda algebra $\on{Ext}_{A(\mathcal{W}(p))}^{*}(X_A, X_A)$ is $1$.
\end{corollary}

The projective covers of the simple $A(\mathcal{W}(p))$-modules are described as follows:
\begin{itemize}
\item for $1 \leq i \leq p-1$, $\cc[x]/(x^2)$ is a projective $\cc[x]/(x^2)$-module and $\text{Rad}(\cc[x]/(x^2))=(x)$ is superfluous, so $\cc[x]/(x^2) \longrightarrow X_i^+(0)$ is the projective cover of $X_i^+(0)$.
\item  for $2p+1 \leq i \leq 3p-1$, $X_{i-2p}^-(0)$ is an irreducible module of the semisimple algebra $M_{h_i}(\cc)$ so it is projective and therefore its own projective cover. Likewise $X_p^-(0)$ is a projective $M_{h_{2p}}(\cc)$-module.
\item  $X_p^+(0)$ is a projective $\cc$-module.
\end{itemize}
It follows that $\bigoplus_{L \in \on{Irr}(A(\mathcal{W}(p)))}P_L \cong  \bigoplus_{i=2p}^{3p-1}\cc^2 \oplus \bigoplus_{i=1}^{p-1}\cc[x]/(x^2) \oplus \cc$. Therefore the endomorphism algebra $\mathcal{E}(A(\mathcal{W}(p)))=\on{End}_{A(\mathcal{W}(p))}(\bigoplus_{L \in \on{Irr}(A(\mathcal{W}(p)))}P_L)$ is given by:

\begin{equation*}
\begin{split}
\mathcal{E}(A(\mathcal{W}(p))) & \cong  \bigoplus_{i=2p}^{3p-1}\on{End}_{A(\mathcal{W}(p))}(\cc^2) \oplus \bigoplus_{i=1}^{p-1}\on{End}_{A(\mathcal{W}(p))}(\cc[x]/(x^2)) \oplus \on{End}_{A(\mathcal{W}(p))}(\cc), \\
&  \cong  \bigoplus_{i=2p}^{3p-1}\cc \oplus \bigoplus_{i=1}^{p-1}\cc[\gamma_i]/(\gamma_i^2) \oplus \cc, \\
&  \cong  kQ_{A(\mathcal{W}(p))}/ \langle \gamma_1^2,\dots,\gamma_{p-1}^2 \rangle.
\end{split}
\end{equation*}

With a reasoning similar to that of Theorem~\ref{Koszul_dual_W(p)} :

\begin{theorem}\label{Koszul_dual_A(W(p))}
The algebra $\mathcal{E}(A(\mathcal{W}(p)))$ is Koszul and there exists an isomorphism of graded algebras
\[
\on{Ext}^*_{A(\mathcal{W}(p))}(X_A, X_A) \cong \mathcal{E}(A(\mathcal{W}(p)))^!.
\]

The category of $A(\mathcal{W}(p))$-modules is Koszul.
\end{theorem}

\delete{
\begin{theorem}
For the Zhu algebra $A(\mathcal{W}(p))$, we have the following homomorphisms:
\begin{equation*}
kQ_{A(\mathcal{W}(p))} \overset{\cong}{\longrightarrow}  \on{Ext}_{A(\mathcal{W}(p))}^*(X_A, X_A)  \  \raisebox{2pt}{\begin{tikzpicture}[scale=1, transform shape] \tikzset{>=stealth} \draw[->]  (0.4,0) -- (1,0); \draw[->]  (0.4,0) -- (0.9,0); \end{tikzpicture}} \ kQ_{A(\mathcal{W}(p))}/ \langle \gamma_1^2,\dots,\gamma_{p-1}^2 \rangle \cong \on{End}_{A(\mathcal{W}(p))}(\mathcal{P}_A). 
\end{equation*}
\end{theorem}

We do not need to add the subscript "op" to the endomorphism algebra because it is commutative. In particular, we see that $\on{End}_{A(\mathcal{W}(p))}(\mathcal{P}_A)$ is generated in degree $0$ and $1$ with relations in degree $2$.
}

\subsection{The Ext algebra and endomorphism algebra of $\on{gr}A(\mathcal{W}(p))$} 
In this section, we set ourselves in the category of $\on{gr}A(\mathcal{W}(p))$-modules without considering the graded structure of $\on{gr}A(\mathcal{W}(p))$. As $\on{gr}A(\mathcal{W}(p))$ is commutative and local, the only simple module is $X=\cc$ and its projective cover is $P_X=\on{gr}A(\mathcal{W}(p))$. It follows that $\bigoplus_{L \in \on{Irr}(\on{gr}A(\mathcal{W}(p)))}P_L =P_X = \on{gr}A(\mathcal{W}(p)) $. Therefore the algebra $\mathcal{E}(\on{gr} A(\mathcal{W}(p)))=\on{End}_{\on{gr}A(\mathcal{W}(p))}(\bigoplus_{L \in \on{Irr}(\on{gr}A(\mathcal{W}(p)))}P_L)$ is given by:
\begin{equation*}
\mathcal{E}(\on{gr} A(\mathcal{W}(p)))  =  \on{End}_{\on{gr}A(\mathcal{W}(p))}(\on{gr}A(\mathcal{W}(p))) \cong \on{gr}A(\mathcal{W}(p)),
\end{equation*}
because the graded Zhu algebra is commutative.

From what we saw, the sum $X_{\on{gr}A}$ of the simple $\on{gr}A(\mathcal{W}(p))$-modules is given by $X_{\on{gr}A} = \cc$. The computation of the Ext algebra $\on{Ext}_{\on{gr}A(\mathcal{W}(p))}^*(X_{\on{gr}A}, X_{\on{gr}A})$ is not practical because the algebra $\on{gr}A(\mathcal{W}(p))$ is not a complete intersection. Therefore when constructing the minimal resolution, there is no way to determine after how many steps the complex becomes exact. The only way would be to compute until we reach an exact complex, which is not reasonable. Instead of computing directly the Ext algebra of $\on{gr}A(\mathcal{W}(p))$, we will construct a complete intersection approximation of $\on{gr}A(\mathcal{W}(p))$, which will then enable us to determine a quotient of the Ext algebra we are looking for.

We can see from Section~\ref{SectiongrA(W(p))} that we can write $R= \on{gr}A(\mathcal{W}(p)) = \cc[E,F,H,\omega]/(r_1, \dots, r_{11})$ where
\begin{align*}
\renewcommand\arraystretch{1.2}
\left\{ \begin{array}{llll}
r_1=E^2, & r_2=F^2, & r_3=H^3, & r_4=\omega^{3p-1},\\
r_5= \omega^pE, & r_6=\omega^pF, & r_7=\omega^pH, & r_8=H^2-C_p\omega^{2p-1},\\
r_9=EF+C_p\omega^{2p-1}, & r_{10}=EH, & r_{11}=FH. &
\end{array} \right.
\end{align*}
with $C_p=\frac{(4p)^{2p-1}}{(2p-1)!^2}$. For all $1 \leq j \leq 11$, we write $r_j=\sum_{i=1}^4 c_{j, i} x_i$ with $x_1=E, x_2=F, x_3=H, x_4=\omega$, and $c_{j, i} \in \cc[E,F,H,\omega]$. Furthermore, for an algebra $A$ with trivial module $\cc$, we will write $H^*(A)$ for the object $\on{Ext}^{*}_{A}(\cc,\cc)$. Define $R'=\cc[E,F,H,\omega]/(r_1, \dots, r_{8})$. The algebra $R$ is isomorphic to $R'/(r_9,r_{10},r_{11})$.  

The sequence $(r_1,r_2,r_3,r_4)$ is regular in $\cc[E,F,H,\omega]$. We add four variables $x_5, x_6, x_7, x_8$ to $\cc[E,F,H,\omega]$ and define the relations:
\begin{align*}
\renewcommand\arraystretch{1.2}
\left\{\begin{array}{ll}
r_1=E^2, & \widetilde{r}_5= \omega^pE-x_5^2,\\
r_2=F^2, &  \widetilde{r}_6=\omega^pF -x_6^2, \\
r_3=H^3, &  \widetilde{r}_7=\omega^pH -x_7^2, \\
r_4=\omega^{3p-1}, & \widetilde{r}_8=H^2-C_p\omega^{2p-1}-x_8^2, 
\end{array} \right.
\end{align*}
so that $(r_1, \dots, \widetilde{r}_8)$ is a regular sequence in $\cc[E,F,H,\omega, x_5, x_6, x_7, x_8]$. It follows that:
\begin{align*}
\widetilde{R}=\cc[E,F,H,\omega, x_5, x_6, x_7, x_8]/(r_1, \dots, \widetilde{r}_8)
\end{align*}
is a complete intersection ring.

The composition $\widetilde{R} \longrightarrow R' \longrightarrow R$ is denoted by $\pi$. It induces a homomorphism on the cohomology $\pi^{\#} :H^*(R) \longrightarrow H^*(\widetilde{R})$. We will compute $\on{Im} \pi^{\#}$, and thus obtain a quotient of $H^*(R)$.

Because $\widetilde{R}$ is a complete intersection, a result of Tate (\cite[Theorem 4]{Tate}) states that a free resolution of $\cc$ as an $\widetilde{R}$-module is given as an $\widetilde{R}$-algebra by:
\begin{align*}
Q= \widetilde{R} \langle t_1, \dots, t_{8}, s_1, \dots, s_8 \rangle,
\end{align*}
where the $t_i$'s have degree $1$, the $s_j$'s have degree $2$ and
\begin{align*}
\left\{ \renewcommand\arraystretch{1.2}\begin{array}{l}
\partial_1 t_1=E, \\
\partial_1 t_2=F,\\
\partial_1 t_3=H,\\
\partial_1 t_4=\omega,\\
\partial_1 t_5=x_5,\\
\partial_1 t_6=x_6,\\
\partial_1 t_7=x_7,\\
\partial_1 t_8=x_8,
\end{array} \right. \hspace{2cm} \left\{ \renewcommand\arraystretch{1.2}\begin{array}{l}
\partial_2 s_1=E t_1,\\
\partial_2 s_2=F t_2,\\
\partial_2 s_3= H^2 t_3,\\
\partial_2 s_4= \omega^{3p-2} t_4,\\
\partial_2 s_5=\omega^p t_1-x_5 t_5,\\
\partial_2 s_6=\omega^p t_2-x_6 t_6,\\
\partial_2 s_7=\omega^p t_3-x_7 t_7,\\
\partial_2 s_8= H t_3-C_p\omega^{2p-2} t_4-x_8 t_8.
\end{array} \right.
\end{align*}

We call $Q_i$ the subspace of $Q$ of degree $i$ and we see that $\partial_i :Q_i \longrightarrow (x_1, \dots , x_{8})Q_{i-1}$. Hence the induced map $\partial_i^*:  \text{Hom}_{\widetilde{R}}(Q_{i-1}, \cc)  \longrightarrow  \text{Hom}_{\widetilde{R}}(Q_i, \cc)$ is zero, and $\text{Hom}_{\widetilde{R}}(Q_i, \cc) = H^{i}(\widetilde{R})$ for all $i \geq 0$.

For all  $1 \leq i \leq 8$, we write 
\[\begin{array}[t]{cccc}
\alpha_i: & Q_1 & \longrightarrow & \cc \\
           & t_j & \longmapsto & \delta_{i, j} 
\end{array} \text{ and } \begin{array}[t]{cccc}
\beta_i: & Q_2 & \longrightarrow & \cc \\
           & s_j & \longmapsto & \delta_{i, j}, \\
           & t_j t_k & \longmapsto & 0.
\end{array}\]
By direct computation of the Yoneda products (cf. \cite[p. 64]{Carlson-Townsley-Valeri-Zhang}), we find that $\alpha_4^2=0$, $-\alpha_8^2=\alpha_3^2$, and for $1 \leq i < j \leq 8$, $\alpha_i \alpha_j=-\alpha_j \alpha_i$. We can describe a basis of $H^2(\widetilde{R})$ as the dual basis of $Q_2$. The correspondence is then:
\begin{align*}
\renewcommand\arraystretch{1.2}
\begin{array}{| l | c | c | c | c | c | c | c | c | c | c |}
\hline 
\text{Basis of } Q_2 & t_i t_j & s_1& s_2 & s_3 & s_4 & s_5 & s_6 & s_7 & s_8 \\
\hline 
\text{Basis of } H^{2}(\widetilde{R}) & \alpha_i \alpha_j & \alpha_1^2 & \alpha_2^2 & \beta_3 & \beta_4 & -\alpha_5^2 & -\alpha_6^2 & -\alpha_7^2 & -\alpha_8^2 \\
\hline
\end{array}
\end{align*}
We then determine that $\beta_3$ and $\beta_4$ commute with the $\alpha_i$'s, and that $\beta_3 \beta_4=\beta_4\beta_3$. We write below the correspondence between the basis of $Q_3$ and $H^3(\widetilde{R})$:
\begin{align*}
\resizebox{\hsize}{!}{$
\renewcommand\arraystretch{1.2}
\begin{array}{| l | c | c | c | c | c | c | c | c | c | c |}
\hline 
\text{Basis of } Q_3 & t_i t_j t_k & t_i s_1& t_i s_2 & t_i s_3 & t_i s_4 & t_i s_5 & t_i s_6 & t_i s_7 & t_i s_8 \\
\hline 
\text{Basis of } H^{3}(\widetilde{R}) & \alpha_i \alpha_j \alpha_k & \alpha_i\alpha_1^2 & \alpha_i\alpha_2^2 & \alpha_i\beta_3 & \alpha_i\beta_4 & \alpha_i(-\alpha_5^2) & \alpha_i(-\alpha_6^2) & \alpha_i(-\alpha_7^2) & \alpha_i(-\alpha_8^2) \\
\hline
\end{array}$}
\end{align*}

Next we prove the following proposition:

\begin{proposition}\label{ExtRtilde}
For any $m \in \mathbb{N}^*$, let $t_1^{a_1}\dots t_8^{a_8}s_1^{b_1}\dots s_8^{b_8} \in Q_m$. Its dual element in $H^{m}(\widetilde{R})$ is a non-trivial multiple of $\alpha_1^{a_1}\dots \alpha_8^{a_8}(\alpha_1^2)^{b_1}(\alpha_2^2)^{b_2} \beta_3^{b_3} \beta_4^{b_4}(-\alpha_5^2)^{b_5}\dots(-\alpha_8^2)^{b_8}$. This implies that:
\begin{align*}
H^{*}(\widetilde{R})=\cc[\beta_3,\beta_4] \otimes \left(\cc\langle \alpha_1, \dots, \alpha_8 \rangle /\left( \begin{array}{c} \alpha_i \alpha_j+\alpha_j \alpha_i \ \forall i \neq j, \\ \alpha_4^2, \\ \alpha_8^2+\alpha_3^2 \end{array}\right) \right).
\end{align*}
\end{proposition}

\begin{proof}
See Section~\ref{Appendix} in the appendix.
\end{proof}

We now look into the minimal free resolution of $\cc$ as an $R$-module. As we mentioned before, $R$ is not a complete intersection and thus we cannot determine the whole minimal resolution. We can however compute the first two steps, which will be sufficient for our purpose. We can verify that the ideal defining $R$ as a quotient of $\cc[E,F,H,\omega]$ is minimally generated by the homogeneous elements $r_1,r_2,r_5, \dots, r_{11}$, meaning that there is no polynomial $p \in \cc[x_1, \dots, x_4][y_1, \dots, y_8]$ without constant term such that $r_i=p(r_1,r_2, r_5, \dots, \widehat{r_i}, \dots, r_{11})$.

The $R$-algebra associated to the $R$-module $\cc$ is obtained by killing cycles (\cite[Section 2]{Tate}). The first complex $X_0$ is concentrated in degree $0$ and is $R$ itself. The maximal ideal of $R$ has 4 minimal generators, so we adjoin $4$ variables $T_1, \dots, T_4$ of degree $1$ such that $\partial_1(T_i)=x_i$, and thus the space of degree $1$ in the second complex $X_1$ is $P_1=\bigoplus_{i=1}^4RT_i$.

\begin{lemma}
The $\sum_{i=1}^4 c_{j, i} T_i$ for $1 \leq j \leq 11$, $j \neq 3, 4$, generate $H_1(X_1)=\on{Ker}\partial_1/\on{Im}\partial_2$.
\end{lemma}

\begin{proof}
Let $z=\sum_{i=1}^4 z_i T_i$ with $z_i \in \cc[x_1,\dots, x_4]$. If $z \in \on{Ker}\partial_1$, then $\sum_{i=1}^4 z_i x_i=p(r_1, r_2, r_5, \dots, r_{11})$ in $\cc[x_1,\dots, x_4]$ with $p$ a polynomial without constant term. We write $p(r_1, \dots, r_{11})=\sum_{j=1}^{11} a_j r_j+\sum_{k, l}a_{k, l}r_k r_l+\sum_{k, l, m}a_{k, l, m}r_k r_l r_m+ \dots $ where the $a_i$, $a_{k, l}, \dots$ are in $\cc[x_1,\dots, x_4]$. By regrouping the terms we obtain:
\[
\sum_{i=1}^4 \big(z_i -\sum_{j=1}^{11}a_j c_{j,i}-\sum_{k, l}a_{k, l}r_k c_{l, i}-\sum_{k, l, m}a_{k, l, m}r_k r_l c_{m, i}+\dots \big)x_i=0
\]
in $\cc[x_1, \dots, x_4]$. We see that we can reduce the $z_i$ modulo $(r_1, \dots, r_{11})$ so that $p$ has only monomials of degree $1$. So after modding the $z_i$'s we get:
\[
\sum_{i=1}^4 (z_i -\sum_{j=1}^{11}a_j c_{j,i})x_i=0
\]
in $\cc[x_1, \dots, x_4]$. Furthermore, we know that the projective resolution of $\cc$ as a $\cc[x_1, \dots, x_4]$-module is given by the Koszul complex with generators $t_1, \dots, t_4$, so in order to have $\partial_1(\sum_{j=1}^4 p_ j t_j)=p_1x_1+\dots+p_4 x_4=0$ for polynomials $p_i$ in $\cc[x_1, \dots, x_4]$, we need $\sum_{j=1}^4 p_j t_j \in \on{Ker}\partial_1=\on{Im}\partial_2$, which is generated by the $x_i t_j-x_j t_i$. It follows that in the complex $X_1$, we have:
\[
\sum_{i=1}^4 (z_i -\sum_{j=1}^{11}a_j c_{j,i})T_i \in \sum_{1 \leq i \neq j \leq 4}R(x_i T_j-x_j T_i).
\]
We have found that $z \in \on{Ker}\partial_1$ if and only if there exist polynomials $a_j$ such that in $P_1/\on{Im}\partial_2$ we have $z=\sum_{j=1}^{11} a_j (\sum_{i=1}^4 c_{j, i}T_i)$. The assertion of the lemma follows.
\end{proof}

We adjoin $9$ variables $S_1,S_2, S_5 \dots, S_{11}$ of degree 2 such that for any $1 \leq j \leq 11$, $j \neq 3,4$, $\partial_2(S_j)=\sum_{i=1}^4 c_{j, i} T_i$, and we obtain $P_2= \bigoplus_{\begin{subarray}{l} j=1 \\ j \neq 3,4 \end{subarray}}^{11} RS_j \oplus \bigoplus_{1 \leq i < j \leq 4} R T_i T_j$. So the beginning of the minimal resolution of the $R$-module $\cc$ is:
 \begin{center}
  \begin{tikzpicture}[scale=0.9,  transform shape]
  \tikzset{>=stealth}
  
\node (1) at ( -8,0){$\dots$};
\node (2) at ( -4,0){$\displaystyle \bigoplus_{j=1}^{11} RS_j \oplus \bigoplus_{1 \leq i < j \leq 4} R T_i T_j$};
\node (3) at ( 0,0){$\displaystyle \bigoplus_{i=1}^4 RT_i$};
\node (4) at ( 2,0) {$R$};
\node (5) at ( 4,0){$\cc$};
\node (6) at ( 6,0){0};
\node (7) at ( -5.6,-0.7){{\scriptsize $j \neq 3,4$}};

\draw [decoration={markings,mark=at position 1 with
    {\arrow[scale=1.2,>=stealth]{>}}},postaction={decorate}] (1) --  (2) node[midway, above] {$\partial_3$};
\draw [decoration={markings,mark=at position 1 with
    {\arrow[scale=1.2,>=stealth]{>}}},postaction={decorate}] (2)  --  (3) node[midway, above] {$\partial_2$};
\draw [decoration={markings,mark=at position 1 with
    {\arrow[scale=1.2,>=stealth]{>}}},postaction={decorate}] (3)  --  (4) node[midway, above] {$\partial_1$};
\draw [decoration={markings,mark=at position 1 with
    {\arrow[scale=1.2,>=stealth]{>}}},postaction={decorate}] (4)  --  (5) node[midway, above] {$\epsilon$};
\draw [decoration={markings,mark=at position 1 with
    {\arrow[scale=1.2,>=stealth]{>}}},postaction={decorate}] (5)  --  (6) ;
\end{tikzpicture}
 \end{center}
where we keep the notation $\epsilon$ for the augmentation map. As $R$ is no longer a complete intersection, we would need to add variables of higher degrees in order to obtain the minimal resolution. However for our purpose, we do not need to know these other variables. 

We call $\psi$ the lift between the minimal resolutions of $\cc$ as an $R$-module and as an $\widetilde{R}$-module:
 \begin{center}
  \begin{tikzpicture}[scale=0.9,  transform shape]
  \tikzset{>=stealth}
  
\node (1) at ( -8,0){$\dots$};
\node (2) at ( -4,0){$\displaystyle \bigoplus_{j=1}^8 \widetilde{R}s_j \oplus \bigoplus_{1 \leq i < j \leq 8} \widetilde{R} t_i t_j$};
\node (3) at ( 1,0){$\displaystyle \bigoplus_{i=1}^{8} \widetilde{R}t_i$};
\node (4) at ( 3,0) {$\widetilde{R}$};
\node (5) at ( 5,0){$\cc$};
\node (6) at ( 7,0){0};
  
\node (7) at ( -8,-2.5){$\dots$};
\node (8) at ( -4,-2.5){$\displaystyle \bigoplus_{j=1}^{11} RS_j \oplus \bigoplus_{1 \leq i < j \leq 4} R T_i T_j$};
\node (9) at ( 1,-2.5){$\displaystyle \bigoplus_{i=1}^4 RT_i$};
\node (10) at ( 3,-2.5) {$R$};
\node (11) at ( 5,-2.5){$\cc$};
\node (12) at ( 7,-2.5){0};
\node (13) at ( -5.6,-3.2){{\scriptsize $j \neq 3,4$}};

\draw [decoration={markings,mark=at position 1 with
    {\arrow[scale=1.2,>=stealth]{>}}},postaction={decorate}] (1) --  (2) node[midway, above] {$\partial_3$};
\draw [decoration={markings,mark=at position 1 with
    {\arrow[scale=1.2,>=stealth]{>}}},postaction={decorate}] (2)  --  (3) node[midway, above] {$\partial_2$};
\draw [decoration={markings,mark=at position 1 with
    {\arrow[scale=1.2,>=stealth]{>}}},postaction={decorate}] (3)  --  (4) node[midway, above] {$\partial_1$};
\draw [decoration={markings,mark=at position 1 with
    {\arrow[scale=1.2,>=stealth]{>}}},postaction={decorate}] (4)  --  (5) node[midway, above] {$\epsilon$};
\draw [decoration={markings,mark=at position 1 with
    {\arrow[scale=1.2,>=stealth]{>}}},postaction={decorate}] (5)  --  (6);
    
\draw [decoration={markings,mark=at position 1 with
    {\arrow[scale=1.2,>=stealth]{>}}},postaction={decorate}] (7) --  (8) node[midway, above] {$\partial_3$};
\draw [decoration={markings,mark=at position 1 with
    {\arrow[scale=1.2,>=stealth]{>}}},postaction={decorate}] (8)  --  (9) node[midway, above] {$\partial_2$};
\draw [decoration={markings,mark=at position 1 with
    {\arrow[scale=1.2,>=stealth]{>}}},postaction={decorate}] (9)  --  (10) node[midway, above] {$\partial_1$};
\draw [decoration={markings,mark=at position 1 with
    {\arrow[scale=1.2,>=stealth]{>}}},postaction={decorate}] (10)  --  (11) node[midway, above] {$\epsilon$};
\draw [decoration={markings,mark=at position 1 with
    {\arrow[scale=1.2,>=stealth]{>}}},postaction={decorate}] (11)  --  (12);  
    
\draw [decoration={markings,mark=at position 1 with
    {\arrow[scale=1.2,>=stealth]{>}}},postaction={decorate}] (2) --  (8) node[midway, right] {$\psi_2$};
\draw [decoration={markings,mark=at position 1 with
    {\arrow[scale=1.2,>=stealth]{>}}},postaction={decorate}] (3)  --  (9) node[midway, right] {$\psi_1$};
\draw [decoration={markings,mark=at position 1 with
    {\arrow[scale=1.2,>=stealth]{>}}},postaction={decorate}] (4)  --  (10) node[midway, right] {$\pi$};
\draw [decoration={markings,mark=at position 1 with
    {\arrow[scale=1.2,>=stealth]{>}}},postaction={decorate}] (5)  --  (11) node[midway, right] {$\on{id}$};
\end{tikzpicture}
 \end{center}
where the arrows going down are homomorphisms of $\widetilde{R}$-modules. Indeed, there exists an action of $\widetilde{R}$ on the minimal resolution of $\cc$ as an $R$-module induced by the homomorphism $\pi: \widetilde{R} \longrightarrow R$. We call $P_i$ the space of degree $i$ of the bottom resolution, the space of degree $i$ of the top resolution remains $Q_i$,  and we keep the same notation for the differentials. In order to make the diagram commute, it is possible to define $\psi_1$ and $\psi_2$ as:
\begin{align*}
\resizebox{\hsize}{!}{$
\setlength{\arraycolsep}{0.3cm}\begin{array}{cc}
\setlength{\arraycolsep}{6pt}\begin{array}[t]{cccl}
\psi_1: & Q_1& \longrightarrow & P_1 \\
           & t_i & \longmapsto & \left\{\begin{array}{cl}
        T_i  & \text{if }  1 \leq  i \leq 4,\\
        0  & \text{if }  5 \leq i \leq 8,
        \end{array}\right.
\end{array}   &  \setlength{\arraycolsep}{6pt}\begin{array}[t]{cccl}
\psi_2: & Q_2& \longrightarrow & P_2 \\
           & t_i t_j & \longmapsto & \left\{\begin{array}{cl}
        T_i T_j & \text{ if } 1 \leq i, j \leq 4,\\
        0 & \text{ otherwise},
        \end{array}\right. \\
        & s_j & \longmapsto &  \left\{\begin{array}{cl}
        S_j  & \text{if }  j \neq 3,4,\\
        H S_8 & \text{if }  j=3, \\
        -\frac{1}{C_p} \omega^p S_8 & \text{if }  j=4.
        \end{array}\right.
\end{array}
\end{array}$}
\end{align*}

The next proposition gives an explicit realisation of $\psi_n$ for any $n \in \mathbb{N}^*$.

\begin{proposition}\label{LiftPsi}
Let $m \in \mathbb{N}^*$ and set $t_1^{a_1}\dots t_{8}^{a_8}s_1^{b_1}\dots s_8^{b_8} \in Q_{m}$. Then:
\begin{align*}
\psi_m(t_1^{a_1}\dots t_{8}^{a_8}s_1^{b_1}\dots s_8^{b_8})=T_1^{a_1} \dots T_4^{a_4}0^{a_5} \dots 0^{a_8}S_1^{b_1}S_2^{b_2}(HS_8)^{b_3}(-\frac{1}{C_p} \omega^p S_8)^{b_4}S_5^{b_5}\dots S_8^{b_8}.
\end{align*}
\end{proposition}

\begin{proof}
We write $\mathcal{P}(m)$ for the statement of the proposition for a fixed $m \in \mathbb{N}^*$. We have just seen that $\mathcal{P}(1)$ and $\mathcal{P}(2)$ are true. As with the previous proof, we will proceed by induction.

Assume $\mathcal{P}(m)$ is true for some $m \in \mathbb{N}^*$. We fix $t_1^{a_1}\dots t_{8}^{a_8}s_1^{b_1}\dots s_8^{b_8} \in Q_{m+1}$ and directly compute:
\begin{equation*}
\resizebox{\hsize}{!}{$
\begin{array}{l}
\psi_n \partial_{m+1}(t_1^{a_1}\dots t_8^{a_8}s_1^{b_1}\dots s_8^{b_8}) \\
 =\displaystyle  \sum_{i=1}^4 (-1)^{a_1+a_2+\dots +a_{i-1}} a_ix_iT_1^{a_1} \dots \widehat{T_i^{a_i}} \dots T_4^{a_4}0^{a_5} \dots 0^{a_8}S_1^{b_1}S_2^{b_2}(HS_8)^{b_3}\dots S_8^{b_8}  \\
+ (-1)^{a_1+\dots a_8} T_1^{a_1}\dots T_4^{a_4}0^{a_5} \dots 0^{a_8} \Bigg[ \displaystyle \sum_{\substack{j=1 \\ j \neq 3,4}}^8 b_j \partial_2(S_j)S_1^{b_1}S_2^{b_2}(HS_8)^{b_3}\dots S_j^{b_j-1} \dots S_8^{b_8}  \\
 +b_3 \partial_2(HS_8)S_1^{b_1}S_2^{b_2}(HS_8)^{b_3-1}\dots S_8^{b_8} + b_4 \partial_2(-\frac{1}{C_p} \omega^p S_8)S_1^{b_1}S_2^{b_2}(HS_8)^{b_3}(-\frac{1}{C_p} \omega^p S_8)^{b_4-1}\dots S_8^{b_8} \Bigg], \\
 = \partial_{m+1}(T_1^{a_1} \dots T_4^{a_4}0^{a_5} \dots 0^{a_8}S_1^{b_1}S_2^{b_2}(HS_8)^{b_3}(-\frac{1}{C_p} \omega^p S_8)^{b_4}S_5^{b_5}\dots S_8^{b_8}).
 \end{array}$}
\end{equation*}
Because all elements of the basis of $Q_{m+1}$ are of the form $t_1^{a_1}\dots t_8^{a_8}s_1^{b_1}\dots s_8^{b_8}$, we can define:
\begin{align*}
\resizebox{\hsize}{!}{$
\begin{array}[t]{cccc}
\psi_{n+1}: & Q_{m+1} & \longrightarrow & P_{m+1} \\
           & t_1^{a_1}\dots t_8^{a_8}s_1^{b_1}\dots s_8^{b_8} & \longmapsto & T_1^{a_1} \dots T_4^{a_4}0^{a_5} \dots 0^{a_8}S_1^{b_1}S_2^{b_2}(HS_8)^{b_3}(-\frac{1}{C_p} \omega^p S_8)^{b_4}S_5^{b_5}\dots S_8^{b_8}.
\end{array}$}
\end{align*}
This makes the diagram commute and so $\mathcal{P}(m+1)$ is true. As $\mathcal{P}(1)$ and $\mathcal{P}(2)$ are true, this concludes the proof.
\end{proof}

The induced homomorphism $\pi^{\#}: H^{*}(R) \longrightarrow H^{*}(\widetilde{R})$ is defined through the lift $\psi$ by $\pi^{\#}(f)=f \circ \psi$. As we have just found a description of $\psi$, we can compute the image of $\pi^{\#}$ for small degrees. For $1 \leq i \leq 4$ and $1 \leq j \leq 11$ ($j \neq 3,4$), let $\gamma_i$ and $\delta_j$ be the duals of $T_i$ and $S_j$. 

We know from the construction of the beginning of the free resolution $P$ that $\partial_1^*=\partial_2^*=0$. It follows that $H^1(R)=\text{Hom}_R(P_1,\cc)$, and $H^2(R)=\text{Ker}(\partial_3^*)$. We cannot express $\text{Ker}(\partial_3^*)$ explicitly because we do not know the variables we may need to add to $P_3$. However, we can prove the following lemma:

\begin{lemma}
The $\delta_j$'s defined above are elements of $H^2(R)$.
\end{lemma} 

\begin{proof}
We need to show that for all $1 \leq j \leq 11$ ($j \neq 3,4$), we have $0=\partial_3^*(\delta_j)=\delta_j \circ \partial_3$. But when the complex is made acyclic, we know that $\on{Im}(\partial_3)=\text{Ker}(\partial_2)$. So $\delta_j \in H^2(R)$ if and only if $\delta_j$ is zero on $\text{Ker}(\partial_2)$. As $\delta_j$ is the dual of $S_j$, this is equivalent to saying that no element of $\text{Ker}(\partial_2)$ contains a non zero scalar multiple of $S_j$.

Let us assume that for $1 \leq l \leq 11$ ($l \neq 3,4$), $1 \leq i < m \leq 4$ there exist $a_l, a_{i, m} \in \cc[x_1, \dots, x_4]$ such that $a_j S_j+\sum_{l \neq j}a_l S_l+\sum_{i <m}a_{i, m}T_iT_m$ is in $\text{Ker}(\partial_2)$, and with $a_j \equiv 1 \text{ mod } (x_1, \dots, x_4)$. It follows that:
\begin{align*}
\renewcommand{\arraystretch}{2}
\begin{array}{cl}
0 &\displaystyle  =\sum_{i=1}^4 a_j c_{j, i}T_i+\sum_{l \neq j}a_l (\sum_{i=1}^4c_{l, i}T_i)+\sum_{i <m}a_{i, m}(x_iT_m-x_mT_i), \\
 & \displaystyle = \sum_{i=1}^4 \big(a_j c_{j, i}+\sum_{l \neq j}a_l c_{l, i}+\sum_{m<i}x_m a_{m, i}-\sum_{i<m}x_m a_{i, m} \big)T_i.
 \end{array}
\end{align*}
As the $T_i$'s are free, we get that for all $1 \leq i \leq 4$:
\begin{align*}
a_j c_{j, i}+\sum_{l \neq j}a_l c_{l, i}+\sum_{m<i}x_m a_{m, i}-\sum_{i<m}x_m a_{i, m}=0 \text{ in } R.
\end{align*}
Thus for all $1 \leq i \leq 4$, there exists a polynomial $p_i \in \cc[x_1, \dots, x_4][y_1, \dots, y_9]$ without constant term such that:
\begin{align*}
a_j c_{j, i}+\sum_{l \neq j}a_l c_{l, i}+\sum_{m<i}x_m a_{m, i}-\sum_{i<m}x_m a_{i, m}=p_i(r_1,r_2, r_5 \dots, r_{11}) \text{ in } \cc[x_1, \dots, x_4].
\end{align*}
By multiplying each expression by $x_i$, and then summing over $i$, the last two terms cancel each other and we get 
\begin{align*}
a_j r_j+\sum_{l \neq j}a_l r_l=\sum_{i=1}^n x_i p_i(r_1, r_2, r_5, \dots, r_{11}).
\end{align*}
The coefficient of $r_j$ in the right hand side is clearly in $(x_1, \dots, x_4)$. We pass it to the left side and regroup all the $r_l$ with $l \neq j$ to the right hand side and obtain:
\begin{align*}
(1+\widetilde{a_j})r_j=p(r_1, r_2, r_5, \dots, \widehat{r_j}, \dots, r_{11}).
\end{align*}
with $\widetilde{a_j} \in(x_1, \dots, x_4)$ and $p \in \cc[x_1, \dots, x_4][y_1, \dots, y_8]$. But as the $r_l$'s are homogeneous, this implies that $r_j=\widetilde{p}(r_1, r_2, r_5, \dots, \widehat{r_j}, \dots, r_{11})$ for some polynomial $\widetilde{p}$ without a constant term, contradicting the minimality of the $r_l$'s as generators of $(r_1, \dots, r_{11})$. Therefore no element of $\text{Ker}(\partial_2)$ is of the assumed form.
\end{proof}

We can directly compute:
\begin{align*}
\setlength{\arraycolsep}{1cm}
\begin{array}{cc}
\pi^{\#}(\gamma_i)=\gamma_i \circ \psi_1=\alpha_i,  \ \forall \ 1 \leq i \leq 4, &
\pi^{\#}(\delta_i) = \left\{\setlength{\arraycolsep}{6pt}\begin{array}{cl}
        \alpha_i^2  & \text{ if }  i=1,2 ,\\
        -\alpha_i^2  & \text{ if }  5 \leq i \leq 8, \\
        0  & \text{ if }  9 \leq i \leq 11. \\
        \end{array}\right.
\end{array}
\end{align*}
We do not know the basis of $P_3$ and $P_4$ so we cannot determine explicitly the commutation relations between $\gamma_i$ and $\delta_j$. However, we can obtain the following relations:
\begin{itemize}
\item $\gamma_j \gamma_i \equiv -\gamma_i \gamma_j \text{ mod }\text{Ker}(\pi^{\#})$ for $1 \leq i < j \leq 4$,
\item $\gamma_4^2=0$, 
\item $\gamma_i \delta_j \equiv \delta_j \gamma_j \text{ mod }\text{Ker}(\pi^{\#})$ for $1 \leq i \leq 4, \ 1 \leq  j \leq 11, \ j \neq 3,4$, 
\item $\delta_j \delta_i \equiv \delta_i \delta_j \text{ mod }\text{Ker}(\pi^{\#})$ for $1 \leq i < j \leq 11, \  i, j \neq 3,4$.
\end{itemize}
We can compute that $\gamma_1^2=\delta_1, \gamma_2^2=\delta_2, \gamma_3^2=\delta_8$. Moreover, modulo $\text{Ker}(\pi^{\#})$, $\gamma_i \gamma_j$ is the dual of $T_i T_j$ in $P_2$.

It is known that the induced homomorphism $\pi^{\#}$ is a homomorphism of algebras preserving the Yoneda product (it is easily seen when the product is done by splicing $n$-fold exact sequences), thus 
\begin{align*}
\cc[\alpha_5^2, \alpha_6^2, \alpha_7^2]  \otimes \left(\cc\langle \alpha_1, \dots, \alpha_4 \rangle /\left( \begin{array}{c} \alpha_i \alpha_j+\alpha_j \alpha_i \ \forall i \neq j, \\ \alpha_4^2 \end{array}\right) \right) \subseteq \on{Im}(\pi^{\#}). 
\end{align*}
We intend to prove that these two spaces are in fact equal.

Let $a_1, \dots ,a_4 \in \{0,1\}$, $b_1,b_2, b_5,  \dots, b_8 \in \mathbb{N}$ such that $a_1+ \dots +a_4+2( b_1+ \dots+ b_8)=m$. We define:
\begin{align*}
\begin{array}[t]{cccl}
(T_1^{a_1}\dots T_4^{a_4}S_1^{b_1}\dots S_8^{b_8})^*: & P_m & \longrightarrow & \cc \\
           & T_1^{a_1}\dots T_4^{a_4}S_1^{b_1}\dots S_8^{b_8} & \longmapsto & 1, \\
           & \text{other} & \longmapsto & 0.
\end{array}
\end{align*}

\begin{proposition}\label{MonomialContainingTheDual}
For any $m \in \mathbb{N}^*$, let $T_1^{a_1}\dots T_4^{a_4}S_1^{b_1}\dots S_8^{b_8} \in P_m$. There exists $z \in \text{Ker}(\pi^{\#})$ such that the corresponding basis element of $H^{m}(R)$ is a non-trivial multiple of $\gamma_1^{a_1}\dots \gamma_4^{a_4}\delta_1^{b_1}\dots \delta_8^{b_8}+z$.
\end{proposition}

\begin{proof}
See Section~\ref{Appendix}.
\end{proof}
 
We know that $\on{Im}(\psi)$ is generated by the $T_i, S_j$ ($j \leq 8$), so any element of $H^*(R)$ corresponding to a monomial containing variables different from $T_i, S_j$ is in the kernel of $\pi^\#$. If we call $P$ the $R$-algebra giving the free resolution of $\cc$ as an $R$-module, we have:
\begin{align*}
\begin{array}{rcl}
\text{Ker}(\pi^{\#}) & = & \on{Span} \{\text{duals of monomials in } P \text{ containing at least one}  \\
 & & \quad \quad \quad \text{ variable different from }T_i, S_j \ (j \leq 8)\}.
\end{array}
\end{align*}
\indent A consequence of Proposition~\ref{MonomialContainingTheDual} is that 
\begin{align*}
\begin{array}{rcl}
\pi^\#((T_1^{a_1}\dots T_4^{a_4}S_1^{b_1}\dots S_8^{b_8})^*) & = & \pi^\#(a\gamma_1^{a_1}\dots \gamma_4^{a_4}\delta_1^{b_1}\dots \delta_8^{b_8}), \\[5pt]
& = & a\alpha_1^{a_1}\dots \alpha_4^{a_4}(\alpha_1^2)^{b_1}\dots (-\alpha_8^2)^{b_8}
\end{array}
\end{align*}
for some $a \in \cc^*$, and so:
\begin{align*}
\begin{array}[t]{l}
\on{Im}(\pi^{\#})= \cc[\alpha_5^2, \alpha_6^2, \alpha_7^2]  \otimes \left(\cc\langle \alpha_1, \dots, \alpha_4 \rangle /\left( \begin{array}{c} \alpha_i \alpha_j+\alpha_j \alpha_i, 1 \leq i \neq j \leq 4, \\ \alpha_4^2 \end{array}\right) \right).
\end{array}
\end{align*}
As $\on{Im}(\pi^{\#})$ can be seen as the quotient of $H^*(R)$ by $\text{Ker}(\pi^{\#})$, we have proved the following:

\begin{proposition}\label{QuotientofExt}
There exists a two sided ideal $I$ of $\on{Ext}_{\on{gr}A(\mathcal{W}(p))}^*(X_{\on{gr}A}, X_{\on{gr}A})$ such that:
\begin{align*}
\resizebox{\hsize}{!}{$
\on{Ext}_{\on{gr}A(\mathcal{W}(p))}^*(X_{\on{gr}A}, X_{\on{gr}A})/I \cong \cc[\alpha_5^2, \alpha_6^2, \alpha_7^2]  \otimes \left(\cc\langle \alpha_1, \dots, \alpha_4 \rangle /\left(\begin{array}{c} \alpha_i \alpha_j+\alpha_j \alpha_i,  1 \leq i \neq j \leq 4, \\ \alpha_4^2 \end{array}\right) \right).$}
\end{align*}
\end{proposition}

\begin{remark}\label{rem:6.19}
\begin{enumerate}
\item We know that $kQ_{\on{gr}A(\mathcal{W}(p))}=\cc \langle E,F,H,\omega \rangle$ is a free algebra, so there is a surjection:
\begin{equation*}
kQ_{\on{gr}A(\mathcal{W}(p))}  \  \raisebox{2pt}{\begin{tikzpicture}[scale=1, transform shape] \tikzset{>=stealth} \draw[->]  (0.4,0) -- (1,0); \draw[->]  (0.4,0) -- (0.9,0); \end{tikzpicture}} \  \cc[E,F,H,\omega]/(r_1, \dots, r_{11}) \cong \on{End}_{\on{gr}A(\mathcal{W}(p))}(\mathcal{P}_{\on{gr}A}). 
\end{equation*}
However the Ext algebra behaves differently than those of $A(\mathcal{W}(p))$ and $\mathcal{W}(p)$. The natural homomorphism $kQ_{\on{gr}A(\mathcal{W}(p))} \longrightarrow \on{Ext}_{\on{gr}A(\mathcal{W}(p))}^*(X_{\on{gr}A}, X_{\on{gr}A})$ sending each arrow to its corresponding element of the basis of $\on{Ext}_{\on{gr}A(\mathcal{W}(p))}^1(X_{\on{gr}A}, X_{\on{gr}A})$ (see Lemma~\ref{kQandExt}) is not surjective anymore. Indeed, for it to be surjective it is necessary that $\on{Ext}_{\on{gr}A(\mathcal{W}(p))}^*(X_{\on{gr}A}, X_{\on{gr}A})$ is generated by elements of degree $1$ over $\on{Ext}_{\on{gr}A(\mathcal{W}(p))}^0(X_{\on{gr}A}, X_{\on{gr}A})$. But we know that for example $\delta_5$ is of degree 2 and is not generated by elements of lower degrees. Therefore such a natural surjective homomorphism does not exist.

\item We have seen that the known presentation of $R(\mathcal{W}(p))$ is not a complete intersection, making the Ext algebra difficult to determine explicitly. This raises the following question: what type of vertex operator algebra $V$ satisfies the property that $R(V)$ is a complete intersection? We know that for $\mathfrak{g}$ a Lie algebra of dimension $n$ and $k \in \cc$ we have $R(V_{\hat{\mathfrak{g}}}(k,0)) \cong \cc[x_1, \dots, x_n]$ and so it is a complete intersection. However it is proved in \cite[Proposition 6.4]{Caradot-Jiang-Lin} that for all $k \in \mathbb{N}^*$, $\on{Spec}(R(L_{\hat{\mathfrak{g}}}(k,0)))$ is not a complete intersection as a subspace of $\mathfrak{g}^*$. Therefore the relationship between the properties of the vertex operator algebra and the complete intersection nature of the $C_2$-algebra is unclear.
\end{enumerate}
\end{remark}

We see that $\on{Im}(\pi^{\#})$ contains $\cc[\alpha_1]$, therefore $\on{Ext}_{\on{gr}A(\mathcal{W}(p))}^*(X_{\on{gr}A}, X_{\on{gr}A})$ contains $\cc[\gamma_1]$, meaning that for all $n \geq 1$, we have $0 \neq \gamma_1^n \in \on{Ext}_{\on{gr}A(\mathcal{W}(p))}^n(X_{\on{gr}A}, X_{\on{gr}A})$. We obtain the following corollary:

\begin{corollary}
The $\on{gr}A(\mathcal{W}(p))$-module $X_{\on{gr}A}$ has infinite projective dimension, and so $\on{gr}A(\mathcal{W}(p))$ has infinite global dimension.
\end{corollary}

\begin{remark}
The algebra $\mathcal{E}(\on{gr}A(\mathcal{W}(p)))(\cong \mathcal{E}(\on{gr}A(\mathcal{W}(p)))^{op})$ is a quotient of the path algebra of the four loops quiver. However, as $\on{Ext}_{\on{gr}A(\mathcal{W}(p))}^*(X_{\on{gr}A}, X_{\on{gr}A})$ is not generated in degree $0$ and $1$ by Remark~\ref{rem:6.19}, we conclude that $\mathcal{E}(\on{gr}A(\mathcal{W}(p)))$ is not Koszul.
\end{remark}

From Proposition~\ref{QuotientofExt}, we know that there is a surjection
\[
\on{Ext}_{\on{gr}A(\mathcal{W}(p))}^*(X_{\on{gr}A}, X_{\on{gr}A}) \longrightarrow \cc[\alpha_5^2, \alpha_6^2, \alpha_7^2] \otimes \big( \cc \langle \alpha_1, \alpha_2, \alpha_3 \rangle/(\alpha_i \alpha_j+\alpha_j \alpha_i, 1 \leq i \neq j \leq 3) \big).
\]
We immediately obtain:

\begin{corollary}
The Gelfand-Kirillov dimension of the algebra $\on{Ext}_{\on{gr}A(\mathcal{W}(p))}^*(X_{\on{gr}A}, X_{\on{gr}A})$ is at least $6$.
\end{corollary}

\section{Appendix}\label{Appendix}
\delete{\subsection{Proofs of Section 6.1}\label{Proofs_of_Section_6.1}

\begin{proof}[Proof of Proposition~\ref{RelationsInExt}]
We first show that $\alpha_1\beta_2+\alpha_2\beta_1=0$ in $\on{Ext}_{\mathcal{W}(p)}^{2}(X_s^+, X_s^+)$. An element of $\on{Ext}_{\mathcal{W}(p)}^{1}(X_s^-, X_s^+)$ is represented by an element of $\text{Hom}_{\mathcal{W}(p)}(P_s^+ \oplus P_s^+, X_s^+)$. Because $(\alpha_1,\alpha_2)$ is a basis $\on{Ext}_{\mathcal{W}(p)}^{1}(X_s^-, X_s^+)$, we choose the basis so that $\alpha_1$ and $\alpha_2$ are represented by:
\begin{align*}
\begin{array}[t]{cccc}
\alpha_1: & P_s^+ \oplus P_s^+ & \longrightarrow & X_s^+\\
                 & (x,y) & \longmapsto      & \overline{x}
\end{array} \text{ and }\begin{array}[t]{cccc}
\alpha_2: & P_s^+ \oplus P_s^+ & \longrightarrow & X_s^+\\
                 & (x,y) & \longmapsto      & \overline{y}
\end{array}
\end{align*} 
where $\overline{x}$ and $\overline{y}$ are the class of $x$ and $y$ modulo $\text{Rad}(P_s^+)$ (we keep the same name for the element of $\on{Ext}_{\mathcal{W}(p)}^{1}(X_s^-, X_s^+)$ and its homomorphism representative). Likewise, we choose $\beta_1,\beta_2$ so that they are represented by:
\begin{align*}
\begin{array}[t]{cccc}
\beta_1: & P_s^- \oplus P_s^- & \longrightarrow & X_s^-\\
                 & (x,y) & \longmapsto      & \overline{x}
\end{array}  \text{ and } \begin{array}[t]{cccc}
\beta_2: & P_s^- \oplus P_s^- & \longrightarrow & X_s^-\\
                 & (x,y) & \longmapsto      & \overline{y}.
\end{array}
\end{align*} 

We know that the product $\alpha_i\beta_1$ is defined as $\alpha_i \circ \beta_1'': P_s^+ \oplus P_s^+ \oplus P_s^+\longrightarrow X_s^+$ where $\beta_1''$ is a lift of $\beta_1$ such that the following diagram commutes:
\begin{center}
 \begin{tikzpicture}[scale=1, transform shape]
\tikzset{>=stealth}
\node (1) at (-2,0) []{$P_s^+ \oplus P_s^+ \oplus P_s^+$};
\node (2) at (-2,-2) []{$P_s^+ \oplus P_s^+$};
\node (3) at (2,0) []{$P_s^- \oplus P_s^-$};
\node (4) at (2,-2) []{$P_s^-$};
\node (5) at (4,-2) []{$X_s^-$};
\node (6) at (-2,-3.5) []{$X_s^+$};
\draw[->,dashed]  (1) -- (2) node[midway, right] {$\beta_1''$};
\draw[->,dashed]  (3) -- (4) node[midway, right] {$\beta_1'$};
\draw[->]  (1) -- (3) node[midway, above] {$\partial_2$};
\draw[->]  (2) -- (4) node[midway, above] {$\partial_1$};
\draw[->]  (4) -- (5) node[midway, above] {$\partial_0$};
\draw[->]  (3) -- (5) node[midway, above right] {$\beta_1$};
\draw[->]  (2) -- (6) node[midway, right] {$\alpha_i$};
\end{tikzpicture}
\end{center}

The homomorphism $\partial_2$ can be represented with module diagrams as follows:
\begin{center}
\resizebox{\hsize}{!}{$
 \begin{tikzpicture}[scale=1, transform shape]
\tikzset{>=stealth}
\node (1) at (0,1.5) []{$X_s^+$};
\node (2) at (-1,0) []{$X_s^-$};
\node (3) at (1,0) []{$X_s^-$};
\node (4) at (0,-1.5) []{$X_s^+$};
\draw[->]  (1) -- (2);
\draw[->]  (1) -- (3);
\draw[->]  (2) -- (4);
\draw[->]  (3) -- (4) ;

\node (5) at (3,1.5) []{$X_s^+$};
\node (6) at (2,0) []{$X_s^-$};
\node (7) at (4,0) []{$X_s^-$};
\node (8) at (3,-1.5) []{$X_s^+$};
\draw[->]  (5) -- (6);
\draw[->]  (5) -- (7);
\draw[->]  (6) -- (8);
\draw[->]  (7) -- (8) ;

\node (9) at (6,1.5) []{$X_s^+$};
\node (10) at (5,0) []{$X_s^-$};
\node (11) at (7,0) []{$X_s^-$};
\node (12) at (6,-1.5) []{$X_s^+$};
\draw[->]  (9) -- (10);
\draw[->]  (9) -- (11);
\draw[->]  (10) -- (12);
\draw[->]  (11) -- (12) ;

\draw[-]  (-0.5,1.75) -- (6.5,1.75);
\draw[-]  (1.75,0.25) -- (4.25,0.25);
\draw[-]  (0.5,-0.25) -- (-0.5,1.75);
\draw[-]  (0.5,-0.25) -- (1.25,-0.25) ;
\draw[-]  (1.75,0.25) -- (1.25,-0.25) ;
\draw[-]  (4.25,0.25) -- (4.75,-0.25);
\draw[-]  (4.75,-0.25) --  (5.5,-0.25) ;
\draw[-]  (5.5,-0.25) -- (6.5,1.75);

\node (13) at (10,1.5) []{$X_s^-$};
\node (14) at (9,0) []{$X_s^+$};
\node (15) at (11,0) []{$X_s^+$};
\node (16) at (10,-1.5) []{$X_s^-$};
\draw[->]  (13) -- (14);
\draw[->]  (13) -- (15);
\draw[->]  (14) -- (16);
\draw[->]  (15) -- (16) ;

\node (17) at (13,1.5) []{$X_s^-$};
\node (18) at (12,0) []{$X_s^+$};
\node (19) at (14,0) []{$X_s^+$};
\node (20) at (13,-1.5) []{$X_s^-$};
\draw[->]  (17) -- (18);
\draw[->]  (17) -- (19);
\draw[->]  (18) -- (20);
\draw[->]  (19) -- (20) ;

\draw[-]  (8.5,0.25) -- (14.5,0.25) ;
\draw[-]  (9.5,-1.75) -- (13.5,-1.75) ;
\draw[-]  (8.5,0.25) -- (9.5,-1.75) ;
\draw[-]  (14.5,0.25) -- (13.5,-1.75) ;

\draw[->]  (7.25,0.5) -- (8.25,0.25) node[midway, above] {$\partial_2$} ;
\end{tikzpicture}$}
\end{center}
such that $\partial_2 \left[ \begin{pmatrix} x &  & y & & z \\  & w &  & t & \end{pmatrix} \right]=\begin{pmatrix} x &  & y & & -y & & z  \\  & w & & & & t &  \end{pmatrix}$ (here the matrices represent the coordinates in the diagram above). Likewise $\partial_1$ is given by:
\begin{center}
 \begin{tikzpicture}[scale=1, transform shape]
\tikzset{>=stealth}
\node (1) at (0,1.5) []{$X_s^+$};
\node (2) at (-1,0) []{$X_s^-$};
\node (3) at (1,0) []{$X_s^-$};
\node (4) at (0,-1.5) []{$X_s^+$};
\draw[->]  (1) -- (2);
\draw[->]  (1) -- (3);
\draw[->]  (2) -- (4);
\draw[->]  (3) -- (4) ;

\node (5) at (3,1.5) []{$X_s^+$};
\node (6) at (2,0) []{$X_s^-$};
\node (7) at (4,0) []{$X_s^-$};
\node (8) at (3,-1.5) []{$X_s^+$};
\draw[->]  (5) -- (6);
\draw[->]  (5) -- (7);
\draw[->]  (6) -- (8);
\draw[->]  (7) -- (8) ;

\draw[-]  (0.5,-0.25) -- (2.5,-0.25) ;
\draw[-]  (-0.5,1.75) -- (3.5,1.75) ;
\draw[-]  (0.5,-0.25) -- (-0.5,1.75) ;
\draw[-]  (2.5,-0.25) -- (3.5,1.75) ;

\node (9) at (7,1.5) []{$X_s^-$};
\node (10) at (6,0) []{$X_s^+$};
\node (11) at (8,0) []{$X_s^+$};
\node (12) at (7,-1.5) []{$X_s^-$};
\draw[->]  (9) -- (10);
\draw[->]  (9) -- (11);
\draw[->]  (10) -- (12);
\draw[->]  (11) -- (12) ;

\draw[-]  (5.5,0.25) -- (8.5,0.25) ;
\draw[-]  (6.5,-1.75) -- (7.5,-1.75) ;
\draw[-]  (5.5,0.25) -- (6.5,-1.75) ;
\draw[-]  (8.5,0.25) -- (7.5,-1.75) ;

\draw[->]  (4,0.5) -- (5.25,0.25) node[midway, above] {$\partial_1$} ;
\end{tikzpicture}
\end{center}
such that $\partial_1 \left[ \begin{pmatrix} x &  &  & y  \\  & w & t &   \end{pmatrix} \right]=\begin{pmatrix} x &  & y   \\   & \frac{1}{2}(w+t) &   \end{pmatrix}$. Define:
\begin{align*}
\begin{array}[t]{cccc}
\beta_1': & P_s^- \oplus P_s^- & \longrightarrow & P_s^-\\
                 & (x,y) & \longmapsto      & x
\end{array}, \quad \begin{array}[t]{cccc}
\beta_2': & P_s^- \oplus P_s^- & \longrightarrow & P_s^-\\
                 & (x,y) & \longmapsto      & y
\end{array} 
\end{align*}
as well as $\beta_1'', \beta_2'': P_s^+ \oplus P_s^+ \oplus P_s^+ \longrightarrow P_s^+ \bigoplus P_s^+ $ such that 
\begin{align*}
\beta_1'' \left[  \begin{pmatrix} x &  & y & & z \\  & w &  & t & \end{pmatrix}  \right]= \begin{pmatrix} x &  &  & y  \\  & w & w &   \end{pmatrix} \text{ and } \beta_2'' \left[  \begin{pmatrix} x &  & y & & z \\  & w &  & t & \end{pmatrix}  \right]= \begin{pmatrix} -y &  &  & z  \\  & t & t &   \end{pmatrix} 
\end{align*}
for all $x,y,z \in X_s^+$, $w, t \in X_s^-$.

We see that $\beta_1' \circ \partial_2=\partial_1 \circ \beta_1''$ and that $\beta_1=\partial_0 \circ \beta_1'$. Hence $\alpha_1 \beta_1$ is represented by the homomorphism $\alpha_1 \circ \beta_1''$ and $\alpha_2 \beta_1$ by the homomorphism $\alpha_2 \circ \beta_1''$. With a similar verification, we check that $\alpha_1 \beta_2$ and $\alpha_2 \beta_2$ are represented by the homomorphisms $\alpha_1 \circ \beta_2''$ and $\alpha_2 \circ \beta_2''$ respectively. We can then verify that:
\[
\alpha_1 \circ \beta_2''+\alpha_2 \circ \beta_1''=0,
\]
which means that $\alpha_1 \beta_2+\alpha_2 \beta_1=0$ in $\on{Ext}_{\mathcal{W}(p)}^{2}(X_s^+, X_s^+)$ and so $\alpha_1\otimes \beta_2+\alpha_2 \otimes \beta_1 \in \text{Ker}(\pi)$. It is straightforward to check that any linear combination of the $\alpha_i\beta_j$ is zero if and only if it is proportional to $\alpha_1 \beta_2+\alpha_2 \beta_1$. It follows that $ \text{Ker}(\pi)=\cc(\alpha_1\otimes \beta_2+\alpha_2 \otimes \beta_1)$, and so $\text{dim Im}(\pi)=3=\on{dim}\on{Ext}_{\mathcal{W}(p)}^{2}(X_s^+, X_s^+)$ by Theorem~\ref{dimExtn}. Therefore $\pi$ is surjective. The other case is obtained by a symmetric construction.
\end{proof}

\begin{proof}[Proof of Theorem~\ref{GeneratorsOfExt}]
Set $i \geq 1$ and let $f \in \on{Ext}_{\mathcal{W}(p)}^{2i}(X_s^+, X_s^+)$ be an element of the basis represented by $f : \bigoplus_{k=1}^{2i+1}P_s^+ \longrightarrow X_s^+$ ($f$ is the quotient of one coordinate by $\text{Rad}(P_s^+)$). The module $X_s^+$ is simple hence $f$ is surjective, and $\bigoplus_{k=1}^{2i-1}P_s^+ \subseteq \text{Ker }f$. Furthermore, based on the construction of the projective resolution fo $X_s^+$, we know that $\Omega_{2i}=\text{Ker }\partial_{2i}$ sits in $\bigoplus_{k=1}^{2i+1}P_s^+$ like:
\begin{center}
 \begin{tikzpicture}[scale=1, transform shape]
\tikzset{>=stealth}
\node (1) at (0,1.5) []{$X_s^+$};
\node (2) at (-1,0) []{$X_s^-$};
\node (3) at (1,0) []{$X_s^-$};
\node (4) at (0,-1.5) []{$X_s^+$};
\draw[->]  (1) -- (2);
\draw[->]  (1) -- (3);
\draw[->]  (2) -- (4);
\draw[->]  (3) -- (4) ;

\node (5) at (4,1.5) []{$X_s^+$};
\node (6) at (3,0) []{$X_s^-$};
\node (7) at (5,0) []{$X_s^-$};
\node (8) at (4,-1.5) []{$X_s^+$};
\draw[->]  (5) -- (6);
\draw[->]  (5) -- (7);
\draw[->]  (6) -- (8);
\draw[->]  (7) -- (8) ;

\node (9) at (2,0) []{$\dots \dots$};

\draw[-]  (-1.5,0.25) -- (-0.5,0.25);
\draw[-]  (-0.5,0.25) -- (0.5,-0.25);
\draw[-]  (0.5,-0.25) -- (1.5,-0.25);

\draw[-]  (-1.5,0.25) -- (-0.5,-1.75) ;
\draw[-]  (-0.5,-1.75) --  (4.5,-1.75) ;
\draw[-]  (4.5,-1.75) -- (5.5,0.25);
\draw[-]   (5.5,0.25) -- (4.5,0.25);

\draw[-]  (4.5,0.25) -- (3.5,-0.25);
\draw[-]  (3.5,-0.25) -- (2.5,-0.25);

\node (10) at (2,-0.25) []{$\dots$};
\node (11) at (2,-1) []{$\Omega_{2i}$};

\end{tikzpicture}
\end{center}
Because $\bigoplus_{k=1}^{2i-1}P_s^+ \subseteq \text{Ker }f$, there exists  $\delta : P_s^+ \oplus P_s^+ \longrightarrow X_s^+$ such that $f=\delta \circ \varpi$ where $\varpi$ is the projection $\bigoplus_{k=1}^{2i+1}P_s^+ \longrightarrow (\bigoplus_{k=1}^{2i+1}P_s^+)/(\bigoplus_{k=1}^{2i-1}P_s^+)\cong P_s^+ \oplus P_s^+$. We also write $\delta$ for its class in $\on{Ext}_{\mathcal{W}(p)}^1(X_s^-, X_s^+)$.

We have the following diagram:
\begin{center}
 \begin{tikzpicture}[scale=1, transform shape]
\tikzset{>=stealth}
\node (1) at (0,0) []{$0$};
\node (2) at (3,0) []{$\Omega_{2i}$};
\node (3) at (6,0) []{$\displaystyle\bigoplus_{k=1}^{2i+1}P_s^+$};
\node (4) at (9,0) []{$\displaystyle\bigoplus_{k=1}^{2i}P_s^-$};
\node (5) at (12,0) []{$\displaystyle\bigoplus_{k=1}^{2i-1}P_s^+$};
\draw[->]  (1) -- (2);
\draw[->]  (2) -- (3);
\draw[->]  (3) -- (4)  node[midway, above] {$\partial_{2i}$};
\draw[->]  (4) -- (5)  node[midway, above] {$\partial_{2i-1}$};

\node (6) at (0,-2) []{$0$};
\node (7) at (3,-2) []{$\varpi(\Omega_{2i})$};
\node (8) at (6,-2) []{$P_s^+ \oplus P_s^+$};
\node (9) at (9,-2) []{$P_s^-$};
\node (10) at (12,-2) []{$X_s^-$};
\node (11) at (14,-2) []{$0$};

\draw[->]  (6) -- (7);
\draw[->]  (7) -- (8);
\draw[->]  (8) -- (9) node[midway, below] {$\partial_{1}$};
\draw[->]  (9) -- (10) node[midway, below] {$\partial_{0}$};
\draw[->]  (10) -- (11) ;

\draw[->]  (2) -- (7) node[midway, right] {$\varpi_{| \Omega_{2i}}$};
\draw[->,dashed]  (3) -- (8);
\draw[->,dashed]  (4) -- (9);
\draw[->,dashed]  (4) -- (10);
\end{tikzpicture}
\end{center}
The first line of the diagram is acyclic because of the construction of the projective resolution of $X_s^+$. In order for the second line to be a complex, we need to check that $\varpi(\Omega_{2i}) \subseteq \text{Ker }\partial_1$. The map $\varpi$ can be illustrated as follows:
\begin{center}
 \begin{tikzpicture}[scale=1, transform shape]
\tikzset{>=stealth}
\node (1) at (0,1.5) []{$X_s^+$};
\node (2) at (-1,0) []{$X_s^-$};
\node (3) at (1,0) []{$X_s^-$};
\node (4) at (0,-1.5) []{$X_s^+$};
\draw[->]  (1) -- (2);
\draw[->]  (1) -- (3);
\draw[->]  (2) -- (4);
\draw[->]  (3) -- (4) ;

\node (5) at (4,1.5) []{$X_s^+$};
\node (6) at (3,0) []{$X_s^-$};
\node (7) at (5,0) []{$X_s^-$};
\node (8) at (4,-1.5) []{$X_s^+$};
\draw[->]  (5) -- (6);
\draw[->]  (5) -- (7);
\draw[->]  (6) -- (8);
\draw[->]  (7) -- (8) ;

\node (9) at (2,0) []{$\dots \dots$};

\draw[-]  (-1.5,0.25) -- (-0.5,0.25);
\draw[-]  (-0.5,0.25) -- (0.5,-0.25);
\draw[-]  (0.5,-0.25) -- (1.5,-0.25);

\draw[-]  (-1.5,0.25) -- (-0.5,-1.75) ;
\draw[-]  (-0.5,-1.75) --  (4.5,-1.75) ;
\draw[-]  (4.5,-1.75) -- (5.5,0.25);
\draw[-]   (5.5,0.25) -- (4.5,0.25);

\draw[-]  (4.5,0.25) -- (3.5,-0.25);
\draw[-]  (3.5,-0.25) -- (2.5,-0.25);

\node (10) at (2,-0.25) []{$\dots$};
\node (11) at (2,-1) []{$\Omega_{2i}$};

\node (12) at (8,1.5) []{$X_s^+$};
\node (13) at (7,0) []{$X_s^-$};
\node (14) at (9,0) []{$X_s^-$};
\node (15) at (8,-1.5) []{$X_s^+$};
\draw[->]  (12) -- (13);
\draw[->]  (12) -- (14);
\draw[->]  (13) -- (15);
\draw[->]  (14) -- (15) ;

\node (16) at (11,1.5) []{$X_s^+$};
\node (17) at (10,0) []{$X_s^-$};
\node (18) at (12,0) []{$X_s^-$};
\node (19) at (11,-1.5) []{$X_s^+$};
\draw[->]  (16) -- (17);
\draw[->]  (16) -- (18);
\draw[->]  (17) -- (19);
\draw[->]  (18) -- (19) ;

\draw[-]  (6.5,0.25) -- (7.5,0.25);
\draw[-]  (7.5,0.25) -- (8.5,-0.25);
\draw[-]  (8.5,-0.25) -- (10.5,-0.25);
\draw[-]  (10.5,-0.25) -- (11.5,0.25);
\draw[-]  (11.5,0.25) -- (12.5,0.25);
\draw[-]  (12.5,0.25) -- (11.5,-1.75);
\draw[-]  (11.5,-1.75) -- (7.5,-1.75);
\draw[-]  (7.5,-1.75) -- (6.5,0.25);
\node (20) at (9.5,-1) []{$\varpi(\Omega_{2i})$};

\draw[->]  (5.5,0) -- (6.5,0) node[midway, above] {$\varpi$} ;
\end{tikzpicture}
\end{center}
where $\varpi$ removes $(2i-1)$ factors $P_s^+$. It follows that $\varpi(\Omega_{2i}) \subseteq \text{Ker }\partial_1$ by construction of $\partial_1$, and the second line is a complex. Moreover it was shown in \cite{Nagatomo-Tsuchiya} that $P_s^+$ and $P_s^-$ are injective $\mathcal{W}(p)$-modules. So the homomorphism $\varpi$ can be pushed down the diagram and give the dashed homomorphisms, the last one being $\gamma : \bigoplus_{k=1}^{2i}P_s^- \longrightarrow X_s^-$, hence representing an element in $\on{Ext}_{\mathcal{W}(p)}^{2i-1}(X_s^+, X_s^-)$. Therefore we see that the homomorphism $f=\delta \circ \varpi$ represents the product $\delta\gamma$ in $\on{Ext}_{\mathcal{W}(p)}^{2i}(X_s^+, X_s^+)$, and it is true for all elements of the basis of $\on{Ext}_{\mathcal{W}(p)}^{2i}(X_s^+, X_s^+)$. Thus the homomorphism $\on{Ext}_{\mathcal{W}(p)}^{1}(X_s^-, X_s^+) \otimes \on{Ext}_{\mathcal{W}(p)}^{2i-1}(X_s^+, X_s^-) \longrightarrow \on{Ext}_{\mathcal{W}(p)}^{2i}(X_s^+, X_s^+)$ is surjective, and with the same type of reasoning, we show that the following homomorphisms are also surjective:
\begin{equation*}
\begin{cases}
\on{Ext}_{\mathcal{W}(p)}^{1}(X_s^+, X_s^-) \otimes \on{Ext}_{\mathcal{W}(p)}^{2i-1}(X_s^-, X_s^+) \longrightarrow \on{Ext}_{\mathcal{W}(p)}^{2i}(X_s^-, X_s^-), \\[5pt]
\on{Ext}_{\mathcal{W}(p)}^{1}(X_s^+, X_s^-) \otimes \on{Ext}_{\mathcal{W}(p)}^{2i}(X_s^+, X_s^+) \longrightarrow \on{Ext}_{\mathcal{W}(p)}^{2i+1}(X_s^+, X_s^-), \\[5pt]
\on{Ext}_{\mathcal{W}(p)}^{1}(X_s^-, X_s^+) \otimes \on{Ext}_{\mathcal{W}(p)}^{2i}(X_s^-, X_s^-) \longrightarrow \on{Ext}_{\mathcal{W}(p)}^{2i+1}(X_s^-, X_s^+).
\end{cases}       
\end{equation*}
Any element of $\on{Ext}_{\mathcal{W}(p)}^*(X_{\mathcal{W}(p)}, X_{\mathcal{W}(p)})$ of degree $n$ can then be expressed as a sum of products of elements of degree $1$ and elements of degree $n-1$. By repeating this process, we see that $\on{Ext}_{\mathcal{W}(p)}^*(X_{\mathcal{W}(p)}, X_{\mathcal{W}(p)})$ is generated in degree $0$ and $1$.
\end{proof}

\begin{proof}[Proof of Proposition~\ref{kQandExtforW(p)}]
We have seen that there is a well-defined homomorphism $kQ_{\mathcal{W}(p)}/\mathcal{I}_{\mathcal{W}(p)} \longrightarrow \on{Ext}_{\mathcal{W}(p)}^*(X_{\mathcal{W}(p)}, X_{\mathcal{W}(p)})$ that sends $(\alpha_{s,1},\alpha_{s,2})$ (resp. $(\beta_{s,1},\beta_{s,2})$) to $(\alpha^s_1, \alpha^s_2)$ (resp. $(\beta^s_1, \beta^s_2)$). Based on Theorem~\ref{W(p)_Yoneda}, the homomorphism is surjective, and it induces a surjection between homogeneous subspaces of the same degree. By proving that these subspaces have the same dimension, we will have an isomorphism.

We already know the dimensions of the homogeneous subspaces of $\on{Ext}_{\mathcal{W}(p)}^*(X_{\mathcal{W}(p)}, X_{\mathcal{W}(p)})$ from \cite{GSTF}. The algebra $\mathcal{A}=kQ_{\mathcal{W}(p)}/\mathcal{I}_{\mathcal{W}(p)}$ can be written as a direct sum:
\begin{align*}
\mathcal{A}=\bigoplus_{s=1}^{p-1}\left[\bigoplus_{i=1}^\infty\mathcal{A}_{\alpha_{s}}^i \oplus \bigoplus_{i=1}^\infty\mathcal{A}_{\beta_{s}}^i \oplus \cc e_{s}^+ \oplus \cc e_{s}^- \right] \oplus \cc e_{p}^+ \oplus \cc e_{p}^- ,
\end{align*}
where $\alpha_{s, k},\beta_{s, k}$ are of degree 1, $e_{j}^\pm$ are of degree 0 for $1 \leq j \leq p$, and $\mathcal{A}_{\alpha_{s}}^i$ (resp. $\mathcal{A}_{\beta_{s}}^i$)  is the subspace of $\mathcal{A}$ spanned by elements of degree $i$ and ending with $\alpha_{s, 1}$ or $\alpha_{s, 2}$ (resp. $\beta_{s, 1}$ or $\beta_{s, 2}$). 

From the relations defining $\mathcal{I}_{\mathcal{W}(p)}$, we see that from a monomial in $\alpha_{s, i},\beta_{s, j}$ we can replace $\alpha_{s, 2}\beta_{s, 1}$ (resp. $\beta_{s, 2}\alpha_{s, 1}$) by $-\alpha_{s, 1}\beta_{s, 2}$ (resp. by $-\beta_{s, 1}\alpha_{s, 2}$) so that the indices are non-decreasing. We construct all independent elements of $\mathcal{A}_{\alpha_s}$ successively by multiplying $\alpha_{s, 1}$ and $\alpha_{s, 2}$ on the left so that the indices do not decrease. We obtain the following trees:
\begin{center}
 \begin{tikzpicture}[scale=1, transform shape]
\tikzset{>=stealth}
\node (1) at (0,0) []{$\alpha_{s, 2}$};
\node (2) at (1,2) []{$\beta_{s, 1}\alpha_{s, 2}$};
\node (3) at (1,-2) []{$\beta_{s, 2}\alpha_{s, 2}$};
\node (4) at (3.5,2) []{$\alpha_{s, 1}\beta_{s, 1}\alpha_{s, 2}$};
\node (5) at (7,2) []{$\beta_{s, 1}\alpha_{s, 1}\beta_{s, 1}\alpha_{s, 2}$};
\node (6) at (11,2) []{$\alpha_{s,1}\beta_{s, 1}\alpha_{s, 1}\beta_{s, 1}\alpha_{s, 2}$};
\node (22) at (14,2) []{$\dots$};
\draw[->]  (1) -- (2);
\draw[->]  (1) -- (3);
\draw[->]  (2) -- (4);
\draw[->]  (4) -- (5);
\draw[->]  (5) -- (6);
\draw[->]  (6) -- (22);

\node (7) at (3.5,0) []{$\alpha_{s, 1}\beta_{s, 2}\alpha_{s, 2}$};
\node (8) at (3.5,-4) []{$\alpha_{s, 2}\beta_{s, 2}\alpha_{s, 2}$};
\node (9) at (7,0) []{$\beta_{s,1}\alpha_{s, 1}\beta_{s, 2}\alpha_{s, 2}$};
\node (10) at (11,0) []{$\alpha_{s,1}\beta_{s,1}\alpha_{s, 1}\beta_{s, 2}\alpha_{s, 2}$};
\node (11) at (14,0) []{$\dots$};
\draw[->]  (3) -- (7);
\draw[->]  (3) -- (8);
\draw[->]  (7) -- (9);
\draw[->]  (9) -- (10);
\draw[->]  (10) -- (11);

\node (12) at (7,-2) []{$\beta_{s,1}\alpha_{s, 2}\beta_{s, 2}\alpha_{s, 2}$};
\node (13) at (11,-2) []{$\alpha_{s,1}\beta_{s,1}\alpha_{s, 2}\beta_{s, 2}\alpha_{s, 2}$};
\node (14) at (14,-2) []{$\dots$};
\draw[->]  (8) -- (12);
\draw[->]  (12) -- (13);
\draw[->]  (13) -- (14);

\node (15) at (7,-5) []{$\beta_{s,2}\alpha_{s, 2}\beta_{s, 2}\alpha_{s, 2}$};
\node (16) at (11,-3.5) []{$\alpha_{s,1}\beta_{s,2}\alpha_{s, 2}\beta_{s, 2}\alpha_{s, 2}$};
\node (17) at (14,-3.5) []{$\dots$};
\node (18) at (14,-3.5) []{$\dots$};
\draw[->]  (8) -- (15);
\draw[->]  (15) -- (16);
\draw[->]  (16) -- (17);
\draw[->]  (16) -- (17);

\node (19) at (11,-7) []{$\alpha_{s,2}\beta_{s,2}\alpha_{s, 2}\beta_{s, 2}\alpha_{s, 2}$};
\node (20) at (14,-6) []{$\dots$};
\node (21) at (14,-8) []{$\dots$};
\draw[->]  (15) -- (19);
\draw[->]  (19) -- (20);
\draw[->]  (19) -- (21);

\node (23) at (-0.2,-10) []{$\alpha_{s, 1}$};
\node (24) at (1.4,-10) []{$\beta_{s, 1}\alpha_{s, 1}$};
\node (25) at (3.7,-10) []{$\alpha_{s, 1}\beta_{s, 1}\alpha_{s, 1}$};
\node (26) at (7,-10) []{$\beta_{s, 1}\alpha_{s, 1}\beta_{s, 1}\alpha_{s, 1}$};
\node (27) at (11,-10) []{$\alpha_{s,1}\beta_{s, 1}\alpha_{s, 1}\beta_{s, 1}\alpha_{s, 1}$};
\node (28) at (14,-10) []{$\dots$};
\draw[->]  (23) -- (24);
\draw[->]  (24) -- (25);
\draw[->]  (25) -- (26);
\draw[->]  (26) -- (27);
\draw[->]  (27) -- (28);

\end{tikzpicture}
\end{center}
By counting the number of terms on the $i$-th column, we get that $\on{dim} \mathcal{A}_{\alpha_s}^i=i+1$, and by the same reasoning $\on{dim} \mathcal{A}_{\beta_s}^i=i+1$. Furthermore, we know of the following surjections:
\begin{align*}
\renewcommand\arraystretch{1.5}
\left\{\begin{array}{lcc}
\mathcal{A}_{\alpha_s}^{2i} & \longrightarrow & \on{Ext}_{\mathcal{W}(p)}^{2i}(X_s^-, X_s^-), \\
\mathcal{A}_{\alpha_s}^{2i+1} & \longrightarrow & \on{Ext}_{\mathcal{W}(p)}^{2i+1}(X_s^-, X_s^+), \\
\mathcal{A}_{\beta_s}^{2i} & \longrightarrow & \on{Ext}_{\mathcal{W}(p)}^{2i}(X_s^+, X_s^+), \\
\mathcal{A}_{\beta_s}^{2i+1} & \longrightarrow & \on{Ext}_{\mathcal{W}(p)}^{2i+1}(X_s^+, X_s^-).
\end{array}\right.
\end{align*}
Using the dimension of the summands of $\on{Ext}^*_{\mathcal{W}(p)}(X_{\mathcal{W}(p)}, X_{\mathcal{W}(p)})$ and the above computations, each of these maps is an isomorphism. Finally, for each $1 \leq s \leq p$, the element $e_s^{\pm}$ is sent to $e^\pm_s \in \on{Ext}_{\mathcal{W}(p)}^0(X_s^\pm, X_s^\pm)$. It follows that the homomorphism $kQ_{\mathcal{W}(p)}/\mathcal{I}_{\mathcal{W}(p)} \longrightarrow \on{Ext}_{\mathcal{W}(p)}^*(X_{\mathcal{W}(p)}, X_{\mathcal{W}(p)})$ is an isomorphism on homogeneous components and preserves the product, so it is an isomorphism of algebras. As $\mathcal{I}_{\mathcal{W}(p)}$ is generated by homogeneous quadratic relations in $kQ_{\mathcal{W}(p)}$, the Ext algebra is indeed generated in degree $0$ and $1$ with relations in degree $2$.
\end{proof}}


\begin{proof}[Proof of Proposition~\ref{ExtRtilde}]
We write $\mathcal{P}(m)$ for the first statement of the proposition for a fixed $m \in \mathbb{N}^*$. We have just seen that $\mathcal{P}(1),\mathcal{P}(2)$ and $\mathcal{P}(3)$ are true. We will proceed by induction.

$\bullet$ \underline{odd case:} Assume $\mathcal{P}(m)$ is true up to some $m$ odd. The basis of $Q_{m+1}$ is composed of the following elements: $s_1^{b_1} \dots s_8^{b_8}$, $t_{i_1} t_{i_2}s_1^{b_1} \dots s_8^{b_8}$, $t_{i_1} t_{i_2} t_{i_3} t_{i_4}s_1^{b_1} \dots s_8^{b_8}$, $t_{i_1} \dots t_{i_6}s_1^{b_1} \dots s_8^{b_8}$, $t_1 \dots t_8 s_1^{b_1} \dots s_8^{b_8}$, where for each element the sum of the degrees is $m+1$. 

For $t_{i_1} t_{i_2}s_1^{b_1} \dots s_8^{b_8} \in Q_{m+1}$, we define $X=t_{i_1}s_1^{b_1} \dots s_8^{b_8} \in Q_{m}$. By the induction hypothesis, it corresponds to $f_X=a\alpha_{i_1}(\alpha_1^2)^{b_1} \dots (-\alpha_8^2)^{b_8}$, $a \in \cc^*$. We now lift this homomorphism:
 \begin{center}
  \begin{tikzpicture}[scale=0.9,  transform shape]
  \tikzset{>=stealth}
  
\node (1) at ( 0,0){$Q_{m+1}$};
\node (2) at ( 3,0){$Q_m$};
\node (3) at ( 0,-2){$Q_1$};
\node (4) at ( 3,-2) {$\widetilde{R}$};
\node (5) at ( 5,-2){$\cc$};
\node (6) at ( 7,-2){0};

\node (7) at ( -2,0){$\dots$};
\node (8) at ( -2,-2){$\dots$};

\draw [decoration={markings,mark=at position 1 with
    {\arrow[scale=1.2,>=stealth]{>}}},postaction={decorate}] (1) --  (2) node[midway, above] {$\partial_{m+1}$};
\draw [decoration={markings,mark=at position 1 with
    {\arrow[scale=1.2,>=stealth]{>}}},postaction={decorate}] (1)  --  (3) node[midway, left] {$g_2$};

\draw [decoration={markings,mark=at position 1 with
    {\arrow[scale=1.2,>=stealth]{>}}},postaction={decorate}] (2)  --  (4) node[midway, left] {$g_1$};
\draw [decoration={markings,mark=at position 1 with
    {\arrow[scale=1.2,>=stealth]{>}}},postaction={decorate}] (3)  --  (4) node[midway, above] {$\partial_1$};

\draw [decoration={markings,mark=at position 1 with
    {\arrow[scale=1.2,>=stealth]{>}}},postaction={decorate}] (2)  --  (5) node[midway, above right] {$f_X$};
\draw [decoration={markings,mark=at position 1 with
    {\arrow[scale=1.2,>=stealth]{>}}},postaction={decorate}] (4)  --  (5) node[midway, above] {$\epsilon$};
    
\draw [decoration={markings,mark=at position 1 with
    {\arrow[scale=1.2,>=stealth]{>}}},postaction={decorate}] (5)  --  (6);
    
\draw [decoration={markings,mark=at position 1 with
    {\arrow[scale=1.2,>=stealth]{>}}},postaction={decorate}] (7)  --  (1) node[midway, above] {$\partial_{m+2}$};
\draw [decoration={markings,mark=at position 1 with
    {\arrow[scale=1.2,>=stealth]{>}}},postaction={decorate}] (8)  --  (3) node[midway, above] {$\partial_2$};
\end{tikzpicture}
 \end{center}
By defining $g_1:  Q_m  \longrightarrow  \cc$ such that it sends $X$ to $1$ and everything else is sent to zero, we can check that for all $i_2 \neq i_1$ we have $g_1 \partial_{m+1}(t_{i_1} t_{i_2}s_1^{b_1} \dots s_8^{b_8})=\partial_1(-t_{i_2})$. Furthermore, by looking at the different possible values for $i_1$, we can compute $g_1 \partial_{m+1}(s_1^{b_1} \dots s_j^{b_j+1} \dots s_8^{b_8})$ for all $j$. By a case by case analysis depending on $i_1$, we determine $g_2$ and see that:
\begin{itemize}
\item $t_{i_1}t_{i_2}s_1^{b_1} \dots s_8^{b_8}$ corresponds to $a\alpha_{i_1}\alpha_{i_2}(\alpha_1^2)^{b_1} \dots (-\alpha_8^2)^{b_8}$,
\item $s_1^{b_1+1} \dots s_8^{b_8}$  corresponds to $\frac{a}{b_1+1}(\alpha_1^2)^{b_1+1} \dots (-\alpha_8^2)^{b_8}$, 
\item $s_1^{b_1}s_2^{b_2+1} \dots s_8^{b_8}$  corresponds to $\frac{a}{b_2+1}(\alpha_1^2)^{b_1}(\alpha_2^2)^{b_2+1} \dots (-\alpha_8^2)^{b_8}$, 
\item $s_1^{b_1} \dots s_{i_1}^{b_{i_1}+1} \dots s_8^{b_8}$ corresponds to $\frac{a}{b_{i_1}+1}(\alpha_1^2)^{b_1} \dots(-\alpha_{i_1}^2)^{b_{i_1}+1} \dots (-\alpha_8^2)^{b_8}, \ i_1=5,6,7,8$.
\end{itemize}
Thus $\mathcal{P}(m+1)$ is true for the elements above. If one starts this procedure with $t_{i_1} t_{i_2} t_{i_3} t_{i_4}s_1^{b_1} \dots s_8^{b_8}$, $t_{i_1} \dots t_{i_6}s_1^{b_1} \dots s_8^{b_8}$ or $t_1 \dots t_8 s_1^{b_1} \dots s_8^{b_8}$, then one obtains that $\mathcal{P}(m+1)$ is true for these elements. It remains to check that $\mathcal{P}(m+1)$ is true for $s_1^{b_1} \dots s_8^{b_8} \in Q_{m+1}$. As one of the $b_i$ is at least $1$, we have $s_1^{b_1} \dots s_i^{b_i-1} \dots s_8^{b_8} \in Q_{m-1}$ which,  according to the induction hypothesis, corresponds to $a(\alpha_1^2)^{b_1} \dots \beta_i^{b_i-1} \dots (-\alpha_8^2)^{b_8}$ where $\beta_i$ is the dual of $s_i$. By lifting this homomorphism twice, we can verify that $\mathcal{P}(m+1)$ is true for $s_1^{b_1} \dots s_8^{b_8}$. Thus $\mathcal{P}(m+1)$ is true for every element of the basis of $Q_{m+1}$.

$\bullet$ \underline{even case:} Assume $\mathcal{P}(m)$ is true up to some $m$ even. The basis of $Q_{m+1}$ is composed of the following elements: $t_{i_1} s_1^{b_1} \dots s_8^{b_8}$, $t_{i_1} t_{i_2} t_{i_3}s_1^{b_1} \dots s_8^{b_8}$, $t_{i_1} \dots t_{i_5}s_1^{b_1} \dots s_8^{b_8}$, $t_{i_1} \dots t_{i_7}s_1^{b_1} \dots s_8^{b_8}$, where for each element the sum of the degrees is $m+1$. By doing the same procedure as above, we can verify that $\mathcal{P}(m+1)$ is true for every element of the basis of $Q_{m+1}$. Hence $\mathcal{P}(m+1)$ is true for $m$ even.

So we proved that the property $\mathcal{P}(m)$ is hereditary. As it is true for $m=1,2,3$, it is true for all $m \in \mathbb{N}^*$. This settles the first statement. Using this first statement and the commutation relations between the $\alpha_i$'s and $\beta_j$'s we found earlier, we obtain a presentation of $H^*(\widetilde{R})$.
\end{proof}

\begin{proof}[Proof of Proposition~\ref{MonomialContainingTheDual}]
 Let us call $\mathcal{P}(m)$ the first statement of the proposition for a fixed $m \in \mathbb{N}^*$. We already know that $\mathcal{P}(1)$ and $\mathcal{P}(2)$ are true.
 
$\bullet$ \underline{even case:} Assume $\mathcal{P}(m)$ is true up to some $m$ even. Take $Y=T_1^{a_1}\dots T_4^{a_4}S_1^{b_1}\dots S_8^{b_8} \in P_{m+1}$. As $m$ is even, the possibilities are $Y=T_iS_1^{b_1}\dots S_8^{b_8}$ or $Y=T_{i_1}T_{i_2} T_{i_3}S_1^{b_1}\dots S_8^{b_8}$. 

Set $Y=T_i S_1^{b_1}\dots S_8^{b_8} \in P_{m+1}$. We have $S_1^{b_1}\dots S_8^{b_8} \in P_{m}$, and by the induction hypothesis it corresponds to $a(\delta_1^{b_1} \dots \delta_8^{b_8}+z)$ for some $z \in \text{Ker}(\pi^{\#})$, $a \in \cc^*$. We now lift this homomorphism. Define $g_1 \in \text{Hom}_{R}(Q_m,R)$ such that $g_1(S_1^{b_1}\dots S_8^{b_8})=1$, $g_1(\text{other})=0$. Then $g_1 \partial_{m+1}(T_j S_1^{b_1}\dots S_8^{b_8})=\partial_1(T_j)$. There might be other $X \in P_{m+1}$ such that $S_1^{b_1}\dots S_8^{b_8} \in \partial_{m+1}(X)$, but they would contain variables different from $T_i$ ($1 \leq i \leq 4$), $S_j$ ($1 \leq j \leq 8$, $j \neq 3,4$). We can thus define:
\begin{align*}
\begin{array}[t]{cccl}
g_2: & P_{m+1} & \longrightarrow & \cc \\
           & T_j S_1^{b_1}\dots S_8^{b_8}  & \longmapsto & T_j   \ \forall \ 1 \leq j \leq 4, \\
           & X & \longmapsto & g_2(X), \\
           & \text{other} & \longmapsto & 0. \\
\end{array} 
\end{align*}
It follows that $a\gamma_j(\delta_1^{b_1} \dots \delta_8^{b_8}+z)=(T_j S_1^{b_1}\dots S_8^{b_8})^*+\sum_X X^*$, where the sum is taken over a subset (or the whole of) the $X$'s above. Therefore each monomial in this sum contains a variable different from $T_i, S_j$.

The kernel of $\pi^{\#}$ is an ideal of $H^{*}(R)$ so $\gamma_j z \in \text{Ker}(\pi^{\#})$. Because each $X$ in $\sum_X X^*$ is composed of monomials with a variable different from $T_i, S_j$, and the fact that $\on{Im}(\psi)$ is generated by the $T_i, S_j$ (Proposition~\ref{LiftPsi}), we know that $\sum_X X^* \in \text{Ker}(\pi^{\#})$.

By taking $j=i$ we obtain that:
\begin{align*}
(T_i S_1^{b_1}\dots S_8^{b_8})^*= a(\gamma_i \delta_1^{b_1}\dots \delta_8^{b_8}+z')
\end{align*}
where $z'=\gamma_i z-\frac{1}{a}\sum_X X^* \in \text{Ker}(\pi^{\#})$. 

We repeat this process for the elements $Y=T_{i_1}T_{i_2} T_{i_3}S_1^{b_1}\dots S_8^{b_8}$ and see that $\mathcal{P}(m+1)$ is true.

$\bullet$ \underline{odd case:} Assume $\mathcal{P}(m)$ is true up to some $m$ odd. A monomial $Y \in P_{m+1}$ is of the form $ S_1^{b_1}\dots S_8^{b_8}$, $T_{i_1}T_{i_2}S_1^{b_1}\dots S_8^{b_8}$, or $T_1 T_2 T_3 T_4 S_1^{b_1}\dots S_8^{b_8}$. For the last two expressions, we can apply the same procedure as above and we see that $\mathcal{P}(m+1)$ is true for these elements. 

Set $Y=S_1^{b_1}\dots S_8^{b_8} \in P_{m+1}$. There exists $1 \leq j \leq 8$ such that $b_j \geq 1$. Using the assumption the $\mathcal{P}(m)$ is true up to $m$, we know that there exist $a \in \cc^*$ and $z \in \text{Ker}(\pi^{\#})$ such that $a(\delta_1^{b_1}\dots \delta_j^{b_j-1} \dots \delta_8^{b_8}+z)$ corresponds to $S_1^{b_1}\dots S_j^{b_j-1} \dots  S_8^{b_8} \in P_{m-1}$. By lifting the homomorphism and computing $g_1$, $g_2$ and $g_3$, we find that:
\begin{align*}
\setlength{\arraycolsep}{0.3cm}
\resizebox{\hsize}{!}{ $\begin{array}{cc}
\setlength{\arraycolsep}{6pt}\begin{array}[t]{cccl}
g_1: & P_{m-1} & \longrightarrow & R\\
           & S_1^{b_1}\dots S_j^{b_j-1} \dots  S_8^{b_8} & \longmapsto & 1,\\
           & \text{other} & \longmapsto & 0 \\
\end{array}  &  \setlength{\arraycolsep}{6pt}\begin{array}[t]{cccl}
g_2: & P_{m} & \longrightarrow & P_1 \\
           & T_iS_1^{b_1}\dots S_j^{b_j-1} \dots  S_8^{b_8} & \longmapsto & T_i \ \forall 1 \leq i \leq 4,  \\
           & X & \longmapsto & g_2(X), \\
           & \text{other} & \longmapsto & 0 \\
\end{array}
\end{array}$}
\end{align*}
\begin{align*}
\begin{array}[t]{cccl}
g_3: & P_{m+1} & \longrightarrow & P_2 \\
           & T_iT_kS_1^{b_1}\dots S_j^{b_j-1} \dots  S_8^{b_8} & \longmapsto & T_iT_k \ \forall 1 \leq i < k \leq 4,  \\
           & S_1^{b_1} \dots S_i^{b_i+1}\dots S_j^{b_j-1} \dots  S_8^{b_8} & \longmapsto & (b_i+1)S_i \ \forall  i \neq j, \\
           &S_1^{b_1} \dots S_j^{b_j} \dots  S_8^{b_8} & \longmapsto & b_j S_j, \\
           & Y & \longmapsto & g_3(Y), \\
           & W & \longmapsto & g_3(W), \\
           & \text{other} & \longmapsto & 0, \\
\end{array} 
\end{align*}
where the $X$'s are monomials of $P_m$ such that $\partial_m(X)$ contains $S_1^{b_1}\dots S_j^{b_j-1} \dots  S_8^{b_8}$, where the $Y$'s and $W$'s are monomials of $P_{m+1}$ such that $\partial_{m+1}(Y)$ contains $T_iS_1^{b_1}\dots S_j^{b_j-1} \dots  S_8^{b_8}$ for some $1 \leq i \leq 4$, and such that $\partial_{m+1}(W)$ contains one of the $X$'s in the definition of $g_2$. We see that in all cases, $X$, $Y$ and $W$ contain at least one variable different from $T_i, S_j$. We have thus found:
\begin{align*}
\frac{a}{b_j}\delta_j(\delta_1^{b_1}\dots \delta_j^{b_j-1} \dots \delta_8^{b_8}+z)=(S_1^{b_1}\dots S_j^{b_j} \dots S_8^{b_8})^*+\sum_Y Y^* +\sum_W W^*,
\end{align*}
where the sums are taken over subsets (or the whole of) the $Y$'s and $W$'s above. Therefore each monomial in these sums contains a variable different from $T_i, S_j$. 

We know that the $\delta_i$'s commute modulo $\text{Ker}(\pi^{\#})$, so there exists $z' \in \text{Ker}(\pi^{\#})$ such that $\delta_j\delta_1^{b_1}\dots \delta_j^{b_j-1} \dots \delta_8^{b_8} =\delta_1^{b_1}\dots \delta_j^{b_j} \dots \delta_k^{b_k} +z'$. Furthermore, because each $Y$ and $W$ in $\sum_Y Y^* +\sum_W W^*$ is composed of monomials with at least one variable different from $T_i, S_j$, we know that $\sum_Y Y^* +\sum_W W^* \in \text{Ker}(\pi^{\#})$. As a consequence:
\begin{align*}
(S_1^{b_1}\dots S_j^{b_j} \dots S_8^{b_8})^*= \frac{a}{b_j}(\delta_1^{b_1}\dots \delta_j^{b_j} \dots \delta_8^{b_8}+z'')
\end{align*}
where $z''=z'+\delta_j z-\frac{b_j}{a}(\sum_Y Y^*+\sum_W W^*) \in \text{Ker}(\pi^{\#})$, so $\mathcal{P}(m+1)$ is true.

We have found that $\mathcal{P}$ is hereditary and is true for $m=1,2$. Therefore $\mathcal{P}(m)$ is true for all $m \in \mathbb{N}^*$. 
\end{proof}

\section*{Acknowledgements}
The authors would like thank Antun Milas for providing detailed information and comments on some references. They also express their gratitude to the reviewer for bringing to their attention some literature on the restricted quantum group as well as providing detailed comments helping to improve the paper.
 
 A. Caradot was supported by Shanghai Jiao Tong University as a postdoctoral researcher.

\setlength\bibitemsep{7pt}

\end{document}